\documentclass[a4paper,11pt]{amsart}
\usepackage{yfonts,amsmath,amsfonts,amssymb,amsthm}
\usepackage{mathtools}
\usepackage{graphics}
\usepackage{epsfig}
\usepackage{esint}
\usepackage[hyperindex,breaklinks]{hyperref} 
\usepackage[utf8]{inputenc}

\usepackage[normalem]{ulem} 

\newcommand{\ignore}[1]{}

\newtheorem{definition}{Definition}

\newtheorem{theorem}{Theorem}
\newtheorem*{theorem*}{Theorem}
\newtheorem{remark}{Remark}
\newtheorem{lemma}{Lemma}

\newtheorem{assumption}{Assumption}

\providecommand{\PfStart}[1]{\newcounter{#1}\refstepcounter{#1}} 
\providecommand{\PfStep}[2]{ \ifnum\value{#1}=1
\else\medskip\fi{\sc Step }\arabic{#1}\label{#2}\refstepcounter{#1}.}

\usepackage{color}
\definecolor{darkred}{rgb}{0.9,0.1,0.1}
\definecolor{darkblue}{rgb}{0,0,0.7}
\definecolor{darkgreen}{rgb}{0,0.5,0}
\definecolor{violet}{rgb}{0.78,0.26,0.95}
\definecolor{gray}{rgb}{0.4,0.4,0.4}

\ignore{
\usepackage{draftwatermark}
\SetWatermarkAngle{45}
\SetWatermarkScale{4}
\SetWatermarkLightness{0.85}
\SetWatermarkFontSize{42pt}
\SetWatermarkText{\bfseries\sffamily Don't distribute}
}

\newenvironment{keyword}
{
\begin{center}
\begin{minipage}{.9\textwidth}\small\textbf{Keywords}:\noindent
}
{
\end{minipage}
\end{center}
}

\newenvironment{msc}
{
\begin{center}
\begin{minipage}{.9\textwidth}\small\textbf{MSC 2020}:\noindent
}
{
\end{minipage}
\end{center}
}

\title[Multi-index Based Solutions to $\Phi^4$]
{Multi-index Based Solution Theory to the $\Phi^4$ Equation in the Full Subcritical Regime}

\author[L.~Broux]{Lucas Broux}
\author[F.~Otto]{Felix Otto}
\author[R.~Steele]{Rhys Steele}

\begin{document}

\begin{abstract}
We obtain (small-parameter) well-posedness for the (space-time periodic) $\Phi^4$ equation in the full subcritical regime in the context of regularity structures based on multi-indices.
As opposed to Hairer's more extrinsic tree-based setting, 
due to the intrinsic description encoded by multi-indices, it is not possible to obtain a solution theory via 
the standard
fixed-point argument.
Instead, we develop a more 
intrinsic approach for existence using
a variant of the continuity method from classical PDE theory 
based on 
a priori estimates for a new `robust' formulation of the equation.
This formulation also allows us to obtain uniqueness of solutions and continuity of the solution map in the model norm even at the limit of vanishing regularisation scale.
Since our proof relies on the structure of the nonlinearity in only a mild way, 
we expect the same ideas to be sufficient to treat a more general class of equations.

\end{abstract}

\maketitle

\begin{msc} 
60H17, 
60L30. 
\end{msc}

\tableofcontents

\section{Introduction} 
In this work, we are interested in the 
equation 
\begin{align}\label{t0}
    Lu = P(\lambda u^3 + \xi)
\end{align}
with space-time periodic boundary conditions, 
where $L = \partial_0 - \sum_{i=1}^d \partial_i^2$
is the heat operator, 
$\lambda \in \mathbb{R}$ is a given parameter 
and $P$ denotes the projection onto 
distributions of vanishing space-time average.
In the majority of the work, we view $\xi$ as a fixed realisation of a singular noise
of parabolic H\"older regularity $\alpha - 2$, with 
	\begin{align} \label{t53}
		\alpha \in (-1, 0) ,
	\end{align}
which corresponds to the singular but subcritical regime (in the sense of Hairer \cite{Hai14}) for\footnote{in practice, a condition of integrable variance (also called Hurst condition) also has to be added for the stochastic estimates to hold, and here takes the form $3\alpha+(d+2)/2>0$} \eqref{t0}.
This equation is reminiscent of the parabolic quantisation equation 
for the $\Phi^4$ model in (Euclidean) Quantum Field Theory, where $\lambda$ is sometimes called the coupling constant. The main difference to the parabolic quantisation equation is that we consider here space-time periodic boundary conditions. This is done mainly for technical simplification to separate out challenges involving renormalisation from those involving incorporating initial data.

\medskip
As is noted many times in the literature
in the case of the parabolic quantisation equation,
\eqref{t0} is not well-posed in the traditional PDE sense.
Indeed, one expects $u$ to behave on small scales in the same way as the solution of the linear problem
i.e.\ with $\lambda = 0$.
In turn, $u^3$ is not expected to have better regularity than $u$, and 
by Schauder theory, this implies that 
the parabolic H\"older regularity of $u$ cannot be expected to be better than $\alpha$.
This implies that the cubic part of the equation is not canonically defined on a large enough space 
for classical solution theories to be applicable.

\medskip
By now, it is widely understood that singular SPDEs of this type require a renormalisation, which formally amounts to inserting an infinite counterterm required to cancel the small-scale divergences. That is, in some sense, one considers instead the PDE
\begin{align}\label{t-1}
    Lu = P( \lambda u^3  + \text{`$\infty u$'} + \xi).
\end{align}
Of course, it is necessary to give a suitable meaning to such a formal expression. A common approach in the literature is to consider a 
sequence of smooth functions
$\xi_\rho$ that converges to $\xi$ as $\rho \downarrow 0$ and exhibit a corresponding sequence of counterterms $h_\rho$ such that the solutions to the equations 
\begin{align}\label{t1}
    Lu_\rho = P(\lambda u_\rho^3 + h_{\lambda \, \rho} u + \xi_\rho)
\end{align}
converge in an appropriate sense as $\rho \downarrow 0$ to a limit which is independent of the choice of smoothing mechanism (though the required family of $h_{\lambda \, \rho}$ will typically depend on this choice). 

\medskip
In the past decade, a number of approaches to obtain a solution theory for such singular SPDEs have been developed including an approach based on paracontrolled calculus \cite{GIP15, CC18} and approaches based on the Polchinski flow \cite{Kup16, Duc22}. In the most developed framework for this type of problem, namely Hairer's theory of regularity structures \cite{Hai14}, this statement is proved by reference to an alternative characterisation given by a fixed-point problem for an abstract analogue of the mild form of the PDE posed at the level of objects known as modelled distributions. This reformulation makes sense even when $\rho = 0$ so long as suitable renormalised data known as a model is provided as an input. Constraints on the choice of model then enforce
a connection to the PDE \eqref{t1} with its choice of counterterm in the case where $\rho > 0$.

\medskip
More recently, an alternative and more geometric viewpoint on regularity structures has arisen in works of the second author with a number of collaborators (see e.g. \cite{OSSW, LOTT, BOT}). In this viewpoint, one attempts to informally parameterise the nonlinear and infinite-dimensional `solution manifold' $\mathcal{M}$ of the equation \eqref{t0}. In particular, symmetries of the manifold $\mathcal{M}$ become a central guiding consideration.
That is,
one should aim to find counterterms $h_{\lambda \, \rho}$ such that
the solution manifold $\mathcal{M}_\rho$ for
the regularised and renormalised equation \eqref{t1} 
inherits the symmetries of $\mathcal{M}$ and converges as $\rho \to 0$ to $\mathcal{M}$.

\medskip
As is argued in more detail in \cite{BOT}, one informative symmetry is symmetry under rescaling.
However, this rescaling does not leave the value of $\lambda$ invariant and thus motivates that one should consider all possible $\lambda \in \mathbb{R}$ at once.
In particular, in seeking a parameterisation, it becomes natural to consider the formal solution manifold as being at least partly parameterised by the coupling constant $\lambda$. By considering the case $\lambda = 0$ we see that this parameterisation does not completely describe the solution manifold since in that setting, it is clear that $\mathcal{M}$ is an affine space over the kernel of the operator $L$. Since 
\begin{align} \label{t67}
    \operatorname{ker} L \subset \{\text{space-time analytic functions $p$}\} ,
\end{align}
it becomes convenient to enrich our parameterisation to be in the couple $(\lambda, p)$ .
This leads to the ansatz that the solution may be described as a
formal power series in these variables\footnote{of course, one does not expect this series to actually converge}
of the form
	\begin{align} \label{t56}
		u ( \cdot ) = \sum_{\beta} \Pi_{x \beta} ( \cdot) \, \lambda^{\beta ( \mathfrak{3} )} \prod_{\mathbf{n} \in \mathbb{N}_0^{1+ d}} \bigg( \frac{1}{\mathbf{n} !}\partial^{\mathbf{n}} p ( x ) \bigg)^{\beta ( \mathbf{n} )} ,
	\end{align}
where the sum runs over the multi-indices $\beta$ (i.e.\ the compactly supported $\beta \colon \lbrace \mathfrak{3} \rbrace \sqcup \mathbb{N}_0^{1 + d} \to \mathbb{N}_0$) 
and where we have denoted by $\mathbf{n}! = \prod_{i = 0}^d n_i !$ for $\mathbf{n} = ( n_0, \cdots, n_{d} ) \in \mathbb{N}_0^{1 + d}$.
As is argued in more detail in a very similar setting in \cite{BOT} (see also \cite{LOTT} for the original quasilinear variant of the problem), this ansatz leads to a hierarchy of
linear PDEs
for the coefficients $\Pi_{x \beta}$ of the power series around a given basepoint $x$ indexed by multi-indices $\beta$.
These coefficients, along with their change of basepoint maps, fulfill the axioms required of a model as given in \cite[Definition 2.17]{Hai14}.

\medskip
Nonetheless, this formulation is not sufficient to adopt the fixed-point formulation as given in \cite{Hai14}. The reason is that the model constructed in this way provides a leaner set of renormalised data, capable of describing only the solution manifold itself; rather than a larger ambient linear space in which Hairer's fixed point argument would be formulated. As a result, until this paper, works in this multi-index-based setting (as opposed to Hairer's tree-based setting) have so far constructed the model in different contexts \cite{LOTT, BOT, GvalaniTempelmayr} and have provided a priori estimates for a quasilinear problem \cite{OSSW}.
These a priori estimates are local and invest smallness in the sense of the supremum norm of the solution and as such they do not yield an existence result even via compactness arguments.
In this work, we remedy this situation
and provide a full small-parameter
(i.e.\ $|\lambda| \ll 1$)
well-posedness theory for the suitably interpreted analogue of \eqref{t-1}.
We view our result as
a proof of concept for more general semilinear problems.

\medskip
To achieve this, we build on the overall strategy for a priori estimates laid out in \cite{OSSW}, now incorporating boundary data to obtain global a priori estimates under an assumption of smallness of the coupling constant.
A novelty of our approach is that 
we perform these estimates at the level of a new notion of \emph{robust formulation} (see Definition~\ref{rf03}) 
which makes sense even when $\rho = 0$.
This robust formulation replaces the nonlinear PDE \eqref{t1} with a linear PDE which is coupled to a nonlinear algebraic constraint for corresponding modelled distributions. 
Perhaps surprisingly, this point of view even allows us to obtain uniqueness, by adapting the a priori estimates to the difference of two solutions of the robust formulation. 
Our notion of robust formulation reveals that for this task one should focus on the algebraic constraints since the PDE ingredients are essentially linear.
Furthermore, we provide an existence result based on a
continuity method
in the spirit of classical PDE theory, 
taking as input the above a priori estimates alongside novel a priori estimates for the linearisation of the equation, coming with its own notion of robust formulation.

\medskip
Our approach is reminiscent of Davie's approach to rough paths \cite{Dav08}, see also the recent monograph \cite{CGZ} for a modern introduction.
Indeed, both approaches avoid using a fixed-point argument.
Rather, 
they appeal to a more intrinsic point of view, 
where a priori estimates for solutions,
as opposed to estimates on operations on linear function spaces,
play a crucial role in the context of establishing a solution theory. We mention here also that as this paper was in the final stages of preparation, the preprint \cite{BBH25} obtained solutions to rough differential equations based on multi-indices via the `log-ODE' method. Since this method is specifically adapted to the ODE setting, it is unclear if that approach would yield any results in the setting of singular SPDE.

\medskip
Our main new result, Theorem~\ref{main_result}, 
requires some terminology, in particular the notions of models and modelled distributions,
and is stated in Subsection~\ref{ss:main_results} below.
This result will fill the role of the analytic part of the solution theory, assuming the probabilistic construction of a model as input. 
In combination with a suitable periodic analogue of\footnote{we do not perform this adaptation here, but see Appendix~\ref{s:model} for a brief sketch of the changes needed} \cite{BOT}, the main result of this paper would yield the following consequence. 
\begin{theorem*}
Let (the law of) $\xi$ satisfy a periodic version of the spectral gap inequality stated in \cite{BOT}, 
see \eqref{t100} below.
Let $\psi$ be a Schwartz function with $\int \psi = 1$, $\psi_{\rho}$ be the corresponding family of (parabolic) mollifiers, 
and $\xi_{\rho} \coloneqq \xi * \psi_{\rho}$ the corresponding mollified version of the noise.
Then there exists a random variable $\lambda_0>0$ depending only on $\alpha, d$ (along with the torus size and on $\xi$),
such that for all $| \lambda | \leq \lambda_0$ there is a deterministic sequence $h_{\lambda \, \rho}$ of counterterms and a sequence $u_{\rho}$ of solutions to \eqref{t1}, 
such that $u_{\rho}$ converges in probability in the parabolic $\alpha$-H\"older space to a non-trivial distribution $u$.
\end{theorem*}

\medskip

\subsection*{Article Structure}
In Section~\ref{s:ass_res}, we describe, first heuristically then rigorously, our notion of model, whilst also introducing our main notations.
We then state our main result, Theorem~\ref{main_result}.

\medskip

In Section~\ref{s:strat}, we provide a detailed description of our proof strategy.
Theorem~\ref{main_result} is in fact obtained as a consequence of three separate results:
an a priori estimate (Theorem~\ref{a priori})
a statement of continuity in the model topology (Theorem~\ref{uniqueness}) 
and a statement of existence (Theorem~\ref{existence}).
The remaining sections are devoted to the proofs.
Namely, Section~\ref{s:proof_a_priori} proves Theorem~\ref{a priori},
Section~\ref{Con} proves Theorem~\ref{uniqueness}
and Section~\ref{ss:exist} proves Theorem~\ref{existence}.

\section{Assumptions and Main Results} \label{s:ass_res}

\subsection{On the (Deterministic) Model}

\subsubsection{Setting and Notations}
Before stating rigorous definitions in Subsection~\ref{ss:defs} below, 
here we briefly introduce some language and notations which we will use throughout this paper.
We refer the interested reader to \cite{BOT} for a more exhaustive discussion, 
see also \cite{LOT, LOTT, OSSW, GvalaniTempelmayr} for an analogous setting in the case of a quasilinear equation.

\medskip
\textbf{Algebraisation}.
Since the ansatz \eqref{t56} takes the form of a formal power series expansion, it is conveniently algebraised by embedding it into the algebra of formal power series in the abstract variables\footnote{we follow the notation of \cite{BOT} in our use of the dummy letter $\mathfrak{3}$, which comes from the cubic nonlinearity in \eqref{t0}}
$\mathsf{z}_{\mathfrak{3}}$ and $\lbrace \mathsf{z}_{\mathbf{n}}: \mathbf{n} \in \mathbb{N}_0^{1 + d} \rbrace$.
Namely we will denote\footnote{we emphasise that the sum over $\beta$ is infinite and should be interpreted as an element of the (rigorously defined) space $X [[ \mathsf{z}_{\mathfrak{3}}, \mathsf{z}_{\mathbf{n}} ]]$ of formal power series in $\mathsf{z}_{\mathfrak{3}}$ and $(\mathsf{z}_{\mathbf{n}})_{\mathbf{n}}$}
	\begin{align}
		\Pi_x ( \cdot ) \coloneqq \sum_{\beta} \Pi_{x \beta} ( \cdot ) \, \mathsf{z}^{\beta} \in X [[ \mathsf{z}_{\mathfrak{3}}, \mathsf{z}_{\mathbf{n}} ]] , \nonumber
	\end{align}
where
$X = C^{\infty} ( \mathbb{R}^{1+d} )$ denotes the algebra of 
space-time 
(smooth)
functions,
and 
	\begin{align}
		\mathsf{z}^{\beta} \coloneqq \mathsf{z}_{\mathfrak{3}}^{\beta ( \mathfrak{3} )} \prod_{\mathbf{n} \in \mathbb{N}_0^{1 + d}} \mathsf{z}_{\mathbf{n}}^{\beta ( \mathbf{n} )} . \nonumber
	\end{align}

Plugging into the renormalised equation \eqref{t1} suggests that $\Pi_x$ should also satisfy a PDE of the form
	\begin{align}
		L \Pi_x = \Pi_x^- \text{ mod analytic functions} , \qquad \Pi_x^{-} \coloneqq \mathsf{z}_{\mathfrak{3}} \Pi_x^3 + c \Pi_x + \xi \mathsf{1} . \label{t65}
	\end{align}
Here 
$c \coloneqq \sum_{k \geq 1} c_k \mathsf{z}_{\mathfrak{3}}^k$ should be such that $h_{\lambda} = c[\lambda, p] = \sum_{k \geq 1} c_k \lambda^k$,
$\mathsf{1}$ 
is\footnote{contrary to the tree-based literature, here the element $\mathsf{1}$ does \emph{not} correspond to the constant polynomial, which is in fact associated to $\mathsf{z}_{\mathbf{0}}$}
the unit element of the algebra $X [[ \mathsf{z}_{\mathfrak{3}}, \mathsf{z}_{\mathbf{n}} ]]$ with coefficients $\mathsf{1}_{\beta} = 0$ unless $\beta = 0$ in which case the coefficient is given by the unit element of $X$.

\medskip
\textbf{Population conditions.}
As it turns out, only a strict subset of multi-indices $\beta$ contribute to the hierarchy of linear PDEs \eqref{t65}.
More precisely, it is shown in \cite[Section~1.9]{BOT} that \eqref{t65} is compatible with the property that $\Pi_{x \beta}$ and $\Pi_{x \beta}^-$ vanish unless $\beta$ is \emph{populated}, where we say that
\begin{align*}
\begin{array}{l}
\beta\;\mbox{is \emph{purely polynomial (pp)}}
\mbox{ iff }
\beta=\delta_{\bf n}\;\mbox{for some ${\bf n}$},\\
\beta\;\mbox{is \emph{populated}}
\mbox{ iff }
\beta\;\mbox{is pp or of the form $\delta_{\mathfrak{3}}\hspace{-.4ex}+\hspace{-.4ex}\sum_{j=1}^3\delta_{{\bf n}_j}$ or}\;[\beta]\geq 0.
\end{array}
\end{align*}
Here, by $\delta_{\mathfrak{3}}$ and $\delta_{\mathbf{n}}$ we mean the multi-index
$\lbrace \mathfrak{3} \rbrace \sqcup \mathbb{N}_0^{1 + d} \to \mathbb{N}_0$ that associates the value 1 to the indices $\mathfrak{3}$ and $\mathbf{n}$, respectively, and 0 elsewhere.
Furthermore, $[\beta] \coloneqq 2 \beta ( \mathfrak{3} ) - \sum_{\mathbf{n}} \beta ( \mathbf{n} ) \in \mathbb{Z}$ is the homogeneity in the noise.
The name `purely polynomial' comes from the expression $\Pi_{x \delta_{\mathbf{n}}} ( \cdot ) = (\cdot - x)^{\mathbf{n}}$, which is easily seen to be compatible with the expansion \eqref{t56} when $\lambda=0$.

\medskip
\textbf{Homogeneity}.
As motivated in \cite[Section~1.10]{BOT},
the scaling behaviour of $\Pi_{x \beta}$ is 
characterised by an exponent $| \beta | \in \mathbb{R}$, called the homogeneity of $\beta$ and defined by the expression
	\begin{align} \label{t45}
		|\beta| = \alpha + 2 \beta ( \mathfrak{3} ) (1 + \alpha) + \sum_{\mathbf{n}} \beta ( \mathbf{n} ) ( | \mathbf{n} | - \alpha ) .
	\end{align}
Throughout this paper, we will assume for populated $\beta$ that
	\begin{align} \label{t93}
		\text{$| \beta | \in \mathbb{Z}$ implies $\beta$=pp or of the form $\delta_{\mathfrak{3}}\hspace{-.4ex}+\hspace{-.4ex}\sum_{j=1}^3\delta_{{\bf n}_j}$,} 
	\end{align}
which follows e.g.\ by \eqref{t45} if one further assumes\footnote{this is a very mild assumption since one may usually reduce infinitesimally the value of $\alpha$ at no cost}
$\alpha \notin \mathbb{Q}$, and is in line with \cite[Assumption~5.3]{Hai14}.

\medskip
\textbf{Further notations}.
In order to rigorously state the analytic estimate of the model, and throughout the remainder of this paper, 
we will use the following notation:
given
any Schwartz function
$\psi \in \mathcal{S} (\mathbb{R}^{1 + d})$ and $\mu >0$, we will denote by $\psi_{\mu}$ its corresponding parabolic rescaling, that is, 
		\begin{align*}
		\psi_{\mu} ( x ) \coloneqq \mu^{- D} \psi \big( \mu^{-2} x_0, \mu^{-1} x_1 , \cdots , \mu^{-1} x_d \big) , 
	\end{align*}
where $D \coloneqq 2 + d$
is the `effective dimension' coming from the parabolic scaling.
Furthermore, given a Schwartz distribution $F$,
we will denote, 
keeping $\psi$ implicit, 
	\begin{align}
		F_{\mu} ( x )
		\coloneqq \langle F, \psi_{\mu} ( x - \cdot ) \rangle , \label{t87}
	\end{align}
so that\footnote{we note that there is essentially only a notational difference with the perhaps more standard choice of pairing $\langle F, \psi_{\mu} ( \cdot - x ) \rangle$}
$F_{\mu} ( x ) = F * \psi_{\mu} ( x )$ amounts to an averaging of $F$ at scale $\mu$ around $x$.
Finally, distances in $\mathbb{R}^{1 + d}$ will be measured in terms of the
parabolic Carnot--Carath\'eodory metric
	\begin{align*}
		| x | \coloneqq \sqrt[4]{ | x_0 |^2 + \Big( \sum\nolimits_{i=1}^d x_i^2 \Big)^2} ,
		\quad \text{for } x = (x_0, \cdots , x_d) .
	\end{align*}

\subsubsection{Rigorous Definitions} \label{ss:defs}

In the spirit of \cite{OSSW}, the majority of this work will be interested in providing a deterministic analysis for \eqref{t1} under the assumption that a suitable (deterministic) model is given. 

\medskip
We first make precise our definition of a model.
Note that at this stage we do not assume any smoothness, so that expressions such as the second item in \eqref{t65} do not yet make sense and must be added later alongside a suitable smoothness assumption.
Similarly, this definition does not yet reflect the choice of periodic boundary conditions.
This definition stays mostly faithful to Hairer's original definition 
(cf.~ \cite[Definition 2.17, Assumption 5.3]{Hai14}).

\medskip
For conciseness, we express this definition in terms of the distinguished Schwartz function
	\begin{align}
		\psi \coloneqq \mathcal{F}^{-1} \big( q \mapsto \exp ( -| q |^4) \big) , \label{t90}
	\end{align}
where
$\mathcal{F}$ denotes the Fourier transform\footnote{defined by $\mathcal{F} ( \psi ) \colon q \mapsto \int_{\mathbb{R}^{1 + d}} d x \, \psi ( x ) \exp ( - i x \cdot q )$} on $\mathcal{S}(\mathbb{R}^{1+d})$.

\medskip
We now fix once and for all a length $\ell > 0$, which we think of as the torus size.
Since we are interested in \emph{space-time} periodicity, 
throughout this paper
by periodicity we will mean periodicity w.~r.~t.~ the 
parabolic lattice $\ell^2 \mathbb{Z} \times ( \ell \mathbb{Z} )^d$.

\begin{definition}[Model] \label{d:model}
    A \emph{model}
    is a triple $(\Pi, \Pi^-, \Gamma^*)$ where 
$\Pi_x, \Pi_x^- : \mathcal{S}(\mathbb{R}^{1+d}) \to \mathbb{R}[[\mathsf{z}_{\mathfrak{3}}, \mathsf{z}_{\mathbf{n}}]]$ are a family of (linear) maps indexed by $x \in \mathbb{R}^{1+d}$ 
and $\Gamma_{x y}^*\in\mathrm{End}(\mathbb{R}[[\mathsf{z}_{\mathfrak{3}}, \mathsf{z}_{\mathbf{n}}]])$ is a family of endomorphisms indexed by $x, y \in \mathbb{R}^{1 + d}$,
such that
for all $\eta \in \mathbb{R}$
		\begin{align}
            \| \Pi \|_{\eta} & \coloneqq \sup_{\substack{x \in \mathbb{R}^{1 + d} \\ |\beta| < \eta}} 
             \Big( \sup_{\mu < \infty} \mu^{-|\beta|} |\Pi_{x\beta \mu}(x)| + 
         \sup_{\mu < \ell} \mu^{-(|\beta| - 2)} |\Pi_{x\beta \mu}^-(x)| \Big) < \infty , \label{mb01}
          \\ 
          \|\Gamma^*\|_{\eta} & \coloneqq \sup_{\substack{x,y \in \mathbb{R}^{1+d} \\ |\beta|, |\gamma| < \eta}}
        |y-x|^{|\gamma| - |\beta|} |(\Gamma_{xy}^*)_\beta^\gamma| < \infty , \label{mb03}
        \end{align}

where convolution is with respect to the kernel defined in \eqref{t90}.
   This triple should satisfy the (algebraic) conditions\footnote{note that the third condition in \eqref{mb04} along with \eqref{mb05} imply the identity $\Gamma_{x x}^* = \mathrm{id}$}
    \begin{align}
        &\Gamma_{xy}^* \Pi_y = \Pi_x, \qquad \Gamma_{xy}^* \Pi_y^- = \Pi_x^-, \qquad \Gamma_{xy}^* \Gamma_{yz}^* = \Gamma_{xz}^* , 
        \label{mb04}
        \\
        & (\Gamma_{xy}^* - \operatorname{id})_\beta^\gamma = 0 \text{ unless } |\gamma| > |\beta| , \label{mb05}
     \end{align}
We also require that\footnote{note that the first two conditions in \eqref{mb07} imply the identity $\Gamma^* \mathsf{1} = \mathsf{1}$ \label{fn:1}}
     \begin{alignat}{3}
     	 & \Pi_{x \delta_{\mathbf{n}}} = (\cdot - x)^{\mathbf{n}}, \quad \hspace{-.5ex}
        && \Pi_{x \delta_{\mathbf{n}}}^- = 0,
        \quad \hspace{-.5ex}
        && \Pi_{x \beta}^- = \sigma_{\beta} (\cdot - x)^{\bf n} , \label{mb08} \\
        & \Gamma^* \pi \tilde{\pi} = (\Gamma^* \pi) \Gamma^* \tilde{\pi} , 
		\quad \hspace{-.5ex}
        && \Gamma_{x y}^* \mathsf{z}_{\mathfrak{3}} = \mathsf{z}_{\mathfrak{3}} ,
        \quad \hspace{-.5ex}
        && 
       (\Gamma_{x \, x + h}^*)_{\delta_{\mathbf{m}}}^{\gamma} \hspace{-.5ex} = \hspace{-.5ex}
        \sum_{\mathbf{n}} \binom{\mathbf{m}}{\mathbf{n}} h^{\mathbf{m} - \mathbf{n}} 
        \mathbf{1}_{\lbrace \gamma=\delta_{\mathbf{n}} \rbrace} ,
        \label{mb07} 
    \end{alignat}
where in the last expression of \eqref{mb08}, $\beta$ is of the form
$\beta=\delta_{\mathfrak{3}}+\sum_{j=1}^3\delta_{{\bf n}_j}$, 
${\bf n}=\sum_{\bf m}{\bf m}\beta({\bf m})$ and 
$\sigma_{\beta}:=(\sum_{\bf m} \beta ({\bf m})) !/\prod_{\bf m}(\beta({\bf m})!)$ is a combinatorial factor.
Finally, we demand the following population conditions: for $\beta$ populated and $\gamma$ non populated, 
	\begin{align} \label{mb09}
		\Pi_{x \gamma} & = 0, 
		\qquad \Pi_{x \gamma}^{-} = 0 ,
		\qquad (\Gamma_{x y}^*)_{\gamma}^\beta = 0 .
	\end{align}
\end{definition}

This definition is very similar to the one of Hairer \cite[Definition 2.17]{Hai14}. It expresses the analytic constraints in \eqref{mb01}, \eqref{mb03}, the algebraic constraints in \eqref{mb04} \eqref{mb05} and the standard action on the polynomial part of the structure in \eqref{mb08}, \eqref{mb07}. In the following two remarks we highlight two minor differences to the definition of Hairer. 
\begin{remark}
[Dual Perspective]
We adopt a dual perspective in which $\Gamma^*$ acts on $\Pi, \Pi^-$ whilst Hairer's $\Gamma$ would act on trees.
We emphasise that the triangular structure in \eqref{mb05} has the opposite direction to that of Hairer. Whilst the first expression in \eqref{mb07} at first appears similar to the multiplicativity that holds in the tree-based setting, it expresses a different algebraic structure naturally arising from the multi-index setup. 
\end{remark}

\begin{remark}
[On the Analytic Bounds]
The analytic constraints \eqref{mb01} and \eqref{mb03} encode global (as opposed to local) bounds in the basepoint and a different upper bound on length scales. 
Both are related to our choice of space-time periodic boundary conditions, which require an adaptation of the construction given in \cite{BOT}. 
Our bounds also differ from Hairer's in that we impose bounds on $\Pi, \Pi^-$ only when tested against a distinguished kernel rather than uniformly over a ball of test functions. However, by the change of kernel result contained in Lemma~\ref{kernel_swap}, our condition is equivalent to a bound holding uniformly over test functions $\psi$ satisfying $\|\psi\|_p \le 1$
for a sufficiently large choice of\footnote{unlike in the setting of compactly supported test functions, the minimal value of $p$ depends on $\eta$ in addition to $\alpha$ and $d$ due to requirements of integrability at $\infty$} $p$ where we write
\begin{align} \label{ssn}
    \|\psi\|_p = \sum_{|\mathbf{n}|, |\mathbf{m}| \le p} \sup_{x \in \mathbb{R}^{1+d}} |x^{\mathbf{n}} \partial^{\mathbf{m}} \psi(x)|
\end{align}
for the $p$-th Schwartz seminorm. We will freely make use of this equivalence throughout the paper.
\end{remark}

In addition to Definition~\ref{d:model}, 
we need to encode the data of the PDE in the form of the differential operator and the boundary condition.

\begin{assumption}\label{ass1}
We assume to be given a model $(\Pi, \Pi^-, \Gamma^*)$ with the following additional properties.
We assume that
       \begin{align}\label{Hi01}
        L \Pi_{x \beta} &= \Pi_{x \beta}^- \text{ mod polynomials of degree $\leq | \beta |$ ,}
    \end{align}
and that\footnote{the first item in \eqref{t15} should of course be interpreted distributionally in the usual way}
for $x, y \in \mathbb{R}^{1 + d}, h \in \ell^2 \mathbb{Z} \times ( \ell \mathbb{Z} )^{d}$,
and multi-indices $\beta, \gamma$,
	    \begin{align} 
    		\Pi_{x + h \, \beta} ( \cdot + h ) = \Pi_{x \beta} ( \cdot ) , 
    		\qquad ( \Gamma_{x + h \, y + h}^* )_{\beta}^{\gamma} = ( \Gamma_{x y}^* )_{\beta}^{\gamma} . \label{t15}
    	\end{align} 
\end{assumption}

As is typical in regularity structures, our solution theory will actually require only the components of the model up to some fixed truncation order.
We now fix a parameter $\kappa$ that will play the role of this truncation order. 
For technical reasons, we assume that
	\begin{align}
		\kappa \in (2, \infty ) \setminus \lbrace |\beta| : \beta \text{ is a populated multi-index} \rbrace . \label{rs01}
	\end{align}

\medskip

Our next definition is that of a smooth model, which will be needed in the existence part of our well-posedness result.
In practice, it is often constructed through mollification of the noise so that $\xi_{\rho}$ and all the subsequent components of $\Pi, \Pi^-$ are actually smooth.
However, it turns out that we only need limited H\"older regularity instead of smoothness. This can be compared with \cite[Remark 8.28]{Hai14}.
Thus, we prefer to state 
relatively minimalistic
assumptions,
which are expressed in terms of the truncation parameter $\kappa$ introduced just above.

\begin{definition}[Smooth model] \label{d:smooth}
We define a \emph{smooth model} to be a model satisfying Assumption~\ref{ass1} such that for each multi-index $\beta$,
$\Pi_{x\beta}$ is a (locally) $\kappa$-H\"older function
and is such that
\begin{align}\label{Hi02}
    \Pi_{x}^- = \mathsf{z}_{\mathfrak{3}} \Pi_x^3 + c \Pi_x + \xi \mathsf{1}
\end{align}
for a given $\xi$ and some $c = \sum_{k \ge 1} c_k \mathsf{z}_{\mathfrak{3}}^k$ with $c_k \in \mathbb{R}$ and $c_k = 0$ unless $k < (1+ \alpha)^{-1}$.
\end{definition}
\begin{remark} \label{r:1}
We note that Definition~\ref{d:smooth} does not stipulate the regularity of $\Pi^-$.
However, since $\Pi_{x \beta}$ is (locally) $\kappa$-H\"older, it follows from \eqref{Hi02} and \eqref{Hi01} that
for $\beta = 0$, $\Pi_{ x \, \beta = 0}^- = \xi = L \Pi_{x \, \beta = 0}$ is (locally) $(\kappa - 2)$-H\"older and
for $\beta \neq 0$, $\Pi_{x \beta}^-$ is (locally) $\kappa$-H\"older.

\medskip
We also note that the choice of a smooth model uniquely determines $c$ via \eqref{Hi02}.
\end{remark}

\begin{assumption}\label{ass2}
    We fix a model as in Assumption~\ref{ass1} and assume to be given a sequence of smooth models $(\Pi^{\rho}, \Pi^{-,\rho}, \Gamma^{*,\rho})_{\rho \downarrow 0}$ such that
    \begin{align*}
        \|\Pi - \Pi^{\rho}\|_{\kappa} + \|\Gamma^* - \Gamma^{*, \rho} \|_{\kappa} \to 0
    \end{align*}
    as $\rho \downarrow 0$.
\end{assumption}

\begin{remark}
Let us mention again that the existence of models satisfying Assumptions~\ref{ass1} and \ref{ass2} does \emph{not} follow from the related paper \cite{BOT}. 
This is because the models constructed therein do not satisfy the periodic boundary condition which we demand here in the form of \eqref{t15}.
In the remainder of the present article, we take Assumptions~\ref{ass1} and \ref{ass2} for granted and do not give a full construction of such a sequence of models, see however Appendix~\ref{s:model} for a sketch of the necessary adaptations to \cite{BOT}.
\end{remark}

\subsection{Modelled Distributions and Robust Formulation}

\medskip
\textbf{Motivation of robust formulation}.
Let us motivate and introduce our robust notion of solution for \eqref{t1}.
We return to the formal power series expansion \eqref{t56}.
We emphasise that in general,
this expansion 
is not expected to converge. 
In practice, this ansatz thus needs to be modified, which we do in two ways.
On the one hand, we truncate this expansion
at the order $\kappa$,
with hope that it produces a sufficiently good local approximation of the desired solution $u$ around the point $x \in \mathbb{R}^{1 + d}$.
In order to still have good approximations around any given point, we are led to consider the family of expansions indexed by the basepoint $x \in \mathbb{R}^{1 + d}$.
On the other hand, 
we need to provide coefficients $\frac{1}{\mathbf{n}!} \partial^{\mathbf{n}} p ( x ) \eqqcolon f_x . \mathsf{z}_{\mathbf{n}}$ such that the family of local approximations corresponds to a single distribution that satisfies the boundary constraint for the problem.
Because of the truncation, we will actually never need those $\mathbf{n}$ with $| \mathbf{n} | > \kappa$, so that we will not attempt to construct
a full analytic function $p$.
Rather, we replace the problem with that of finding 
a finite number of
scalar fields $(f_x . \mathsf{z}_{\mathbf{n}})_{x, \mathbf{n}}$, 
$x \in \mathbb{R}^{1 + d}$, 
$| \mathbf{n} | < \kappa$,
such that
for all $x$, 
	\begin{align} \label{t57}
		u ( \cdot ) = \sum_{| \beta | < \kappa} \Pi_{x \beta} ( \cdot) \, \lambda^{\beta ( \mathfrak{3} )} \prod_{\mathbf{n} \in \mathbb{N}_0^{1 + d}} (f_x . \mathsf{z}_{\mathbf{n}} )^{\beta ( \mathbf{n} )} + O ( | \cdot - x |^{\kappa} ) .
	\end{align}
The r.~h.~s.~of \eqref{t57} inspires Definition~\ref{rf01} of
what 
we will call
a
modelled distribution $f$ in this work.
On the other hand,
our robust formulation, 
Definition~\ref{rf03},
makes the notion of approximation 
that is denoted by $O ( | \cdot - x |^{\kappa} )$ in \eqref{t57} precise, where we take again $\kappa$ as in \eqref{rs01}.
\medskip

Given $\eta \in \mathbb{R}$, we will denote by 
	\begin{align} \label{nr11}
		Q_{\eta} \colon \mathbb{R} [[ \mathsf{z}_{\mathfrak{3}}, \mathsf{z}_{\mathbf{n}} ]] \to \mathbb{R} [ \mathsf{z}_{\mathfrak{3}}, \mathsf{z}_{\mathbf{n}} ] \subset  \mathbb{R} [[ \mathsf{z}_{\mathfrak{3}}, \mathsf{z}_{\mathbf{n}} ]] ,
	\end{align}
the projection\footnote{the fact that $Q_{\eta}$ is indeed valued in the polynomial algebra $\mathbb{R} [ \mathsf{z}_{\mathfrak{3}}, \mathsf{z}_{\mathbf{n}} ]$ follows from the coercivity of the homogeneity $| \cdot |$, recall \cite[Section~1.9]{BOT}}
onto multi-indices of homogeneity $| \cdot | < \eta$.

\begin{definition}\label{rf01}
A \emph{modelled distribution} is an 
	\begin{align*}
		f \colon \mathbb{R}^{1 + d} \to ( \mathbb{R} [\mathsf{z}_{\mathfrak{3}}, \mathsf{z}_{\mathbf{n}}] )^* ,
	\end{align*}
of the form
				\begin{align}\label{rf02}
					f_x . \pi = \sum_\beta \pi_{\beta} \lambda^{\beta ( \mathfrak{3} )} \prod_{| \mathbf{n} | < {\kappa}} ( f_x . \mathsf{z}_{\mathbf{n}} )^{\beta ( \mathbf{n} )} .
				\end{align}
Equivalently, 
for all 
$\pi, \tilde{\pi} \in \mathbb{R} [\mathsf{z}_{\mathfrak{3}}, \mathsf{z}_{\mathbf{n}}]$, one has\footnote{note that, provided $\lambda \neq 0$, the conditions \eqref{t8} and \eqref{t8b} imply $f_x.\mathsf{1} = 1$ \label{fn:2}} 
			\begin{align}
				f_x . \pi \tilde{\pi} & = (f_x . \pi) ( f_x . \tilde{\pi} ) , \label{t8} \\
				f_x . \mathsf{z}_{\mathfrak{3}} & \eqqcolon \lambda \quad \text{does not depend on\footnotemark $\, x$}, \label{t8b} \\
				f_x . \mathsf{z}_{\mathbf{n}} & = 0 \quad \text{for $| \mathbf{n} | > {\kappa}$, $\mathbf{n} \in \mathbb{N}_0^{1 + d}$}. \label{t5}
			\end{align}	
\footnotetext{note that \eqref{md01} in combination with the fact that $\Gamma^* \mathsf{z}_{\mathfrak{3}} = \mathsf{z}_{\mathfrak{3}}$ already imply that $x \mapsto f_x . \mathsf{z}_{\mathfrak{3}}$ is constant}
Furthermore, we assume 
			\begin{align}\label{md01}
				\| f \|_{\kappa} \coloneqq \sup_{|\beta| < \kappa} \left ( \sup\limits_{x \in \mathbb{R}^{1 + d}} | f_x . \mathsf{z}^\beta | + \sup_{x, y \in \mathbb{R}^{1 + d}} \frac{| ( f_y . - f_x . Q_{\kappa} \Gamma_{x y}^* ) \mathsf{z}^\beta |}{| y - x |^{\kappa - | \beta |}} \right ) < \infty
			\end{align}
where the supremum in $\beta$ is over populated 
non-zero 
multi-indices.
\end{definition}

\begin{remark}[Comparison with {\cite[Definition~3.1]{Hai14}}]
	This definition is more restrictive than the corresponding notion in \cite{Hai14}. 
	The main difference is that our modelled distributions are assumed to be multiplicative, which corresponds to the notion of coherence in the tree-based literature; see \cite[Definition 3.20]{BCCH}. 
	We emphasise that as a consequence, our space of modelled distributions is not a linear
	space.
	In addition, as with our definition of a model, we enforce global rather than local bounds in \eqref{md01}, due to our choice of periodic boundary conditions. 
\end{remark}

\begin{remark}
To motivate the second quantity in the r.~h.~s.~ of \eqref{md01},
one can informally remove the truncation in \eqref{t57} and assume that $u = f_x . \Pi_x$ for all $x$.
Then by 
\eqref{mb04} 
we would have that 
$0 = f_y. \Pi_y - f_x . \Pi_x = (f_y. - f_x . \Gamma_{x y}^* ) \Pi_y$,
which suggests the 
reexpansion property $f_y = f_x \Gamma_{x y}^*$.
$\|f\|_\kappa$ 
quantifies the error in this reexpansion caused by the truncation at order $\kappa$. 
We refer to \cite[Section~3]{OSSW} for an expanded discussion.
\end{remark}

\begin{definition}\label{rf03}
We will say that a modelled distribution $f$ satisfies the \emph{robust formulation} of the equation \eqref{t1} 
with respect to the model $(\Pi, \Pi^-, \Gamma^*)$ if
	\begin{enumerate}
		\item the map $x \mapsto f_x$ is periodic,
		\item there exist periodic (Schwartz) distributions $u, u^-$ such that $\fint u = 0$, 
			\begin{align}\label{rfPDE}
				L u = P u^- , 
			\end{align}
			and the germs\footnote{we adopt the language of \cite{CZ} where a (measurable) map $\mathbb{R}^{1 + d} \to \mathcal{S}^{\prime} ( \mathbb{R}^{1 + d} )$ is called a Germ}
			\begin{align}\label{t12}
				R_x \coloneqq u - f_x . Q_{\kappa} \Pi_x, \quad \text{and}\quad R_x^- \coloneqq u^- - f_x . Q_{\kappa} \Pi_x^{-} ,
			\end{align}
			satisfy
			\begin{align}\label{t13}
				R_{x \mu} ( x ) \to 0 , \qquad R_{x \mu}^- ( x ) \to 0 ,
			\end{align}
			as $\mu \to 0$,
			locally uniformly in $x$, where in \eqref{t13} the convolution is with respect to the Schwartz function defined in \eqref{t90}.
	\end{enumerate}
\end{definition}

In \eqref{rfPDE}, we have denoted by $P$ the projection against (Schwartz) distributions of vanishing space-time average, the definition of which we now recall.
By standard distribution theory, 
periodic distributions $F$ admit a Fourier series with coefficients\footnote{formally written as 
$\hat{F} ( k ) = \int_{[ 0, \ell^2 ) \times [0, \ell)^d} d x \, F( x ) \exp(-i k \cdot x)$} $\hat{F} ( k )$ 
for 
$k \in 2 \pi/\ell^2 \mathbb{Z} \times (2 \pi/\ell \mathbb{Z})^d$.
The space-time average is defined by $\fint F \coloneqq \hat{F} ( 0 )$.
Accordingly, $P F \coloneqq F - \fint F$.

\begin{remark}[Consistency]
We note that
the conditions in Definitions~\ref{rf01} and \ref{rf03} relate to the modelled distribution $f$, whilst the desired solution is a distribution $u$.
In particular, even in the smooth case,
it might at first not seem clear how to pass from a smooth solution $u$ of \eqref{t1} to a corresponding modelled distribution $f$ satisfying the robust formulation.
This connection between the PDE and the robust formulation will be made precise in Lemma~\ref{l:c} below.
\end{remark}

\begin{remark}[Generalised Da Prato--Debussche] \label{r:gDP}
Note that since only multi-indices of the form $\beta = k \delta_{\mathfrak{3}}$ may have negative homogeneity, the relation $\Gamma_{x y}^* ( \mathsf{z}_{\mathfrak{3}}^k ) = \mathsf{z}_{\mathfrak{3}}^k$ implies that,
for these $\beta$, $\Pi_{x \beta}$ does not depend on $x$.
Hence, appealing to the model estimates \eqref{mb01} and \eqref{t13}, we learn that 
	\begin{align*}
		\big( u - \sum\nolimits_{|\beta|<0} \lambda^{\beta ( \mathfrak{3} )} \Pi_{0 \beta} - f_x . \mathsf{z}_{\mathbf{0}} \big)_{\mu} ( x )
		\to 0 
	\end{align*}
	as $\mu \to 0$.
As a consequence, the `generalised Da Prato--Debussche remainder' $u - \sum_{|\beta|<0} \lambda^{\beta ( \mathfrak{3} )} \Pi_{0 \beta}$ (compare with \cite[Section~5.5]{BCCH}) is actually a function which coincides with $x \mapsto f_x . \mathsf{z}_{\mathbf{0}}$.
\end{remark}

\subsection{Main Results} \label{ss:main_results}

 Throughout the paper, various statements will involve a small parameter $\lambda_0$ which by default is allowed to depend on the dimension $d$, the torus size $\ell$, the truncation order $\kappa$, the bare regularity $\alpha$ and an additional free parameter $M$ which controls the size of various objects. Since this is mostly consistent, we will only make the dependence explicit when $\lambda_0$ depends on parameters in addition to this set. We are now prepared to state our main result, which is a well-posedness statement for the robust formulation of Definition~\ref{rf03}.

\begin{theorem}[Main Theorem]\label{main_result}
    For any $M< \infty$ there exists a $\lambda_0 > 0$ such that if $(\Pi, \Pi^-, \Gamma)$ is a model satisfying Assumptions~\ref{ass1} and \ref{ass2} and 
    $$\|\Pi\|_\kappa + \|\Gamma^*\|_\kappa \le M$$
    then for all $|\lambda| \leq \lambda_0$ there exists a unique solution $f$ to 
    the robust formulation, 
    Definition~\ref{rf03}, satisfying 
    \begin{align} \label{i4}
        \| f . \mathsf{z}_{\mathbf{0}} \| \coloneqq \sup_x |f_x. \mathsf{z}_{\mathbf{0}}| \le M.
    \end{align}
    
    Furthermore, the map $(\Pi, \Pi^-, \Gamma) \mapsto f$ is Lipschitz continuous with respect to the topologies induced by \eqref{mb01}-\eqref{mb03} 
    (with $\eta$ replaced by $\kappa$) 
    and \eqref{md01}.
\end{theorem}

The constraint of smallness \eqref{i4} on $f$ is natural in our setting, since we treat the parabolic equation \eqref{t0} as an elliptic one, 
where typically large solutions may coexist with small ones. 
Recalling Remark~\ref{r:gDP}, one notes that
	\begin{align} \label{t48}
	\| f . \mathsf{z}_{\mathbf{0}} \|
		= \sup_{x \in \mathbb{R}^{1 + d}} \Big| \Big( u - \sum\nolimits_{| \beta | < 0} \lambda^{\beta ( \mathfrak{3} )} \Pi_{0 \beta} \Big) ( x ) \Big| ,
	\end{align}
so that the quantity $\| f . \mathsf{z}_{\mathbf{0}} \|$ can be interpreted as the supremum norm of the `generalised Da Prato--Debussche remainder’.

\medskip
Theorem~\ref{main_result} actually follows from 
a
straightforward combination of the three following results.

\begin{theorem}[A Priori Estimates]\label{a priori}
    For any $M < \infty$ there exists $\lambda_0 > 0$ such that if $(\Pi, \Pi^-, \Gamma^*)$ is a model satisfying Assumption~\ref{ass1},
    if $f$ solves the corresponding robust formulation with $|\lambda| \leq \lambda_0$ and also    
    \begin{align*}
        \|\Pi\|_\kappa + \|\Gamma^*\|_\kappa + \| f . \mathsf{z}_{\mathbf{0}} \| \le M
    \end{align*}
    then
    \begin{align} \label{t58}
        \|f\|_{\kappa} \lesssim |\lambda| 
    \end{align}
    where the implicit constant depends on $M, \alpha, d, \kappa$ and $\ell$.
\end{theorem}

\begin{theorem}[Uniqueness/Continuity in the Model] \label{uniqueness}
    For any $M < \infty$ there exists $\lambda_0 > 0$ such that if $(\Pi, \Pi^-, \Gamma^*)$ and $(\tilde{\Pi}, \tilde{\Pi}^-, \tilde{\Gamma}^*)$ 
    are models
     and $f, \tilde{f}$ are solutions of the corresponding robust formulations
which satisfy the hypotheses of 
Theorem~\ref{a priori}
then 
    \begin{align}
    \|f; \tilde{f}\|_{\kappa} \lesssim \|\Pi - \tilde{\Pi}\|_\kappa
    	+ \|\Gamma^* - \tilde{\Gamma}^* \|_\kappa.  \label{t91}
    \end{align}
    Here the implicit constant depends on $M, \alpha, d, \kappa$ and $\ell$. Furthermore $\|f; \tilde{f}\|_{\kappa}$ is defined to be the quantity
    \begin{align*}
         \sup_{|\beta| < \kappa} \left (\sup_{x} |(f_x - \tilde{f}_x). {\mathsf{z}}^\beta | + \sup_{x, y} \frac{|(f_y . - f_x . Q_{\kappa} \Gamma_{xy}^* - \tilde{f}_y . + \tilde{f}_x . Q_{\kappa} \tilde{\Gamma}_{xy}^*) {\mathsf{z}}^\beta|}{|y-x|^{\kappa - |\beta|}} \right ) .
    \end{align*}
    where
    the supremum in $\beta$ is over populated multi-indices.
\end{theorem}
In order to see that this result indeed contains a uniqueness statement, 
it suffices to take the two models to be equal, so that \eqref{t91} implies that $f$ and $\tilde{f}$ must coincide.

\begin{theorem}[Existence] \label{existence}
    For any $M < \infty$ there exists a $\lambda_0 > 0$ such that if $(\Pi, \Pi^-, \Gamma_{}^*)$ is a smooth model
as in Definition~\ref{d:smooth} 
     such that 
    \begin{align} \label{t31}
        \|\Pi\|_\kappa + \|\Gamma^*\|_\kappa \le M ,
    \end{align}
   and $|\lambda| \leq \lambda_0$ then there exists a 
    solution $f$ of the robust reformulation which satisfies $\| f . \mathsf{z}_{\mathbf{0}} \| \leq M$.
    Furthermore, if $u$ is as in \eqref{rfPDE} then 
    \begin{align}\label{t30}
        Lu = P( \lambda u^3 + h_{\lambda} u + \xi ),
      \qquad \fint u = 0 ,
    \end{align}
    where
    \begin{align}\label{t22}
    		h_{\lambda} = \sum\nolimits_k c_k \lambda^k . 
    \end{align}
\end{theorem}

\section{Proof Strategies} \label{s:strat}

\subsection{A Priori Estimates: Strategy Of Proof}\label{ss:apriori}

We now sketch the strategy of proof for the a priori estimate for solutions of the robust formulation stated in Theorem~\ref{a priori}, which proceeds in five steps. 
Here our approach is heavily inspired by the one taken in \cite{OSSW} in the sense that the first four steps of our proof can be directly compared to those of \cite{OSSW}. 
In particular, these steps consist of applications of an Algebraic Continuity Lemma, a Reconstruction bound, a Schauder estimate and finally a post-processing of the output of the Schauder estimate via a Three-Point Argument. 
We note that a similar loop of estimates also appeared in 
\cite{CMW,MW18,EW24} where the strong damping effect of the cubic nonlinearity (when $\lambda < 0$) was further exploited to derive global a priori estimates, 
see also \cite{CFW24} in the case of the gPAM equation.
However our approach does differ in several points which we highlight below.
\medskip

In the first step, we establish our \emph{Algebraic Continuity Lemma} which is an analogue of the graded continuity lemma given in \cite[Section 3.1]{OSSW}. The idea is to use multiplicativity (cf. \eqref{rf02}) to upgrade the bounds \eqref{md01} on the purely-polynomial components of the modelled distribution $f$ to bounds on the non-purely-polynomial components. The crucial point in these estimates is that the right-hand side comes with a prefactor $|\lambda|$ and
depends on $[f]_{\kappa, \delta_{\mathbf{n}}}$ in an affine manner. Both of these facts are required in order to allow us to buckle after the fourth stage of our argument.

\medskip

 It will also be helpful to track the different components of $\|f\|_\kappa$. 
 To that end, for $\eta > 0$ and $\beta$ a populated multi-index with $|\beta| < \eta$ we introduce the quantities
\begin{align}
	 \| f . \mathsf{z}^{\beta} \| &
	 \coloneqq \sup_{x \in \mathbb{R}^{1+d}} |f_x. \mathsf{z}^{\beta}|  , \label{t46}
    \\
    [f]_{\eta, \beta} & \coloneqq \sup_{x, h \in \mathbb{R}^{1+d}} \frac{|(f_{x+h} . - f_x . Q_{\eta} \Gamma_{x \, x+h}^*) \mathsf{z}^{\beta}|}{|h|^{\eta - |\beta|}} , \nonumber
    \\
    [f]_{\eta, \mathrm{pol}} & \coloneqq \sum_{| \mathbf{n} | < \eta} [f]_{\eta, \delta_{\mathbf{n}}} , \label{rop15}
\end{align}
where in the first item (and throughout the remainder of this article)
we use the notation $\| \cdot \|$
without a subscript
to denote the supremum norm.

\begin{lemma}[Algebraic Continuity Lemma]\label{continuity}

    Fix a model $(\Pi, \Pi^-, \Gamma^*)$ and suppose that $f$ is a modelled distribution with 
    $|\lambda| + \|\Gamma^*\|_\kappa + \| f . \mathsf{z}_{\mathbf{0}} \| \le M$. 
    Then 
    for each populated and non-purely-polynomial multi-index $\beta$ with $|\beta| < \kappa$, 
    \begin{align}\label{md03}
        [f]_{\kappa, \beta} \lesssim |\lambda| ([f]_{\kappa, \mathrm{pol}} + \| f . \mathsf{z}_{\mathbf{0}} \| + |\lambda|) 
    \end{align}
    where the implicit constant depends only on $M, \alpha, d$ and $\kappa$. 
\end{lemma}

\begin{remark}
    In comparison to \cite[Section 3]{OSSW}, we note that our Algebraic Continuity Lemma is
	slightly more systematised in the sense that it is    
    valid to all truncation orders rather than only to truncation order $\kappa \in (2,3)$.
    
    \medskip

    In the case of the $\Phi^4$ equation, such a property was also observed 
in a tree-based context 
by \cite[Theorem~6.10]{CMW}, although only up to truncation order $2$. 
    We expect that a similar estimate should hold at a much higher level of generality and also in the tree-based setting, where in the latter case multiplicativity is replaced by the (multi-)pre-Lie morphism property 
of the coefficients of a coherent modelled distribution, 
as first described in the SPDE setting in \cite{BCCH} and in the related rough path setting in the earlier work \cite{BCFP19}. 
\end{remark}

\medskip
Our second step is a \emph{Reconstruction} step.
We observe that in combination with the assumed bounds on the model, the bound \eqref{md03} allows us to bound the so-called `coherence norms' of the germ $R_x$. 
This is the required input to apply a version of Hairer's Reconstruction Theorem \cite[Theorem~3.10]{Hai14}, 
which we write here in the language of germs \cite{CZ} and as an a priori estimate.
Since many versions of this result appear in the literature, we omit the proof here and instead refer the reader to \cite{Hai14,otto2018parabolic,MW18,FH,CZ} for several different proofs of mild variants of this result which would adapt in a straightforward manner to our setting.
\begin{lemma}[Reconstruction] \label{Recon}
      Fix $L \le \infty$, $\eta > 0$, $\alpha < 0$.
    Suppose that $x \mapsto F_x \in \mathcal{S}'(\mathbb{R}^{1+d})$ is a measurable map such that 
 	\begin{align*}
    	\sup_{x, h \in \mathbb{R}^{1+d}} \sup_{\mu \leq L} \frac{|(F_{x+h} - F_x)_{\mu}(x)|}{\mu^{\alpha} (|h| + \mu)^{\eta - \alpha}} \le 1, 
    \end{align*}
	where the convolution is with respect to the Schwartz function defined in \eqref{t90}.
    Suppose also that, 
    with respect to that same test function,
    $x \mapsto F_{x \mu}(x) \to 0$ in the sense of tempered distributions as $\mu \downarrow 0$.
    Then there exists $p > 0$ depending only on $\eta, \alpha$ and $d$ such that
    \begin{align*}
    	\sup_{\|\psi\|_p \le 1} \sup_{x \in \mathbb{R}^{1+d}} \sup_{\mu \leq L} \mu^{- \eta} |F_{x\mu}(x)| \le C
    \end{align*}
    where $C$ is a constant depending on $\alpha, \eta$ and $d$.
\end{lemma}

We note that the germ $R^-$ depends on the values of $f . \mathsf{z}^{\beta}$ only for non-purely polynomial multi-indices $\beta$. Indeed, by \eqref{mb08}, $\Pi_{x \delta_{\mathbf{n}}}^- = 0$ for $\mathbf{n} \in \mathbb{N}_0^{1 + d}$. 
In particular, applying Lemma~\ref{Recon} with $F = R^{-}$ in combination with the output of the Algebraic Continuity Lemma will yield the bound 
\begin{align}\label{md04}
    \frac{|R_{x\mu}^-(x)|}{\mu^{\kappa - 2}} \lesssim |\lambda| ([f]_{\kappa, \text{pol}} + \| f . \mathsf{z}_{\mathbf{0}} \| + |\lambda|) 
   \lesssim | \lambda | \, \| f \|_{\kappa} ,
\end{align}

for $\mu \in (0,\ell]$, where the upper bound on $\mu$ comes from the same upper bound in \eqref{mb01}.

\medskip
Our third step is an \emph{Integration} step.
Here, we make use of 
the PDEs \eqref{rfPDE} and \eqref{Hi01} satisfied by $u$ and the model respectively, 
which imply that $R_x$ and $R_x^-$ are coupled via the simple linear PDE
\begin{align}\label{t2}
	L R_x = R_x^- + P_x 
\end{align}
where $P_x$ is the polynomial
	\begin{align}\label{po1}
		P_x = - \fint u^- - \sum_{|\beta| < \kappa} f_x.\mathsf{z}^\beta (L\Pi_{x\beta} - \Pi_{x\beta}^-).
	\end{align}
In particular, we deduce from the PDE \eqref{Hi01} and the model bounds \eqref{mb01} that if $P_x = \sum_{\mathbf{n}} a_{x \mathbf{n}} (\cdot - x)^\mathbf{n}$ then $a_{x\mathbf{n}} = 0$ unless $|\mathbf{n}| < \kappa$.
On the other hand, using the fact that $P_x = L R_x - R_x^-$
together with \eqref{md04} and its analogue for\footnote{which we emphasise does not come with a prefactor $|\lambda|$ at this stage} $R$
that $| P_{x \mu} ( x ) | \lesssim \mu^{\kappa - 2}$ as $\mu \to 0$
whence $a_{x\mathbf{n}} = 0$ unless $\kappa - 2 < |\mathbf{n}|$.
This allows us to transfer our estimates on $R_x^-$ to ones on $R_x$.
In this step, our approach differs from the corresponding step in \cite{OSSW}. On the one hand, we do not have to be careful to get the optimal behaviour in 
terms of a quasilinearity that appears in their setting. On the other hand, in order to have a priori bounds that do not require qualitative smoothness, we have to take care to provide a proof that is robust to the lack of that assumption. To that end, we adopt the approach to Schauder estimates taken in \cite{BOT} which we refine to more carefully track the precise seminorms needed.
We note that as opposed to 
\cite[Proposition~2]{OSSW} (see also the recent \cite{SS24a})
our version of the Schauder estimate does not come with a `three-point condition' on the right-hand side since we work with a `pointed' Schauder estimate rather than a Schauder estimate for germs.
\begin{lemma}[Integration]\label{Schauder}
	Suppose that 
$\eta \in (0,\infty) \setminus \mathbb{Z}$, $p \in \mathbb{N}$ and
that  $v, v^-$  are distributions such that $Lv = v^-$ and that
\begin{align} \label{schauder_ass}
	\sup_{\|\psi\|_p \le 1} \sup_{\mu < \infty} \mu^{2 - \eta} |v^-_\mu (0) | \le 1.
\end{align}
Suppose further that there exists a $\psi \in \mathcal{S} ( \mathbb{R}^{1 + d} )$ with $\int \psi = 1$ such that $$\sup_{\mu < \infty} \mu^{-\eta} |v_\mu(0)| < \infty.$$
Then
there exists a $p^\prime \in \mathbb{N}$ depending only on $\eta, p$ and $d$ such that 
\begin{align}\label{int_bound}
	\sup_{\|\psi\|_{p^\prime} \le 1} \sup_{\mu < \infty} \mu^{-\eta} |v_\mu(0)| \lesssim 1
\end{align}
where the implicit constant depends only on $\eta, d$ and $p$.
\end{lemma}

The careful reader will note that there is a mismatch in the scales allowed in the estimate \eqref{md04} and the required input for Lemma~\ref{Schauder} since the former holds for $\mu \le \ell$ whilst the latter is required for all $\mu < \infty$. 
This detail will be taken care of via 
a separation of scales argument where large scales are estimated in a different way to small scales. Since the estimates for large scales will not come with a good prefactor $|\lambda|$, it will be necessary to carefully optimise in the division between large and small scales. To that end, we introduce a scale parameter $\rho$ which quantifies the difference between the large and small scales.
For the purpose of exposition we do not elaborate on this detail in this sketch proof and simply claim that we can apply Lemma~\ref{Schauder} with $v = R_x$, $v^- = R_x^-$ to the effect of 
\begin{align}\label{md06}
    \frac{|R_{x\mu}(x)|}{\mu^{\kappa}} \lesssim \Big( \Big( \frac{\rho}{\ell} \Big)^{-\delta} + \Big( \frac{\rho}{\ell} \Big)^{p} \, |\lambda| \Big) \|f\|_\kappa
\end{align}
uniformly over $\mu \leq \rho$, where 
$\delta, p > 0$
are fixed constants depending only on $\alpha, \kappa$ and $d$.

\medskip
In our fourth step, which consists of a \emph{Three-Point Argument}, we make use of \eqref{md06} in combination with the output of the Continuity Lemma \eqref{md03} to obtain estimates on $[f]_{\kappa, \delta_{\mathbf{n}}}$ for $|\mathbf{n}| < \kappa$. 
In \cite{OSSW}, the Three-Point Argument appears as a step within the integration argument therein, but was later isolated as a convenient tool in a related setting \cite[Proposition~4.4]{LOTT}.
Here we present a slight variant of the version in the aforementioned works, due to our distributional point of view and our localisation scale.

\begin{lemma}[Three-Point Argument]\label{3_point}
   There exists a finite set $\mathcal{B}$ of test funtions\footnote{which can be taken to be supported in $B_1(0)$} 
   such that if $f: \mathbb{R}^{1+d} \to (\mathbb{R}[\mathsf{z}_\mathfrak{3}, \mathsf{z}_\mathbf{n}])^*$ 
    and 
    \begin{align}\label{3ps01}
    	R_x ( y) \coloneqq f_y . \mathsf{z}_\mathbf{0} - \sum_{0 \le |\beta| < \kappa} f_x . \mathsf{z}^{\beta} \Pi_{x\beta} ( y)
    \end{align}
    then for $|\mathbf{n}| < \kappa$, we have that 
    \begin{align*}
        [f]_{\kappa, \delta_{\mathbf{n}}} \lesssim  \sup_{\psi \in \mathcal{B}} \sup_{x \in \mathbb{R}^{1 + d}} \sup_{\mu < \infty} \frac{|R_{x\mu}(x)|}{\mu^{\kappa}} + \|\Pi\|_{\kappa} \, \max_{\beta} [f]_{\kappa, \beta} 
    \end{align*}
    where the maximum is over populated non-purely-polynomial multi-indices $\beta$ such that $|\beta| < \kappa$ and the implicit constant depends only on $\alpha, d$ and $\kappa$.
\end{lemma}

\begin{remark}\label{3p_remark}
    In the above statement, since $f$ is not assumed to be a solution of the robust formulation, $R_x$ is not given as its reconstruction. However for the $\beta$ appearing in the sum $\Pi_{x \beta}$ is automatically a function since we can write
    \begin{align*}
    \Pi_{x \beta}(y) &= \left ( \Gamma_{xy}^* \Pi_{y} \right )_\beta (y) = \sum_{\gamma} (\Gamma_{xy}^*)_\beta^\gamma \Pi_{y \gamma}(y).
    \end{align*}
    Now if $|\gamma| < 0$ then $\gamma \in \{\mathsf{z}_{\mathfrak{3}}^k: k \in \mathbb{N}_0\}$. However $\Gamma_{xy}^* \mathsf{z}_{\mathfrak{3}}^k = \mathsf{z}_{\mathfrak{3}}^k$ so that no such multi-indices contribute. Furthermore $\Pi_{y \gamma}(y) = 0$ for $|\gamma| > 0$ so that 
	\begin{align*}
		\Pi_{x\beta}(y) = (\Gamma_{xy}^*)_{\beta}^{\delta_{\mathbf{0}}} .
\end{align*}

    In particular, it follows that when $f$ is a solution of the robust formulation one has that $R_x(y) = f_y . \mathsf{z}_{\mathbf{0}} - f_x . Q_\kappa \Gamma_{xy}^* \mathsf{z}_{\mathbf{0}}$ thus explaining our choice of notation.
\end{remark}

The combination of Lemma~\ref{3_point} and \eqref{md06} will allow us to estimate the part of $[f]_{\kappa, \delta_\mathbf{n}}$ corresponding to increments $h$ satisfying $|h| < \rho$. In the complementary regime, we simply make use of the triangle inequality and the seminorms $\| f . \mathsf{z}^{\beta} \|$. In total, this allows us deduce from the previous estimates that for $|\mathbf{n}| < \kappa$
\begin{align}\label{pb01}
    [f]_{\kappa, \mathrm{pol}} \lesssim \Big( \Big( \frac{\rho}{\ell} \Big)^{-\delta} + \Big( \frac{\rho}{\ell} \Big)^{p} \, |\lambda| \Big) \|f\|_\kappa .
\end{align}
At this point, we note that it follows from Lemma~\ref{continuity} and the interpolation estimates contained in its proof (see Lemma~\ref{l:intI} below) that
\begin{align} \label{pb02}
    \|f\|_\kappa \lesssim [f]_{\kappa, \mathrm{pol}} + \| f . \mathsf{z}_{\mathbf{0}} \| + |\lambda|.
\end{align}
Substituting this estimate into \eqref{pb01} is not quite sufficient for our purposes since a supremum-type norm of the modelled distribution $f$ still appears on the right-hand side.

\medskip
In order to correct for this, in our fifth and final step (\emph{Incorporating Boundary Data}), we make use of the periodic structure of the problem.
More precisely, 
observe by \eqref{t15} that for $|\beta|<0$, $\Pi_{0 \beta}$ is a periodic distribution of space-time average $\fint \Pi_{0 \beta} = \lim_{\mu \to \infty} \Pi_{0 \beta \mu} ( 0 ) = 0$ by the model bounds \eqref{mb01}.
It follows from Remark~\ref{r:gDP} that $x \mapsto f_x . \mathsf{z}_{\mathbf{0}}$ is periodic of vanishing space-time average. Therefore
\begin{align*}
    |f_x . \mathsf{z}_\mathbf{0}| \lesssim_{\ell, \eta} [f . \mathsf{z}_\mathbf{0} ]_{\eta}  \qquad \text{for any $\eta \in (0, 1)$,}
\end{align*}
where $[\cdot]_{\eta}$ denotes the (parabolic) $\eta$-H\"older seminorm
\begin{align}\label{b15}
		[ u ]_{\eta} & \coloneqq
		\sup_{x, y \in \mathbb{R}^{1 + d}} \frac{\big| u ( y ) - \sum\nolimits_{| \mathbf{n} | < \eta} \frac{1}{\mathbf{n}!} ( y - x )^{\mathbf{n}} \partial^{\mathbf{n}} u ( x ) \big|}{| x - y |^{\eta}} 
		\qquad \text{for any $\eta> 0$.}
\end{align}
Making the choice $\eta = \kappa^-$ with 
\begin{align} \label{i7}
    \kappa^- \coloneqq \min \{ |\beta| > 0: \beta \text{ is populated.}\}
\end{align}
leads to the estimate 
\begin{align}\label{md08}
    |f_x . \mathsf{z}_\mathbf{0}| \lesssim_{\ell, \eta} [f]_{\kappa^-, \mathrm{pol}}
\end{align}
since the right-hand side coincides with $[f.\mathsf{z}_\mathbf{0}]_{\kappa^-}$.
We now wish to replace the right-hand side of this estimate with a term depending instead on $[f]_{\kappa, \mathrm{pol}}$. It is essential that we do this in such a way that any dependence on $\| f . \mathsf{z}_{\mathbf{0}} \|$ appears with a small constant so that we may absorb this term into the left-hand side in the final estimate. To this end, we make use of the interpolation estimate 
\begin{align*}
	 [f]_{\kappa^-, \mathrm{pol}}
	& \lesssim ( [f]_{\kappa, \mathrm{pol}} + |\lambda| )^\frac{\kappa^-}{\kappa} \, ( \| f . \mathsf{z}_{\mathbf{0}} \| + |\lambda| )^{\frac{\kappa - \kappa^-}{\kappa}}
\end{align*}
which is part of the statement of Lemma~\ref{l:intI} appearing in the proof of Lemma~\ref{continuity}.
By combining this estimate, Young's inequality and \eqref{md08} we obtain the bound
\begin{align*}
	\| f . \mathsf{z}_{\mathbf{0}} \| \lesssim C_{\varepsilon} ([f]_{\kappa, \mathrm{pol}} + |\lambda|) + \varepsilon \| f . \mathsf{z}_{\mathbf{0}} \| 
\end{align*}
which is valid for any $\varepsilon > 0$. In particular, by taking $\varepsilon$ small enough this implies that
\begin{align*}
	\| f . \mathsf{z}_{\mathbf{0}} \| \lesssim [f]_{\kappa, \mathrm{pol}} + |\lambda| .
\end{align*}
Substituting this bound and the estimate \eqref{pb01} into the right-hand side of \eqref{pb02} yields the estimate
\begin{align*}
    \|f\|_\kappa 
    \lesssim \Big( \Big( \frac{\rho}{\ell} \Big)^{-\delta} + \Big( \frac{\rho}{\ell} \Big)^{p} |\lambda| \Big) \|f\|_\kappa 
    	+ |\lambda|.
\end{align*}
Taking first the localisation scale $\rho$ to be sufficiently large and then $|\lambda|$ sufficiently small so as to buckle completes the proof.

\subsection{Uniqueness and Continuity in the Model: Strategy of Proof}\label{ss:cont}

With our a priori estimate in hand, we now turn to sketch our strategy of proof for Theorem~\ref{uniqueness}. 
At its most basic, the idea is to apply the same kind of PDE estimates as sketched in the previous section to the difference of two solutions to the robust formulation. 
At the level of the nonlinear PDE \eqref{t1}, it may seem that these differences do not have good structure.
However, one has more hope when considering our robust formulation, since it replaces the nonlinear (and singular) PDE with a linear PDE that is coupled to a nonlinear constraint at the level of the corresponding modelled distributions.
At the level of differences, this algebraic constraint behaves in a way that is akin to a (discrete) Leibniz' rule.

\medskip
In particular, the parts of the argument sketched in the previous section that purely depend on the PDE \eqref{rfPDE} go largely unaltered. 
Instead, the main novelty in our proof of Theorem~\ref{uniqueness} in comparison to that of Theorem~\ref{a priori} is in the more algebraic steps; namely in adapting Lemma~\ref{continuity} and Lemma~\ref{3_point}. 
We have the following analogues of those results, which we state only for solutions of the robust formulation so as to make use of Theorem~\ref{a priori} to simplify the statements.

\begin{lemma}[Algebraic Continuity Lemma II]\label{continuity_2}
	Suppose that $(\Pi, \Pi^-, \Gamma^*)$, $(\tilde{\Pi}, \tilde{\Pi}^-, \tilde{\Gamma}^*)$ 
	are models
	and $f, \tilde{f}$ corresponding modelled distributions such that
	\begin{align*}
		|\lambda| + \|f\|_\kappa + \|\tilde{f}\|_\kappa + \|\Gamma^*\|_\kappa + \|\tilde{\Gamma}^*\|_\kappa \le M
	\end{align*}
 then for each populated and non-purely polynomial multi-index $\beta$ with $|\beta| < \kappa$, 
	\begin{align}\label{md09}
	 [f;\tilde{f}]_{\kappa, \beta} \lesssim |\lambda| \left ( [f; \tilde{f}]_{\kappa, \mathrm{pol}} + |\lambda| \|\Gamma^* - \tilde{\Gamma}^*\|_\kappa + \| (f - \tilde{f}) . \mathsf{z}_{\mathbf{0}} \|  \right )
	\end{align}
	where
	\begin{align*}
		[f; \tilde{f}]_{\eta, \beta} \coloneqq \sup_{x, h \in \mathbb{R}^{1+d}} \frac{|(f_{x+h}. - f_x. Q_\eta \Gamma_{x \, x+h}^* - \tilde{f}_{x+h}. + \tilde{f}_x. Q_\eta \tilde{\Gamma}_{x \, x+h}^*) \mathsf{z}^\beta|}{|h|^{\eta - |\beta|}}
	\end{align*}
        and 
	\begin{align} \label{refonp19}
		[f; \tilde{f}]_{\eta, \mathrm{pol}} = \sum_{|\mathbf{n}| < \eta} [f; \tilde{f}]_{\eta, \delta_\mathbf{n}} .
	\end{align}
\end{lemma}

\begin{lemma}[Three-Point Argument II]\label{3_point_2}
	There exists a finite set $\mathcal{B}$ of smooth functions\footnote{which can be taken to be supported in $B_1(0)$} such that if $(\Pi, \Pi^-, \Gamma^*)$, $(\tilde{\Pi}, \tilde{\Pi}^-, \tilde{\Gamma}^*)$ are models and 
	 $f, \tilde{f} \colon \mathbb{R}^{1+d} \to (\mathbb{R}[\mathsf{z}_\mathfrak{3}, \mathsf{z}_n])^*$ 
	then for $|\mathbf{n}| < \kappa$
	\begin{align*}
		[f; \tilde{f}]_{\kappa, \delta_{\mathbf{n}}} \lesssim  \sup_{\psi \in \mathcal{B}} \sup_{x \in \mathbb{R}^{1 + d}} \sup_{\mu < \infty} & \frac{|(R_x - \tilde{R}_x)(\psi_x^\mu)|}{\mu^\kappa}
		\\ &+ \max_\beta [\tilde{f}]_{\kappa, \beta} \|\Pi - \tilde{\Pi}\|_{\kappa} + \|\Pi\|_\kappa \max_{\beta} [f; \tilde{f}]_{\kappa, \beta}
	\end{align*}
	where the maxima are over populated non-purely polynomial multi-indices $\beta$ with $|\beta| < \kappa$ and we have defined 
    \begin{align*}
        	(R_x - \tilde{R}_x)(y) &  \coloneqq (f_y - \tilde{f}_y) . \mathsf{z}_\mathbf{0} - \sum_{0 \le |\beta| < \kappa}  \big ( f_x . \mathsf{z}^{\beta} \Pi_{x\beta} ( y)- \tilde{f}_x . \mathsf{z}^{\beta} \tilde{\Pi}_{x\beta} ( y) \big) .
    \end{align*}
\end{lemma}

\subsection{Pathwise Existence of Solutions: Strategy of Proof} \label{ss:exist_strat}

We now sketch our strategy of proof for the existence result of  Theorem~\ref{existence}.
This relies on constructing sufficiently smooth (in fact, $\kappa$-H\"older) solutions $u$ to equation \eqref{t30}, where $\xi$ is now sufficiently smooth (in fact, $(\kappa - 2)$-H\"older).
We emphasize that this result does not follow immediately from classical PDE theory.
The existence of such a $u$ in a range of $| \lambda | \leq \lambda_0$ is nontrivial, since it is essential that $\lambda_0 > 0$
depends only on the norm of the model rather than the stronger $(\kappa-2)$-H\"older norm of $\xi$, so that we can pass to the limit.
On the other hand,
we will take advantage of the smoothness of the given model $(\Pi, \Pi^-, \Gamma^*)$ in a qualitative way; see Definition~\ref{d:smooth}.

\medskip
\textbf{Consistency}.
One of the ingredients coming with the smoothness
of the model is a consistency statement, 
which asserts that to any solution $u$ of \eqref{t30}
one can associate a corresponding modelled distribution $f$ satisfying the robust formulation given in Definition~\ref{rf03}.
\begin{lemma}[Consistency]\label{l:c}
Let $(\Pi, \Pi^-, \Gamma^*)$ be a smooth model (as in Definition~\ref{d:smooth}).
Let 
$\lambda \in \mathbb{R}$
and $u$ be a $\kappa$-H\"older periodic function solving \eqref{t30}.
Then, there is a unique
modelled distribution
$f$ satisfying 
	\begin{align}
		f_x . \mathsf{z}_{\mathfrak{3}} & = \lambda , \label{t9} \\
			f_x . \mathsf{z}_{\mathbf{n}} & = \frac{1}{\mathbf{n}!} \Big ( \partial^{\mathbf{n}} u ( x ) - f_x . Q_{| \mathbf{n} |} \partial^{\mathbf{n}} \Pi_{x} ( x ) \Big ) ,  \quad \text{for all } | \mathbf{n} | < \kappa , \label{t10} \\
		f_x . \mathsf{z}_{\mathbf{n}} & = 0 \quad \text{for all } | \mathbf{n} | > \kappa . \label{t9b}
	\end{align}
Furthermore, $f$ satisfies the robust formulation of the equation \eqref{t30} as 
in Definition~\ref{rf03}.
\end{lemma}

We will prove Lemma~\ref{l:c} in Section~\ref{ss:lift} below.

\medskip
\textbf{A continuity method}.
We resume the discussion of our existence result, Theorem~\ref{existence}.
First, we note that up to choosing $\lambda_0$ small enough, by \eqref{t58}, we may (and will) assume that $M \geq 1$.
Our construction of solutions to \eqref{t30}
under the assumptions of Theorem~\ref{existence} 
is based on a continuity method in $\lambda$ in the spirit of bifurcation theory.
In particular, we simultaneously monitor
\eqref{t30} and its linearisation
	\begin{align*}
		\dot{u} \mapsto L \dot{u} - P \big( (3 \lambda u^2 + h_{\lambda} ) \dot{u} \big) .
	\end{align*}
More precisely, we will consider the following properties $P_1 ( \lambda, u )$ and $P_2 ( \lambda, u )$ of a real number $\lambda$ and a 
$\kappa$-H\"older
periodic function $u$ with vanishing space-time average: 
	\begin{align}
		\label{eq:P1} \tag{$P_1 ( \lambda , u)$}
		& L u = P ( \lambda u^3 + h_{\lambda} u + \xi ) , \text{ and $\| f . \mathsf{z}_{\mathbf{0}} \| \leq 1$ (recall \eqref{t48})} . \\
		\label{eq:P2} \tag{$P_2 ( \lambda, u )$}
		& \text{The linearised operator} \\
		& \quad 
		\begin{array}[t]{lrcl}
		\dot{\Phi} ( \lambda, u ) : & C_{\mathrm{per}, 0}^{\kappa} & \longrightarrow & C_{\mathrm{per}, 0}^{\kappa - 2} \\
    & \dot{u} & \longmapsto & L \dot{u} - P \big( ( 3 \lambda u^2 + h_{\lambda} ) \dot{u} \big) ,
  		\end{array} \nonumber \\
  		&\text{is continuously invertible.} \nonumber
	\end{align}

Here, we denote by $C_{\mathrm{per}, 0}^{\eta}$ the space of $\eta$-H\"older  functions which are $\ell$-periodic with vanishing space-time average.
In particular, the standard (parabolic) H\"older seminorm
$[\cdot]_{\eta}$  defined in \eqref{b15},
is now a norm for which this space is complete.
For $\lambda_0 >0$ to be adjusted later, we define a subset of $[- \lambda_0, \lambda_0]$ by
\begin{align*}
		\Lambda 
		\coloneqq \lbrace \lambda \in [- \lambda_0, \lambda_0] , \text{there is $u$ satisfying } \eqref{eq:P1} \text{ and } \eqref{eq:P2} \rbrace .
	\end{align*}
We aim to prove that so long as $\lambda_0$ was suitably chosen, we have that
$\Lambda = [- \lambda_0, \lambda_0]$, see \eqref{t31} in Theorem~\ref{existence}.
To that effect, it suffices to prove that
$\Lambda$ contains $0$, that it is closed, and open
(relative to $[- \lambda_0, \lambda_0]$).

\medskip
Let us justify in a few words how those properties are obtained.
The rigorous proofs will be presented in Section~\ref{ss:exist} below.
When $\lambda = 0$, 
note that both \eqref{eq:P1} and \eqref{eq:P2} follow (with $u = L^{-1} \xi = \Pi_{\beta = 0}$) from 
the
(standard)
fact that $L$ is continuously invertible from $C_{\mathrm{per}, 0}^{\kappa}$ to $C_{\mathrm{per}, 0}^{\kappa-2}$.

\medskip
The openness will follow from more classical arguments in the standard H\"older topologies.
More precisely,
for $\lambda \in \Lambda$, 
the combination of \eqref{eq:P2} and the Implicit Function Theorem imply that $(P_1)$ remains valid in a neighbourhood of $\lambda$. Furthermore, $(P_2)$ also remains valid in a neighbourhood of $\lambda$ as will follow from a perturbation argument.
It is worth mentioning that
we have no quantitative control on the size of the corresponding neighbourhood in terms of $M$.
Indeed, such a control is not to be expected. This will not be problematic for our argument.

\medskip
The argument for the closedness of $\Lambda$ dictates the choice of $\lambda_0$ and relies on a priori estimates (in the less classical topologies of modelled distributions) for both \eqref{t30} and for $\dot{\Phi} ( \lambda, u )^{-1}$, 
the latter being reformulated as the linearised PDE with right-hand side
	\begin{align}\label{t7}
		L \dot{u} = P \dot{u}^- + \zeta , 
		\quad \dot{u}^- \coloneqq (3 \lambda u^2 + h_{\lambda} ) \dot{u},
		\quad \zeta \in C_{\mathrm{per}; 0}^{\kappa - 2} ,
		\quad \fint \dot{u} = 0 .
	\end{align}
	
\medskip
\textbf{Linearisation}.
The a priori estimates on \eqref{t30} have been discussed above in Subsection~\ref{ss:apriori}.
Thus,
the remaining ingredient towards our proof of Theorem~\eqref{existence} is
a suitable a priori estimate for the linearisation \eqref{t7} of the equation \eqref{t30}.
They will be obtained by yet another loop of Algebraic Continuity Lemma, Reconstruction, Integration, Three-Point Argument, and an
incorporation of boundary data, similar to the loop presented in Subsections~\ref{ss:apriori} and \ref{ss:cont} above.
Indeed, it turns out that \eqref{t7} also comes with a corresponding robust formulation based on a notion of linearised modelled distributions, where multiplicativity \eqref{t8} is replaced by Leibniz’ rule 
\eqref{t8l}.

\medskip
Let us motivate our definition of a linearised modelled distribution by a geometric heuristic.
Consider a curve $\mathbb{R} \ni t \mapsto f ( t )$ of modelled distributions as defined in Definition~\ref{rf01} with 
 $f_x ( t ) . \mathsf{z}_{\mathfrak{3}} \equiv \lambda$.
Suppose also that the curve passes through a given modelled distribution $f$ at $t = 0$.
We think of the linearised modelled distribution $\dot{f} = \frac{d}{dt}|_{t = 0} f ( t )$ as the corresponding `tangent vector’ at $f$ in the (nonlinear) space of modelled distributions.
Differentiating the identity \eqref{rf02} at $t = 0$, we learn 
that $\dot{f}.\pi$
is explicitly given
for $\pi \in \mathbb{R} [\mathsf{z}_{\mathfrak{3}}, \mathsf{z}_{\mathbf{n}}]$ by
	\begin{align}
		\dot{f}_x.\pi
		& = \sum_{\beta} \pi_{\beta} \lambda^{\beta ( \mathfrak{3} )} \sum_{| \mathbf{n} | < \kappa} \beta ( \mathbf{n} ) \, (\dot{f}_x.\mathsf{z}_{\mathbf{n}}) \, (f_x . \mathsf{z}_{\mathbf{n}})^{\beta ( \mathbf{n}) - 1} \, \prod_{\substack{{|\mathbf{m}|<\kappa} \\ \mathbf{m} \neq \mathbf{n}}} (f_x.\mathsf{z}_{\mathbf{m}})^{\beta ( \mathbf{m})} . \nonumber
	\end{align}
Defining the derivative
$\partial_{\mathsf{z}_{\mathbf{n}}} \colon \mathbb{R} [[ \mathsf{z}_{\mathfrak{3}}, \mathsf{z}_{\mathbf{n}} ]] \to \mathbb{R} [[ \mathsf{z}_{\mathfrak{3}}, \mathsf{z}_{\mathbf{n}} ]]$
to be the linear extension of
	\begin{align}
		\partial_{\mathsf{z}_{\mathbf{n}}}\mathsf{z}^{\beta} \coloneqq \beta ( \mathbf{n} ) \mathsf{z}_{\mathbf{n}}^{\beta ( \mathbf{n} ) - 1} \lambda^{\beta ( \mathfrak{3} )} \prod_{\mathbf{m} \neq \mathbf{n}} \mathsf{z}_{\mathbf{m}}^{\beta ( \mathbf{m} )} , \nonumber
	\end{align}
this may be reformulated as $\dot{f}_x . \pi = \sum_{\beta} \pi_{\beta} \sum_{| \mathbf{n} | < \kappa} (\dot{f}_x . \mathsf{z}_{\mathbf{n}}) \, f_x . \partial_{\mathsf{z}_{\mathbf{n}}} \mathsf{z}^{\beta}$.
We now simply record the output of this heuristic argument as the algebraic constraint in our definition of a linearised modelled distribution.

\begin{definition}\label{rf01l}
Let $f$ be a modelled distribution.
A \emph{linearised modelled distribution} $\dot{f}$ at $f$ is an 
	\begin{align*}
		\dot{f} \colon \mathbb{R}^{1 + d} \to ( \mathbb{R} [\mathsf{z}_{\mathfrak{3}}, \mathsf{z}_{\mathbf{n}}] )^* ,
	\end{align*}
of the form
	\begin{align}
		\dot{f}_x . \pi
		& = \sum_{\beta} \pi_{\beta} \sum_{| \mathbf{n} | < \kappa} (\dot{f}_x . \mathsf{z}_{\mathbf{n}}) \, f_x . \partial_{\mathsf{z}_{\mathbf{n}}} \mathsf{z}^{\beta} . \label{t49}
	\end{align}
	
Equivalently,
for all $x \in \mathbb{R}^{1 + d}$, $\pi, \tilde{\pi} \in \mathbb{R} [\mathsf{z}_{\mathfrak{3}}, \mathsf{z}_{\mathbf{n}}]$, one has 
			\begin{align}
				\dot{f}_x . \pi \tilde{\pi} & = (\dot{f}_x . \pi) ( f_x . \tilde{\pi} ) + (f_x . \pi) ( \dot{f}_x . \tilde{\pi} ) , \label{t8l} 
				\\
				\dot{f}_x . \mathsf{z}_{\mathfrak{3}} &= 0 , \qquad \dot{f}_x . \mathsf{z}_{\mathbf{n}} = 0 \quad \text{for {$| \mathbf{n} | > \kappa$}, $\mathbf{n} \in \mathbb{N}_0^{1 + d}$.} \nonumber
			\end{align}	
	
Furthermore, we assume that
$\| \dot{f} \|_{\kappa} < \infty$, 
where $\| \cdot \|_{\kappa}$ is as in \eqref{md01}.
\end{definition}

\begin{definition}\label{rf03l}
We will say that a linearised modelled distribution $\dot{f}$ at $f$ satisfies the \emph{robust formulation} of the linearised equation \eqref{t7} (for the model $(\Pi, \Pi^-, \Gamma^*)$ and $\zeta \in C_{\mathrm{per}, 0}^{\kappa - 2}$) if
	\begin{enumerate}
		\item the map $x \mapsto \dot{f}_x$ is periodic,
		\item there exist periodic (Schwartz) distributions $\dot{u}, \dot{u}^-$ such that $\fint \dot{u} = 0$,
			\begin{align}\label{rfPDEl}
				L \dot{u} = P \dot{u}^- + \zeta , 
			\end{align}
			and the germs 
			\begin{align}\label{t12l}
				\dot{R}_x \coloneqq \dot{u} - \dot{f}_x . Q_{\kappa} \Pi_x, \quad \text{and}\quad \dot{R}_x^- \coloneqq \dot{u}^- - \dot{f}_x . Q_{2 + \kappa^-} \Pi_x^{-} ,
			\end{align}
			satisfy
			\begin{align}\label{t13l}
				\dot{R}_{x \mu} ( x ) \to 0 , \qquad \dot{R}_{x \mu}^- ( x ) \to 0 ,
			\end{align}
		as $\mu \to 0$ 
		locally uniformly in $x$, where in \eqref{t13l} the convolution is with respect to the Schwartz function defined in \eqref{t90}.
	\end{enumerate}
\end{definition}

We note that arguing as in Remark~\ref{r:gDP} it follows that $\dot{u}$ is a function and 
	\begin{align}\label{i2}
		\dot{u} ( x ) = \dot{f}_x.\mathsf{z}_{\mathbf{0}} .
	\end{align}
	
Similarly to Lemma~\ref{l:c}, such an $\dot{f}$ may be related to actual solutions to \eqref{t7}
when the model is qualitatively smooth
by the following consistency lemma.
\begin{lemma}[Consistency II]\label{l:c2}
Let $(\Pi, \Pi^-, \Gamma^*)$ be 
a smooth model (as in Definition~\ref{d:smooth}).
Let $u$ be a 
periodic
$\kappa$-H\"older
solution of \eqref{t30}
with 
$\lambda \in \mathbb{R}$.
Let $\dot{u}$ be a periodic $\kappa$-H\"older solution of \eqref{t7}.
Let $f$ be as in Lemma~\ref{l:c}.
Then there exists a unique linearised modelled distribution $\dot{f}$ satisfying 
	\begin{align}
		\dot{f}_x . \mathsf{z}_{\mathfrak{3}} & = 0, \label{t50} \\
			\dot{f}_x . \mathsf{z}_{\mathbf{n}} & = \frac{1}{\mathbf{n}!} \Big ( \partial^{\mathbf{n}} \dot{u} ( x ) - \dot{f}_x . Q_{| \mathbf{n} |} \partial^{\mathbf{n}} \Pi_{x} ( x ) \Big ),  \quad \text{for all } | \mathbf{n} | < \kappa , \label{t51} \\
		\dot{f}_x . \mathsf{z}_{\mathbf{n}} & = 0 \quad \text{for all } | \mathbf{n} | > \kappa . \label{t50b}
	\end{align}

Furthermore, $\dot{f}$ satisfies the robust formulation of the linearised equation \eqref{t7}, 
at $f$
w.r.t.\
$\zeta \in C_{\mathrm{per} ,0}^{\kappa - 2}$, 
as defined in Definition~\ref{rf03l}.
\end{lemma}

We now state our a priori estimate for such an $\dot{f}$.
\begin{theorem}[A Priori Estimates for the linearisation]\label{a priori lin}
    For any $M < \infty$ there exists $\lambda_0 > 0$ such that if 
    $(\Pi, \Pi^-, \Gamma^*)$ is a model satisfying Assumption~\ref{ass1}, 
    $f$ satisfies the robust reformulation of Definition~\ref{rf03} with $|\lambda| \leq \lambda_0$,
    $\dot{f}$ satisfies the robust formulation of Definition~\ref{rf03l} at $f$
with respect to $\zeta \in C_{\mathrm{per},0}^{\kappa - 2}$, 
    and further
    \begin{align*}
        \|\Pi\|_{\kappa} + \|\Gamma^*\|_{\kappa} + \| f . \mathsf{z}_{\mathbf{0}} \| \le M
    \end{align*}
    then 
    \begin{align*}
		\| \dot{f} \|_{\kappa} 
		\lesssim [ \zeta ]_{\kappa - 2} ,
    \end{align*}
where the implicit constant depends on $M, \alpha, d, \kappa$ and $\ell$.
\end{theorem}

\medskip
In order to prove Theorem~\ref{a priori lin}, we again broadly follow the strategy outlined in Subsection~\ref{ss:apriori}.
The main difference lies in the Algebraic Continuity Lemma, which needs to be adapted.
This is due to the replacement of the multiplicativity \eqref{t8} by the Leibniz-type rule \eqref{t8l} and is analogous to the phenomenon already described in Subsection~\ref{ss:cont}; see Lemma~\ref{continuity_2}.
\begin{lemma}[Algebraic Continuity Lemma III] \label{continuity_3}
  Fix a model $(\Pi, \Pi^-, \Gamma^*)$, suppose that $f$ is a modelled distribution with $$|\lambda| + \| \Pi \|_{\kappa} + \|\Gamma^*\|_{\kappa} + \| f\|_\kappa \le M$$
  and that $\dot{f}$ is a linearised modelled distribution at $f$.
  Then for each populated and non-purely-polynomial multi-index $\beta$ with $|\beta| < \kappa$, 
    \begin{align}\label{md03-l} 
        [\dot{f}]_{\kappa, \beta} 
        \lesssim |\lambda| ([\dot{f}]_{\kappa, \mathrm{pol}} + \| \dot{f} . \mathsf{z}_{\mathbf{0}} \| ) 
    \end{align}
    where the implicit constant depends on $M, \alpha, d$ and $\kappa$. 
\end{lemma}

\section{A Priori Estimate for the Robust Formulation} \label{s:proof_a_priori}

\subsection{Proof of Algebraic Continuity Lemma I}
Our proof of the first Algebraic Continuity Lemma is split into two main parts respectively consisting of an interpolation estimate and a preliminary algebraic estimate.

\medskip
More precisely, in a first part, we will prove 
an
interpolation estimate
for modelled distributions.
It will be useful both 
in the proof of Lemma~\ref{continuity}
and also to later post-process our bounds on higher-order seminorms to the full form of the a priori estimate of Theorem~\ref{a priori}.
By an interpolation estimate,
we mean an estimate that controls a seminorm of $f$ by the geometric average
of two other seminorms. The norms involved are the H\"older-type seminorm
$[\cdot]_{\kappa,{\rm pol}}$, see \eqref{rop15}, 
and the supremum norm $\|\cdot\|$ of the Gubinelli derivatives 
$\{f.\mathsf{z}_{\bf n}\}_{|{\bf n}|<\kappa}$, see \eqref{t46}. 
These norms are characterized by a space-time scaling, 
which determine the exponents in the geometric average. We note that Lemma~\ref{l:intI} coincides with the standard interpolation estimates in H\"older spaces when $\lambda = 0$. Lemma~\ref{l:intI} will be established as a perturbation of its H\"older analogue by a continuity argument in $\lambda$.
While \eqref{wr01} assumes the form of a linear
interpolation estimate, the objects $f$ are elements of a nonlinear space in view of
(\ref{rf02}). This constraint is essential for restricting the r.~h.~s.~of (\ref{wr01})
to $[\cdot]_{\kappa,{\rm pol}}$ rather than 
$\sup_{|\beta|<\kappa} [\cdot]_{\kappa, \beta}$,
which in turn is crucial
for establishing the Algebraic Continuity Lemma \ref{continuity}. The price 
is an additional additive term $\|f.\mathsf{z}_{\bf 0}\|$ 
in the r.~h.~s.~of (\ref{wr01}), which is the lowest-order norm in terms of scaling.
It is crucial for the buckling argument leading to the a priori estimate of
Theorem \ref{a priori} in the regime $|\lambda|\ll 1$
that this term is multiplied with a positive power of $|\lambda|$.

\begin{lemma}[Interpolation I]\label{l:intI}
 Suppose that $f$ is a modelled distribution as in Definition~\ref{rf01} with respect to a model $(\Pi, \Pi^-, \Gamma^*)$.
We assume that
\begin{align}\label{wr07}
|\lambda| + \| \Gamma^* \|_{\kappa} + \|f.\mathsf{z}_{\bf 0}\| \le M
\end{align}
for some $M<\infty$. Then we have
\begin{align}
\|f.\mathsf{z}_{\bf n}\|&\lesssim
([f]_{\kappa,{\rm pol}}+|\lambda|)^\frac{|{\bf n}|}{\kappa}
(\|f.\mathsf{z}_{\bf 0}\|+|\lambda|)^{1-\frac{|{\bf n}|}{\kappa}}\quad
\mbox{for}\;|{\bf n}|<\kappa,\label{wr02} 
\\
[f]_{\tilde{\kappa},{\rm pol}}&\lesssim
([f]_{\kappa,{\rm pol}}+|\lambda|)^\frac{\tilde{\kappa}}{\kappa}
(\|f.\mathsf{z}_{\bf 0}\|+|\lambda|)^{1-\frac{\tilde{\kappa}}{\kappa}}\quad
\mbox{for}\;\tilde{\kappa}\le\kappa , \label{wr01} 
\end{align}
where the implicit constant depends on $M, \alpha, d$ and $\kappa$.
\end{lemma}

In a second part, we will prove 
the following
preliminary
algebraic estimate.
\begin{lemma}\label{l:pcI}
Suppose that $f$ is a modelled distribution in the sense of Definition~\ref{rf01} with respect to a model $(\Pi, \Pi^-, \Gamma^*)$.
We assume that
\begin{align}\label{wr44}
|\lambda| + \| \Gamma^* \|_{\kappa} \le M
\end{align}
for some $M<\infty$. 
Then
for all $\beta$ and $\eta$ with $|\beta|<\eta\le\kappa$ we have
\begin{align}\label{wr40}
[f]_{\eta,\beta}\lesssim|\lambda|^{\beta ( \mathfrak{3} )}
\max_{\tilde\eta+\sum_{\bf n}|\mathbf{n}|\tilde\beta({\bf n})\le\eta}[f]_{\tilde\eta,{\rm pol}}
\prod_{\bf n}\|f.\mathsf{z}_{\bf n}\|^{\tilde\beta({\bf n})},
\end{align}
where the max runs over all $\tilde\eta> 0$ of the form $\tilde{\eta} = \eta - ( | \gamma | - \alpha )$ and all multi-indices
$\tilde\beta$, 
and where the implicit constant depends on $M, \alpha, d$ and $\kappa$.
\end{lemma}

We now demonstrate how Lemmas~\ref{l:intI} and \ref{l:pcI} imply 
the Algebraic Continuity Lemma \ref{continuity}.
\begin{proof}[Proof of Lemma~\ref{continuity}]
	Inserting the interpolation estimates \eqref{wr02} and \eqref{wr01} into \eqref{wr40} yields
	\begin{align*}
		[ f ]_{\kappa, \beta}
		\lesssim|\lambda|^{\beta ( \mathfrak{3} )}
\max_{\tilde\kappa+\sum_{\bf n}|\mathbf{n}|\tilde\beta({\bf n})\le\kappa}
	( [f]_{\kappa, \mathrm{pol}} + |\lambda| )^{a_1 ( \tilde{\kappa}, \tilde{\beta} )} \, 
			( \| f . \mathsf{z}_{\mathbf{0}} \| + |\lambda| )^{a_2( \tilde{\kappa}, \tilde{\beta} )} ,
	\end{align*}
for the exponents
	\begin{align*}
		a_1 \coloneqq \frac{\tilde{\kappa} + \sum_{\mathbf{n}} |\mathbf{n}| \tilde{\beta} ( \mathbf{n} )}{\kappa} \in [0, 1] ,
		\qquad a_2 \coloneqq (1 - a_1) + \sum_{\mathbf{n}} \tilde{\beta} ( \mathbf{n} ) .
	\end{align*}
Since by assumption one has
$| \lambda | + \| f . \mathsf{z}_{\mathbf{0}} \| \leq M$, 
we may absorb the term 
$( \| f . \mathsf{z}_{\mathbf{0}} \| + |\lambda| )^{\sum_{\mathbf{n}} \tilde{\beta} ( \mathbf{n} )}$ into the implicit prefactor.
The result then follows from Young's inequality.
\end{proof}

We turn to the proof of Lemma~\ref{l:intI}.
\begin{proof}[Proof of Lemma~\ref{l:intI}]
We start from the definition~\ref{rop15}
for ${\bf n}={\bf 0}$.
In 
	\begin{align*}
		f_x.Q_\kappa\Gamma_{x\,x+h}^*\mathsf{z}_{0}
		=\sum_{0\le|\beta|<\kappa}f_x.\mathsf{z}^\beta
(\Gamma_{x\,x+h}^*)_{\beta}^{\delta_{\bf 0}}
	\end{align*}
we distinguish purely polynomial $\beta$,
on which we use \eqref{mb07}, and the other 
$\beta$, which 
are populated by \eqref{mb09}
thus satisfy 
$\beta ( \mathfrak{3} )\not=0$ and on which we use
the model bound
(\ref{mb03}). 
This yields
\begin{align}\label{wr21}
\big|\sum_{|{\bf n}|<\kappa}f_x.\mathsf{z}_{\bf n}h^{\bf n}\big|
\le[f]_{\kappa,{\rm pol}}|h|^\kappa+\|f.\mathsf{z}_{\bf 0}\|
+\|\Gamma^*\|\sum_{\stackrel{0\le|\beta|<\kappa}{\beta ( \mathfrak{3} )\not=0}}
\|f.\mathsf{z}^\beta\||h|^{|\beta|}.
\end{align}
Given a scale $r$ we have by equivalence of norms on the finite-dimensional 
space of polynomials in $h$ of (parabolic) degree less than $\kappa$ and a scaling argument
\begin{align}\label{wr22}
\max_{|{\bf n}|<\kappa}|f_x.\mathsf{z}_{\bf n}|r^{|{\bf n}|}
\lesssim\sup_{r\leq|h|\leq 2r}
\big|\sum_{|{\bf n}|<\kappa}f_x.\mathsf{z}_{\bf n}h^{\bf n}\big|.
\end{align}
Hence taking the supremum over $x$ yields
\begin{align}\label{wr08bis}
\max_{0<|{\bf n}|<\kappa}\|f.\mathsf{z}_{\bf n}\|r^{|{\bf n}|}
\lesssim[f]_{\kappa,{\rm pol}}r^\kappa
+\|f.\mathsf{z}_{\bf 0}\|+\max_{\stackrel{0\le|\beta|<\kappa}{\beta ( \mathfrak{3} )\not=0}}
\|f.\mathsf{z}^\beta\|r^{|\beta|}.
\end{align}
According to (\ref{rf02}) we have
\begin{align}\label{wr08}
\|f.\mathsf{z}^\beta\|\le|\lambda|^{\beta ( \mathfrak{3} )}
\prod_{\bf n}\|f.\mathsf{z}_{\bf n}\|^{\beta({\bf n})},
\end{align}
so that by (\ref{wr07}) for $\beta$ with $\beta ( \mathfrak{3} )\not=0$
\begin{align*}
\|f.\mathsf{z}^\beta\|r^{|\beta|}
\lesssim|\lambda|
\big(\max_{0<|{\bf n}|<\kappa}\|f.\mathsf{z}_{\bf n}\|r^{|{\bf n}|}
\big)^{\sum_{\bf n\not={\bf 0}}\beta({\bf n})}r^{|\beta|-\sum_{\bf n}|{\bf n}|\beta({\bf n})}.
\end{align*}
Note that for $\beta$ with $|\beta|\ge 0$ and $\beta ( \mathfrak{3} )\not=0$
we have $\sum_{{\bf n}\not={\bf 0}}\beta({\bf n})$ 
$\le\sum_{\bf n}|{\bf n}|\beta({\bf n})$ $<|\beta|$ so that
\begin{align}\label{wr09}
\max_{\stackrel{0\le|\beta|<\kappa}{\beta ( \mathfrak{3} )\not=0}}\|f.\mathsf{z}^\beta\|r^{|\beta|}
\lesssim|\lambda|
\big(\max_{0<|{\bf n}|<\kappa}\|f.\mathsf{z}_{\bf n}\|r^{|{\bf n}|}
+1\big)^{\kappa}\max\{r^{\kappa},r^{\delta}\}
\end{align}
where $\delta:=\min\{\,|\beta|-\lfloor|\beta|\rfloor\,
|\,\mathbb{Z}\not\ni|\beta|<\kappa\,\}$.

\medskip
Momentarily introducing the abbreviations
$A:=|\lambda|\max\{r^\kappa,r^\delta\}$, 
$B:=[f]_{\kappa,{\rm pol}}r^\kappa
+\|f.\mathsf{z}_{\bf 0}\|$ and 
$F:=C^{-1}\max_{0<|{\bf n}|<\kappa}\|f.\mathsf{z}_{\bf n}\|r^{|{\bf n}|}$
for some large constant $C\lesssim_M1$,
we combine (\ref{wr08bis}) and (\ref{wr09}) to
\begin{align}\label{wr10}
F<A(F+1)^\kappa+B.
\end{align}
We then note that for fixed real numbers $B \le M + 1$ and $A \le (M+3)^{-\kappa}$, the set of real numbers $F$ that
satisfy (\ref{wr10}) is disconnected from infinity. Indeed, setting $F = (M+3)^\kappa A + B$ yields $F \ge A(F+1)^\kappa + B$.
As a consequence, the connected component containing
$F=0$ satisfies $F\lesssim_M A+B$.

\medskip
 Since $F$, $A$, and $B$ as defined above are continuous 
and monotone functions of $r$ with $F_{r=0}=0$, we infer from a simple continuity argument that
\begin{align}\label{wr12}
F\lesssim A+B\quad\mbox{provided that}\quad A\le (M+3)^{-\kappa}\;\mbox{and}\;B\le M+1.
\end{align}
Since $\delta>0$ and in view of the first item in \eqref{wr07}, by writing $r^\delta = (r^\kappa)^{\frac{\delta}{\kappa}}$ we have 
$A \le (M+3)^{-\kappa}$
if $|\lambda|r^\kappa \ll_M 1$ by which we mean that the l.h.s. is smaller than a small $M$ dependent constant whose precise value is not important here.
Hence 
in view of the second item in (\ref{wr07}), (\ref{wr12}) translates back into
\begin{align*}
&\|f.\mathsf{z}_{\bf n}\|r^{|{\bf n}|}
\lesssim [f]_{\kappa,{\rm pol}}r^\kappa+\|f.\mathsf{z}_{\bf 0}\|
+|\lambda|\max\{r^\kappa,r^\delta\}\quad
\mbox{for}\;0<|{\bf n}|<\kappa\\
&\mbox{provided}\quad 
[f]_{\kappa,{\rm pol}}r^\kappa\le 1\quad\mbox{and}\quad|\lambda|r^\kappa\ll 1,
\end{align*}
which by $\max\{r^\kappa,r^\delta\}$ $\le r^\kappa+1$ consolidates to
\begin{align*}
&\|f.\mathsf{z}_{\bf n}\|r^{|{\bf n}|}
\lesssim([f]_{\kappa,{\rm pol}}+|\lambda|)r^\kappa+(\|f.\mathsf{z}_{\bf 0}\|+|\lambda|)
\quad\mbox{for}\;0<|{\bf n}|<\kappa\\
&\mbox{provided}\quad
([f]_{\kappa,{\rm pol}}+|\lambda|)r^\kappa\ll 1.
\end{align*}
In order to establish (\ref{wr02}),
it remains to choose $r$ to be an $M$-dependent small fraction of 
$(\frac{\|f.\mathsf{z}_{\bf 0}\|+|\lambda|}{[f]_{\kappa,{\rm pol}}+|\lambda|})^\frac{1}{\kappa}$,
which is admissible because of (\ref{wr07}).

\medskip

In preparation for (\ref{wr01}), we post-process (\ref{wr02}) to
\begin{align}\label{wr26}
\|f.\mathsf{z}^\beta\|
\lesssim ([f]_{\kappa,{\rm pol}}+|\lambda|)^{\frac{|\beta|}{\kappa}}
(\|f.\mathsf{z}_{\bf 0}\|+|\lambda|)^{1-\frac{|\beta|}{\kappa}}
\end{align}
for populated $\beta$ with $0\le|\beta|<\kappa$.
Indeed, inserting (\ref{wr02}) into (\ref{wr08}) we obtain
\begin{align}\label{wr27}
\lefteqn{\|f.\mathsf{z}^\beta\|}\nonumber\\
&\lesssim|\lambda|^{\beta ( \mathfrak{3} )}
([f]_{\kappa,{\rm pol}}+|\lambda|)^{\frac{\sum_{\bf n}|{\bf n}|\beta({\bf n})}{\kappa}}
(\|f.\mathsf{z}_{\bf 0}\|+|\lambda|)^{\sum_{\bf n}\beta({\bf n})
-\frac{\sum_{\bf n}|{\bf n}|\beta({\bf n})}{\kappa}}.
\end{align}
In order to pass from (\ref{wr27}) to (\ref{wr26}) we distinguish the cases $\beta ( \mathfrak{3} )=0$
and $\beta ( \mathfrak{3} )$ $\not=0$. In the first case of $\beta ( \mathfrak{3} )=0$, we note that since $|\beta| \ge 0$
we have in particular $\beta\not=0$. Therefore since $\beta$ is populated
we must have $\sum_{\bf n}\beta({\bf n})=1$. Then also $\sum_{\bf n}|{\bf n}|\beta({\bf n})$ 
$=|\beta|$, so that (\ref{wr27}) is identical to (\ref{wr26}). 
In the second case where $\beta ( \mathfrak{3} )\ge 1$
we appeal to the assumption (\ref{wr07}) to estimate the first factor in (\ref{wr27}) 
by $|\lambda|^{1-\frac{\sum_{\bf n}|{\bf n}|\beta({\bf n})}{\kappa}}$ and
the last factor by $1$.
We then appeal to the inequality $\sum_{\bf n}|{\bf n}|\beta({\bf n})$ $\le|\beta|$,
which holds for any multi-index $\beta$ with $|\beta|\ge 0$, to pass to (\ref{wr26}).

\medskip

We now turn to (\ref{wr01}) proper; following definition (\ref{rop15}) we consider
an ${\bf n}$ with $|{\bf n}|<\tilde{\kappa}$. We distinguish the ranges $|h|\ge r$ and
$|h|\le r$ for some $r$ to be optimised later.
Using (\ref{mb03}) and (\ref{mb05}) we have
in the large-increment range where $|h|\ge r$
\begin{align*}
|(f_{x+h}.-f_x.Q_{\tilde{\kappa}}\Gamma_{x\,x+h}^*)\mathsf{z}_{\bf n}|
&\le\|f.\mathsf{z}_{\bf n}\|
+\|\Gamma^*\|_{\tilde{\kappa}}\sum_{\stackrel{\beta\;\mbox{\tiny populated}}
{|{\bf n}|\le|\beta|<\tilde{\kappa}}}
\|f.\mathsf{z}^\beta\|
|h|^{|\beta|-|{\bf n}|}\\
&\lesssim |h|^{\tilde{\kappa}-|{\bf n}|}
\max_{\stackrel{\beta\;\mbox{\tiny populated}}{0\le|\beta|<\tilde{\kappa}}}
r^{|\beta|-\tilde{\kappa}}\|f.\mathsf{z}^\beta\|,
\end{align*}
and we have in the small-increment range $|h|\le r$
\begin{align*}
\lefteqn{|(f_{x+h}.-f_x.Q_{\tilde{\kappa}}\Gamma_{x\,x+h}^*)\mathsf{z}_{\bf n}|}\nonumber\\
&\le[f]_{\kappa,{\rm pol}}|h|^{\kappa-|{\bf n}|}
+\|\Gamma^*\|_{\kappa}\sum_{\stackrel{\beta\;\mbox{\tiny populated}}
{\tilde{\kappa}\le|\beta|<\kappa}}\|f.\mathsf{z}^\beta\|
|h|^{|\beta|-|{\bf n}|}\\
&\lesssim |h|^{\tilde{\kappa}-|{\bf n}|}\big(r^{\kappa-\tilde{\kappa}}[f]_{\kappa,{\rm pol}}
+\max_{\stackrel{\beta\;\mbox{\tiny populated}}{\tilde{\kappa}\le|\beta|<\kappa}}
r^{|\beta|-\tilde{\kappa}}\|f.\mathsf{z}^\beta\|\big).
\end{align*}
Combining both gives 
\begin{align}\label{wr30}
[f]_{\tilde{\kappa},{\rm pol}}
&\lesssim r^{\kappa-\tilde{\kappa}}[f]_{\kappa,{\rm pol}}
+\max_{\stackrel{\beta\;\mbox{\tiny populated}}{0\le|\beta|<\kappa}}
\|f.\mathsf{z}^\beta\|r^{|\beta|-\tilde{\kappa}},
\end{align}
into which we insert (\ref{wr26}) to obtain
\begin{align*}
[f]_{\tilde{\kappa},{\rm pol}}
&\lesssim r^{\kappa-\tilde{\kappa}}([f]_{\kappa,{\rm pol}}+|\lambda|)\nonumber\\
&\quad+\max_{|\beta|<\kappa}r^{|\beta|-\tilde{\kappa}}
([f]_{\kappa,{\rm pol}}+|\lambda|)^{\frac{|\beta|}{\kappa}}
(\|f.\mathsf{z}_{\bf 0}\|+|\lambda|)^{1-\frac{|\beta|}{\kappa}}.
\end{align*}
With the choice of 
\begin{align*}
r=\Big(\frac{\|f.\mathsf{z}_{\bf 0}\|+|\lambda|}
{[f]_{\kappa,{\rm pol}}+|\lambda|}\Big)^\frac{1}{\kappa}
\end{align*}
this yields (\ref{wr01}).
\end{proof}

We turn to the proof of Lemma~\ref{l:pcI}.
\begin{proof}[Proof of Lemma~\ref{l:pcI}]
We establish (\ref{wr40}) by induction
in the `plain length' $\beta(\mathfrak{3})+\sum_{{\bf n}}\beta({\bf n})$.
To this purpose, we derive two formulas that relate the application of the increment
$f_{x+h}.-f_x.Q_\cdot\Gamma_{x\,x+h}^*$ to
$\mathsf{z}^{\beta+\delta_{\mathfrak{3}}}$ and $\mathsf{z}^{\beta+\delta_{\bf n}}$
to its application to $\mathsf{z}^{\beta}$ and $\mathsf{z}_{\bf n}$.
These formulas rely on the multiplicativity of $\Gamma^*$, 
see the first item in (\ref{mb07}), and of $f$, see (\ref{t8}).
They also rely on the commutation relation between the truncation operator
(\ref{nr11}) and the multiplication with a monomial
\begin{align}\label{wr43}
Q_\eta\mathsf{z}^\gamma=\mathsf{z}^\gamma Q_{\eta-(|\gamma|-\alpha)},
\end{align}
which is a consequence of the additivity of $|\cdot|-\alpha$, cf.~(\ref{t45}).
The formulas are given by
\begin{align}
(f_{x+h}.-f_x.Q_\eta\Gamma_{x\,x+h}^*)\mathsf{z}^{\beta+\delta_{\mathfrak{3}}}
&=(f_{x+h}.-f_x.Q_{\eta-(|\delta_{\mathfrak{3}}|-\alpha)}\Gamma_{x\,x+h}^*)\mathsf{z}^{\beta}\,
\lambda,\label{wr41}
\end{align}
and
\begin{align}\label{wr42}
\lefteqn{(f_{x+h}.-f_x.Q_\eta\Gamma_{x\,x+h}^*)\mathsf{z}^{\beta+\delta_{\bf n}}}\nonumber\\
&=(f_{x+h}.\mathsf{z}^{\beta})
(f_{x+h}.-f_x.Q_{\eta-(|\beta|-\alpha)}\Gamma_{x\,x+h}^*)\mathsf{z}_{\bf n}\nonumber\\
&\quad+\sum_{|\gamma|<\eta-(|\beta|-\alpha)}
(f_{x+h}.-f_x.Q_{\eta-(|\gamma|-\alpha)}\Gamma_{x\,x+h}^*)\mathsf{z}^\beta
(\Gamma_{x\,x+h}^*)_\gamma^{\delta_{\bf n}}f_x.\mathsf{z}^\gamma.
\end{align}

\medskip

Identity (\ref{wr41}) follows from the trivial action of $\Gamma^*$ and $f$
on $\mathsf{z}_{\mathfrak{3}}$, see the middle item in (\ref{mb07}) and (\ref{t8b}),
and on (\ref{wr43}) with $\beta=\delta_{\mathfrak{3}}$.
Identity (\ref{wr42}) follows from the operator identity
\begin{align*}
Q_\eta\Gamma_{x\,x+h}^*\mathsf{z}_{\bf n}
\stackrel{(\ref{wr43})}{=}\sum_{\gamma}(\Gamma_{x\,x+h}^*)_\gamma^{\delta_{\bf n}}
\mathsf{z}^\gamma Q_{\eta-(|\gamma|-\alpha)},
\end{align*}
which we apply to $\Gamma_{x\,x+h}^*\mathsf{z}^\beta$. By the triangularity (\ref{mb05})
of $\Gamma^*$, the result vanishes unless $|\beta|<\eta-(|\gamma|-\alpha)$.

\medskip

We recall the definition (\ref{rop15}) of $[\cdot]_{\eta,\beta}$ and note
that (\ref{wr41}) yields
\begin{align}\label{wr45}
[f]_{\eta,\beta+\delta_{\mathfrak{3}}}\le|\lambda|[f]_{\eta-(|\delta_{\mathfrak{3}}|-\alpha),\beta}.
\end{align}
Evoking also (\ref{mb03}), (\ref{wr42}) yields the inequality
\begin{align*}
\lefteqn{[f]_{\eta,\beta+\delta_{\bf n}}}\nonumber\\
&\le\|f.\mathsf{z}^{\beta}\|[f]_{\eta-(|\beta|-\alpha),\delta_{\bf n}}
+\|\Gamma^*\|_{\eta}\sum_{|\gamma|<\eta-(|\beta|-\alpha)}
[f]_{\eta-(|\gamma|-\alpha),\beta}\|f.\mathsf{z}^\gamma\|.
\end{align*}
This in turn implies that
\begin{align}\label{wr46}
[f]_{\eta,\beta+\delta_{\bf n}}
&\lesssim \|f.\mathsf{z}^{\beta}\|[f]_{\eta-(|\beta|-\alpha),{\rm pol}}
+\max_{|\gamma|<\eta-(|\beta|-\alpha)}
[f]_{\eta-(|\gamma|-\alpha),\beta}\|f.\mathsf{z}^\gamma\|\nonumber\\
&\stackrel{(\ref{wr08})}{\lesssim} 
|\lambda|^{\beta ( \mathfrak{3} )}[f]_{\eta-(|\beta|-\alpha),{\rm pol}}
\prod_{\bf m}\|f.\mathsf{z}_{\bf m}\|^{\beta({\bf m})}\nonumber\\
&\quad+\max_{|\gamma|<\eta-(|\beta|-\alpha)}
[f]_{\eta-(|\gamma|-\alpha),\beta}\prod_{\bf m}\|f.\mathsf{z}_{\bf m}\|^{\gamma({\bf m})}.
\end{align}

\medskip

Equipped with (\ref{wr45}) and (\ref{wr46}), we finally embark on the
induction argument. The base case $\beta=0$ is trivial since by
footnotes~\ref{fn:1} and \ref{fn:2},
the l.~h.~s.~of (\ref{wr40}) vanishes. Turning to the induction step,
we note that (\ref{wr40}) is obviously preserved under passage from $\beta$ to $\beta + \delta_\mathfrak{3}$.
We now consider (\ref{wr46}), where we shall appeal to the following
easy consequence of (\ref{t45})
\begin{align}\label{wr47}
\sum_{\bf m}|{\bf m}|\gamma({\bf m})\le|\gamma|-\alpha.
\end{align}
We start by noting that the first r.~h.~s.~term in (\ref{wr46}) is contained 
in the r.~h.~s.~of (\ref{wr40}) by (\ref{wr47}) (with $\gamma$ replaced by $\beta$).
Inserting the induction hypothesis (\ref{wr40}) into the second r.~h.~s.~term of (\ref{wr46})
we see that the latter is estimated by 
\begin{align*}
|\lambda|^{\beta ( \mathfrak{3} )}\max_{\stackrel{|\gamma|<\eta-(|\beta|-\alpha)}
{\tilde\eta+\sum_{\bf m}|{\bf m}|\tilde\beta({\bf m})\le
\eta-(|\gamma|-\alpha)}}[f]_{\tilde\eta,{\rm pol}}
\prod_{\bf m}\|f.\mathsf{z}_{\bf m}\|^{(\tilde\beta+\gamma)({\bf m})};
\end{align*}
It remains to note that by (\ref{wr47}) the multi-index $\tilde\beta+\gamma$
satisfies the desired 
constraint $\tilde\eta+\sum_{\bf m}|{\bf m}|(\tilde\beta+\gamma)({\bf m})\le\eta$.
\end{proof}

\subsection{Proof of (Pointed) Schauder Estimate}

As in \cite{LOTT, BOT}, 
our Schauder theory is based on the semi-group $(\Psi_t)_{t>0}$ of Schwartz kernels generated by the symmetric positive operator $LL^*$, i.e.\
	\begin{align}
		\Psi_t \coloneqq \mathcal{F}^{-1} \big( q \mapsto \exp ( - t | q |^4) \big) . \label{t92}
	\end{align}
Given a Schwartz distribution $F$,
we denote $F_t \coloneqq F * \Psi_t$.
We emphasise that this is not quite compatible with the notation \eqref{t87} because $\Psi_t$ scales in $\sqrt[4]{t}$, see \eqref{t73} below, however this notation is useful in view of the properties of $\Psi$, see \eqref{t89} and \eqref{t89b} below.
In what follows,
the notation $F_{\mu}$ with the letter $\mu$ is reserved for the convolution \eqref{t87} with respect to an implicit Schwartz function $\psi$, 
while the notation $F_{t}$ with the letter $t$ is reserved for the convolution with respect to the semi-group $\Psi$.

\medskip
In the following proof, we will rely on the following properties of the semi-group $\Psi$ defined in \eqref{t92}, which are readily obtained on the Fourier level:
	\begin{align}
		& \Psi_{t} * \Psi_T = \Psi_{t + T} , \label{t89} \\ 
		& \partial_t \Psi_t + L L^* \Psi_t = 0 , \label{t89b} \\
		& \Psi_t ( x ) = (\sqrt[4]{t})^{- D} \Psi_{t = 1} \big( (\sqrt[4]{t})^{-2} x_0, (\sqrt[4]{t})^{-1} x_1 , \cdots , (\sqrt[4]{t})^{-1} x_d \big) . \label{t73}
	\end{align}

\begin{proof}[Proof of Lemma~\ref{Schauder}] 
	We fix $v$ and $v^-$ as in the statement. We aim to show that
	\begin{align}\label{int_rep}
		v = \int_0^\infty dt (\operatorname{Id} - T_0^\eta) L^* v^-_t ,
		\quad \text{where} \quad T_0^{\eta} g ( y ) = \sum_{| \mathbf{n} |< \eta} \frac{1}{\mathbf{n} !} \partial^{\mathbf{n}} g ( 0 ) \, y^{\mathbf{n}} , 
	\end{align}	
	In a first step, we obtain bounds for the integral formula appearing in \eqref{int_rep} which will in particular imply that it converges in the sense of distributions. 
We assume to be given $p \in \mathbb{N}$ and shall prove that the statement holds for $p^{\prime} = \max \lbrace \eta+1, p+2, p+D+1 \rbrace$.
We choose an arbitrary $\psi$ such that 
$\| \psi \|_{p^{\prime}} \leq 1$. 
We also fix $\mu < \infty$ and proceed by considering separately the near field regime where $t \le \mu^4$ and the far field regime where $t > \mu^4$. 
	\medskip
	
	We start with the near field. Here we consider each term in the Taylor expansion separately. For the first term, we have	
	\begin{align}\label{sp02}
		\Big | \int_0^{\mu^4} (L^* v^-_t)_\mu(0) dt \Big | 
		& \le \big\| (L \psi_\mu \ast \Psi_t)_{\mu^{-1}} \big\|_p \int_0^{\mu^4} \mu^{\eta - 4} dt 
		\\ \nonumber
		&
		= \big\| (L \psi_\mu \ast \Psi_t)_{\mu^{-1}} \|_p \, \mu^\eta
	\end{align}
	where we have made use of the assumption \eqref{schauder_ass}. Now we claim that 
$\| (L \psi_\mu \ast \Psi_t)_{\mu^{-1}} \|_p \lesssim \|\psi\|_{p^{\prime}}$ 
uniformly in $0 < t \le \mu^4 < \infty$. To this end, it suffices to note that
	\begin{align*}
		\| (L \psi_\mu \ast \Psi_t)_{\mu^{-1}} \|_p 
		&= \|(L\psi_\mu)_{\mu^{-1}} \ast \Psi_{t/\mu^4} \|_p
		\\
		&= \mu^{-2} \| L \psi \ast \Psi_{t/\mu^4} \|_p
		\\
		& \lesssim \mu^{-2} \|\psi\|_{p+2}
	\end{align*}
	where the last line follows by writing
	\begin{align*}
		& \sup_x |x|^\alpha \int |L\psi(x-y) \Psi_{t/\mu^4}(y)| dy 
		\\
		&\le \sup_x \int (1+ |x-y|)^\alpha |L\psi(x-y)| (1+|y|)^\alpha |\Psi_{t/\mu^4}(y)| dy
		\\
		&\lesssim \| \psi \|_{p+2} \int (1+|y|)^\alpha |\Psi_{t/\mu^4}(y)| dy
	\end{align*}
	and noting that a simple change of variables allows one to bound the integral in terms of the Schwartz seminorms of $\Psi$.	
	In the near field range, it remains to consider the terms appearing in the Taylor jet. For $|\mathbf{n}| < \eta$, we have that 
	\begin{align}\label{sp03}
		\left | \int_0^{\mu^4} dt ((\cdot)^{\mathbf{n}})_\mu(0) \partial^{\mathbf{n}} L^* v^-_t(0) \right | 
		& \lesssim \| \psi \|_{p^{\prime}} \, \mu^{|\mathbf{n}|} \int_0^{\mu^4} dt (\sqrt[4]{t})^{\eta - 4 - |\mathbf{n}|} 
		\\ \nonumber
		& \lesssim \| \psi \|_{p^{\prime}} \, \mu^{\eta}
	\end{align}
	where we have made use of the fact that $p^{\prime} > \eta$. Combining \eqref{sp02}, \eqref{sp03} yields the near field bound
	\begin{align*}
		\left | \int_0^{\mu^4} dt (\operatorname{Id} - T_0^\eta) (L^* v^-_t)_\mu(0) \right | \lesssim \| \psi \|_{p^{\prime}} \mu^{\eta} .
	\end{align*}
	We now turn to the far field regime. As in \cite[Section 3.6]{BOT}, we can apply Taylor’s Theorem to the effect of
	\begin{align*}
		(\operatorname{Id} - T_0^\eta) L^* v^-_t(y) = \sum_{\substack{|\mathbf{n}| \ge k \\ \sum_i n_i \le k}} c_{\mathbf{n}} y^{\mathbf{n}} \int_0^1 ds (1-s)^{k-1} s^{|\mathbf{n}| - k} \partial^{\mathbf{n}} L^* v^-_t(Sy)
	\end{align*}
	where $Sy = (s^2 y_0, sy_1, \dots, sy_d)$ and we have set $k = \lceil \eta \rceil$.
	Therefore 
	\begin{align*}
		[(\operatorname{Id} -  &T_0^\eta) L^* v^-_t]_\mu(0) \nonumber \\ & = \sum_{\substack{|\mathbf{n}| \ge k \\ \sum_i n_i \le k}} c_{\mathbf{n}} \mu^{|\mathbf{n}|} \int_0^1 ds (1-s)^{k-1} s^{|\mathbf{n}| - k} L^* v^- \left ( (\cdot^{\mathbf{n}} \psi)_{s \mu} \ast \partial^{\mathbf{n}} \Psi_t \right ).
	\end{align*}
	
	Therefore, we obtain that
	\begin{align*}
		& \Bigg | \int_{\mu^4}^\infty dt  |((\operatorname{Id} - T_0^\eta) L^* v_t^-)_\mu(0) \Bigg | 
		\\ 
		& \lesssim \sum_{\substack{|\mathbf{n}| \ge k \\ \sum_i n_i \le k}} \mu^{|\mathbf{n}|} \int_{\mu^4}^\infty dt \sup_{s \in (0,1)} 
		\big\|
		( (\cdot^{\mathbf{n}} \psi)_{s \mu} \ast L \partial^{\mathbf{n}} \Psi_t)_{(\sqrt[4]{t})^{-1}} \big\|_p
		(\sqrt[4]{t})^{\eta - 4 - |\mathbf{n}|}. 
	\end{align*}
	Now we note that
	\begin{align*}
		\big\| ( (\cdot^{\mathbf{n}} \psi)_{s \mu} \ast L \partial^{\mathbf{n}} \Psi_t)_{(\sqrt[4]{t})^{-1}} \big\|_p
		& = (\sqrt[4]{t})^{-|\mathbf{n}| - 2} \big\| (\cdot^{\mathbf{n}} \psi)_{s \mu/ \sqrt[4]{t}} \ast L \partial^{\mathbf{n}} \Psi \big\|_p
	\end{align*}	
	and write
	\begin{align*}
		& \sup_x \int |x|^\alpha |\psi_{s\mu/\sqrt[4]{t}}(x-y)| |L\partial^{\mathbf{n} + \beta} \Psi(y)| dy 
		\\
		& \leq \sup_x \int (1+ |x-y|)^\alpha |\psi_{s\mu/\sqrt[4]{t}}(x-y)| (1+|y|)^\alpha |L\partial^{\mathbf{n} + \beta} \Psi(y)| dy
		\\
		& \lesssim \int \left ( 1 + \frac{s\mu}{\sqrt[4]{t}} |x-y| \right )^\alpha |\psi(x-y)| dy
	\end{align*}	
	where the implicit constant depends on the Schwartz seminorms of $\Psi$. Since $s\mu \le \sqrt[4]{t}$, we conclude that 
	\begin{align*}
		\sup_{s \in (0,1)} \big\| ( (\cdot^{\mathbf{n}} \psi)_{s \mu} \ast L \partial^{\mathbf{n}} \Psi_t)_{(\sqrt[4]{t})^{-1}} \big\|_p 
		\lesssim (\sqrt[4]{t})^{-|\mathbf{n}| - 2} \| \psi \|_{p + D + 1} 
	\end{align*}
	which yields the far field estimate
	\begin{align*}
		\left | \int_{\mu^4}^\infty dt \langle (\operatorname{Id} - T_0^\eta) L^* v_t^-, \psi^\mu\rangle  \right | 
		\lesssim \|\psi\|_{p^{\prime}} \, \mu^{\eta}.
	\end{align*}
	
	We now argue that $L \tilde{v} = v^-$ where we defined $\tilde{v}$ to be the r.~h.~s.~ of \eqref{int_rep}. Indeed, for any Schwartz function $\psi$ we can write
	\begin{align*}
		L \int_{\tau}^T dt \, \left [(\operatorname{Id} - T_0^\eta) L^* v^-_t \right](y) &=\int_{\tau}^T dt \, \langle (\operatorname{Id} - T_0^{\eta - 2}) L^* L v^-_t , \psi \rangle  
		\\
		&= - \int_\tau^T dt \, \Big \langle  \frac{d}{dt}  (\operatorname{Id} - T_0^{\eta - 2}) v^-_t , \psi \Big \rangle 
		\\ & = \langle (\operatorname{Id} - T_0^{\eta - 2})  v^-_\tau ,\psi \rangle -  \langle (\operatorname{Id} - T_0^{\eta - 2}) L^* v^-_T ,\psi \rangle.
	\end{align*}
	By the same techniques as used in the integral estimates above we see that
	\begin{align*}
		\langle T_0^{\eta - 2}v^-_\tau, \psi \rangle \to 0, \qquad \langle  (\operatorname{Id} - T_0^{\eta - 2}) v^-_T, \psi \rangle \to 0
	\end{align*}
	as $\tau \to 0$ and $T \to \infty$. Since $ v^-_\tau \to v^-$ as $\tau \to 0$, we indeed have that $L \tilde{v} = v^-$. 
	Since we also have that $L v = v^-$ and 
	\begin{align*}
		\sup_{\mu < \infty} \mu^{-\eta} |v_\mu(0)| + \sup_{\mu < \infty} \mu^{-\eta} |\tilde{v}_\mu(0)| < \infty
	\end{align*}
	by a Liouville principle argument of the same type as given \cite[Section 1.11]{BOT}, we conclude that $v = \tilde{v}$. This yields the integral formula \eqref{int_rep} and hence also the bound \eqref{int_bound}.
\end{proof}

\subsection{Proof of Three-point Argument I}
\begin{proof}
We note that 
\begin{align}\label{3pid}
	\langle R_x - R_y, \psi \rangle = \sum_{0 \le |\beta| < \kappa} (f_y. - f_x. Q_\kappa \Gamma_{xy}^*) \mathsf{z}^{\beta} \, \langle \Pi_{y\beta}, \psi \rangle.
\end{align}
As a consequence, for any choice of Schwartz function $\psi$ and any $\mu < \infty$, we can write 
\begin{align*}
	 \Big | \langle R_x &- R_y, \psi_{\mu}(y- \cdot) \rangle - \sum_{\substack{0 \le |\beta| < \kappa \\ \beta \neq \text{pp}}} (f_y - f_x . Q_\kappa \Gamma_{xy}^*)\mathsf{z}^{\beta} \, \langle \Pi_{y \beta}, \psi_{\mu}(y - \cdot)\rangle \Big | \\ \nonumber &= \Big | \int dz \Bigg [\sum_{|\mathbf{n}| < \kappa} (f_y - f_x. Q_\kappa \Gamma_{xy}^*) \mathsf{z}_{\mathbf{n}} (z-y)^{\mathbf{n}} \Bigg ] \psi_{\mu}(y - z) \Big |.
\end{align*}
The idea of the proof is then to obtain bounds for the left-hand side of the above for choices of $\psi$ that allow us to isolate terms on the right-hand side and with $\mu = |y-x|$.
We first note that for any finite\footnote{in fact, this claim holds for any subset which is bounded w.r.t.\ the Schwartz topology}
subset $\mathcal{B} \subset \mathcal{S} ( \mathbb{R}^{1+d})$, the l.h.s.\ is bounded uniformly over $\psi$ in that set by a constant times
	\begin{align*}
		|y-x|^\kappa \Big ( \sup_{\psi \in \mathcal{B}} \sup_{x \in \mathbb{R}^{1 + d}} \sup_{\mu < \infty} \frac{|R_{x\mu}(x)|}{\mu^{\kappa}} + \sum_{\substack{0\le |\beta| < \kappa \\ \beta \neq \text{pp}}} [f]_{\kappa, \beta} \|\Pi\|_{\kappa} \Big ) .
	\end{align*}
Indeed, the bound for the sum follows by definition of the norms involved, whilst the bound for each term containing $R$ follows 
by observing that we may write $\psi_{|y-x|}(y-\cdot) = \tilde{\psi}_{2|y-x|}(x - \cdot)$ for some Schwartz functions $\tilde{\psi}$ whose seminorms are bounded in terms of those of $\psi$. 
By choosing a $\psi$ such that $ \int dz\, \psi(z) z^\mathbf{n} = \delta_{\mathbf{n}, \mathbf{m}}$ for $|n| < \kappa$ and noting that then $$ \int dz\, \psi_{y}^{|y-x|}(z) (z - y)^{\mathbf{n}} = |y-x|^{|\mathbf{n}|} \delta_{\mathbf{n}, \mathbf{m}}$$ for $|\mathbf{n}| < \kappa$, we obtain the bound for $[f]_{\kappa, \delta_{\mathbf{m}}}$. 
\end{proof}

\subsection{Proof of A Priori Estimates}

\begin{proof}[Proof of Theorem~\ref{a priori}]
We now demonstrate how to make use of the ingredients assembled above to obtain Theorem~\ref{a priori}. We assume to be given $f$ as in the statement of the result with $|\lambda| \le \lambda_0$ where $\lambda_0$ is free to be tuned inside the proof. 
Throughout this proof, we assume for convenience and without loss of generality that the torus is of unit size i.e.\ that $\ell=1$.
\medskip

\PfStart{ap} 
\PfStep{ap}{ap:1} 
\textsc{Algebraic continuity.}
By our Algebraic Continuity Lemma (Lemma~\ref{continuity}), for populated multi-indices $\beta$ with $|\beta| < \kappa$ and $\beta \neq \text{pp}$, we have that
\begin{align*}
    [f]_{\kappa, \beta} \lesssim |\lambda| ([f]_{\kappa, \mathrm{pol}} + \| f . \mathsf{z}_{\mathbf{0}} \| + | \lambda |) 
    \lesssim | \lambda | \, \| f \|_{\kappa} .
\end{align*}

\PfStep{ap}{ap:2}
\textsc{Reconstruction.}
We then note that
\begin{align*}
    R_x^- = u^- - f_x. Q_{\kappa} \Pi_x^-
\end{align*}
so that
\begin{align}\label{ap01}
    R_{x+h}^- - R_x^- =  \sum_{0 < |\beta| < \kappa} (f_{x}. - f_{x+h}. Q_{\kappa} \Gamma_{x+h \, x}^* ) \mathsf{z}^{\beta} \Pi_{x\beta}^-.
\end{align}

As a result, 
for $\mu \leq \ell = 1$,
we have the bound 
\begin{align}
    |(R_{x+h}^- - R_x^-)_{\mu}(x)| &\lesssim  \sum_{\substack{\kappa^- \le |\beta| < \kappa \\ \beta \neq \text{pp}}} [f]_{\kappa, \beta} |h|^{\kappa - |\beta|} \mu^{|\beta|-2} \nonumber
	\\
    &\lesssim |\lambda| \, \| f \|_{\kappa} \, \mu^{\kappa^- - 2} ( \mu + |h|)^{\kappa - \kappa^-} \label{ap02}
\end{align}
where we made use of the definition
\eqref{i7}
of $\kappa^-$ and the fact that $\Pi^-$ has vanishing purely-polynomial components.
Since $R_{x \mu}^-(x) \to 0$ in the sense of tempered distributions as $\mu \downarrow 0$ by assumption
\eqref{t13} together with periodicity of $x \mapsto R_{x \mu}^{-} ( x )$ and $\kappa > 2$,
we can apply the Reconstruction Theorem of Lemma~\ref{Recon} to conclude that there exists a $p > 0$
depending only on $\alpha, \kappa$ and $d$,
such that
$\mu \leq 1$,
\begin{align}\label{ap03}
    |R_{x\mu}^-(x)| \lesssim |\lambda| \, \| f \|_{\kappa} \, \mu^{\kappa - 2} ,
\end{align}
uniformly over $x \in \mathbb{R}^{1+d}, \mu \le \ell = 1$ and $\psi$ such that $\|\psi\|_p \le 1$. We now consider this choice of $p$ to be fixed for the remainder of the proof.

\PfStep{ap}{ap:3} 
\textsc{Integration}.
We would like to apply our Schauder estimate in the form of Lemma~\ref{Schauder} to the PDE \eqref{t2}
\begin{align*}
    L R_x = R_x^- + P_x
\end{align*}
where we recall from \eqref{po1} and the following discussion that 
\begin{align} \label{apn10b}
    P_x(\cdot) = \sum_{\kappa - 2 < |\mathbf{n}| < \kappa} a_{x\mathbf{n}} (\cdot - x)^{\mathbf{n}}
\end{align}
and in a first step 
we claim that 
	\begin{align} \label{apn10}
		\sup_{x \in \mathbb{R}^{1 + d}} |a_{x \mathbf{n}}|
		& \lesssim |\lambda| \| f \|_{\kappa} .
	\end{align}
To that effect, define coefficients $b_{x \beta \mathbf{n}}$
by 
$L \Pi_{x \beta} - \Pi_{x \beta}^- = \sum_{| \mathbf{n}| \leq \beta} b_{x \beta \mathbf{n}} ( \cdot - x )^{\mathbf{n}}$, recall \eqref{Hi01}.
On the one hand, combining the expression \eqref{po1} of $P$ with the fact that $a_{x\mathbf{n}}$ vanishes unless 
$\kappa - 2 < |\mathbf{n}| < \kappa$, 
we learn that 
	\begin{align*}
		a_{x \mathbf{n}}
		& = \sum_{| \beta | < \kappa, \beta \neq \mathrm{pp}} f_x . \mathsf{z}^{\beta} \, b_{x \beta \mathbf{n}}
		\qquad \text{for $\kappa - 2 < |\mathbf{n}| < \kappa$}.
	\end{align*}
where the restriction to non-pp multi-indices comes from the observation that $b_{x \delta_{\mathbf{m}} \mathbf{n}} = 0$ unless $| \mathbf{n} | \leq | \mathbf{m} | - 2$,
recall \eqref{mb08}.
On the other hand, 
by equivalence of norms on the finite-dimensional space of polynomials of (parabolic) degree less than $| \beta |$,
	\begin{align*}
		\max_{| \mathbf{n} | \leq | \beta |} | b_{x \beta \mathbf{m}} |
		& \lesssim
		\sup_{\frac{1}{2} \leq \mu \leq 1} \Big| \sum_{|\mathbf{n}| \leq |\beta|} b_{x \beta \mathbf{n}} \mu^{\mathbf{n}} \Big|
		\lesssim \sup_{\frac{1}{2} \leq \mu \leq 1} | (L\Pi_{x \beta} - \Pi_{x \beta}^-)_{\mu} ( x ) | 
		\lesssim \| \Pi \|_{\kappa} ,
	\end{align*}
where in the second inequality we have fixed an arbitrary test-function $\psi$ such that $\int \psi ( z ) z^{\mathbf{n}} dz \neq 0$ for all $| \mathbf{n} | \leq | \beta |$,
and where 
we have implicitly used Lemma~\ref{kernel_swap} to upgrade the model bounds in terms of the semigroup to ones for a generic test function.
It follows that
$\sup_{x \in \mathbb{R}^{1 + d}} |a_{x \mathbf{n}}|\lesssim \max_{| \beta | < \kappa, \beta \neq \mathrm{pp}} \| f . \mathsf{z}^{\beta} \|$.
Arguing as in the proof of Lemma~\ref{continuity} one may post-process \eqref{wr08} and \eqref{wr02} into 
	\begin{align*}
		\max_{| \beta |< \kappa, \beta \neq \mathrm{pp}}
		\| f . \mathsf{z}^{\beta} \| 
		\lesssim |\lambda| \| f \|_{\kappa} ,
	\end{align*}
which
in combination with the above
yields the desired inequality
\eqref{apn10}.

\medskip

We now claim that 
	\begin{align} \label{apn11}
		|(L R_x)_\mu (x) | 
		& \lesssim \big( \rho^{- \delta} + \rho^{p} | \lambda | \big) \| f \|_{\kappa}  \, \mu^{\kappa - 2} ,
	\end{align}
where the implicit constant is uniform over $x \in \mathbb{R}^{1 + d}$, $\mu < \infty$, $\| \psi \|_{p} \leq 1$ and $\rho \geq 1$, 
and where 
	\begin{align} \label{is03}
		\delta \coloneqq \min \lbrace \kappa - | \beta | : |\beta| < \kappa \rbrace > 0 .
	\end{align}
For that purpose, we distinguish the cases $\mu \lessgtr \rho$.
We start with the case $\mu \geq \rho$, where we use the expression
(recall Remark~\ref{r:gDP})
	\begin{align} \label{apn12}
		R_{x \mu} ( x )
		& = [( f_{\cdot} - f_x ).\mathsf{z}_\mathbf{0}] * \psi_\mu ( x ) - \sum_{\kappa^- \leq | \beta | < \kappa} f_x . \mathsf{z}^{\beta} \, \Pi_{x \beta \mu} ( x ).
	\end{align}
The same identity holds for $LR$ provided one replaces $\psi$ by $\mu^{-2} L \psi$ in the r.h.s.\  therefore we obtain
	\begin{align*}
		| (L R_{x} )_{\mu}(x) |
		& \lesssim \|f\|_{\kappa} \sum_{0 \leq | \beta | < \kappa} \mu^{| \beta |-2} 
		\lesssim \mu^{\kappa - 2} \|f\|_{\kappa} \sum_{0 \leq | \beta | < \kappa} \rho^{| \beta | - \kappa} .
	\end{align*}
The desired 
$| (L R_{x} )_{\mu}(x) | \lesssim \rho^{- \delta} \|f\|_{\kappa} \, \mu^{\kappa - 2}$ follows.

\medskip
We turn to the case $\mu \leq \rho$, 
where we use 
the observation that
$\psi_{\mu} = ( \psi_{\rho} )_{\mu/\rho}$
to write
	\begin{align*}
		(L R_{x} )_{\mu}(x) = R_x^{-} * ( \psi_{\rho} )_{\frac{\mu}{\rho}} ( x ) + P_{x \mu} ( x ) .
	\end{align*}
On the one hand, we learn from \eqref{apn10b} and \eqref{apn10}
that 
	\begin{align*}
		| P_{x \mu} ( x ) |
		& \lesssim \sum_{\kappa - 2 < |\mathbf{n}|< \kappa} |\lambda | \|f\|_{\kappa} \, \mu^{| \mathbf{n}|}
		\lesssim \rho^2 |\lambda | \|f\|_{\kappa} \, \mu^{\kappa - 2} 
		\leq \rho^p |\lambda | \|f\|_{\kappa} \, \mu^{\kappa - 2} ,
	\end{align*}
where we implicitly used the fact that $p$ can be chosen so large that $p > \kappa + D \geq 2$.
On the other hand, by
\eqref{ap03}
	\begin{align*}
		| R_x^{-} * ( \psi_{\rho} )_{\frac{\mu}{\rho}} ( x ) |
		& \lesssim | \lambda | \| f \|_{\kappa} \| \psi_{\rho} \|_p \, \mu^{\kappa - 2}
		\leq \rho^p | \lambda | \| f \|_{\kappa} \, \mu^{\kappa - 2} ,
	\end{align*}
where the latter inequality follows from the estimate 
$\| \psi_{\rho} \|_p \leq \rho^p \| \psi \|_p$
which can be obtained by directly computing the corresponding Schwartz seminorms \eqref{ssn}.

\medskip
By applying our Schauder estimate, Lemma~\ref{Schauder} with \eqref{apn11} as input, 
we conclude that there exists $p^\prime \in \mathbb{N}$
depending only on $\alpha, \kappa$ and $d$,
such that 
 \begin{align}\label{apn06}
		| R_{x \mu} ( x ) |
		& \lesssim ( \rho^{- \delta} + \rho^{p} | \lambda | ) \, \| f \|_{\kappa} \, \mu^{\kappa} ,
    \end{align}
uniformly over $x \in \mathbb{R}^{1 + d}$, $\mu < \infty$, $\| \psi \|_{p^{\prime}} \leq 1$ and $\rho \geq 1$.

\PfStep{ap}{ap:4} 
\textsc{Three-Point Argument.}
    We now 
    apply the Three-Point Argument contained in Lemma~\ref{3_point} to transfer this estimate to a bound on $f$. 
In particular, we see that the combination of \eqref{apn06}, Lemma~\ref{continuity} and Lemma~\ref{3_point} imply that
    \begin{align} \label{apn08}
		[ f ]_{\kappa, \mathrm{pol}}
		& \lesssim ( \rho^{- \delta} + \rho^{p} | \lambda | ) \, \| f \|_{\kappa}.
    \end{align}

\PfStep{ap}{ap:5} 
\textsc{Incorporating Boundary Data.}
    We now claim that as a consequence of our boundary data, namely the periodicity of $f.\mathsf{z}_\mathbf{0}$, we have the estimate 
    \begin{align}\label{apn07}
        \|f\|_{\kappa} \le \sum_{\substack{|\beta| < \kappa \\ \beta \ne 0}} \left ( \| f . \mathsf{z}^{\beta} \| + [f]_{\kappa, \beta} \right ) \lesssim |\lambda| + [f]_{\kappa, \mathrm{pol}}.
    \end{align}
    Due to the action of $f$ and $\Gamma^*$ on multi-indices of the form $\mathsf{z}_\mathfrak{3}^k$, it suffices to consider $\beta$ such that $|\beta| \ge 0$ in the above sum. We then note that as a consequence of Lemma~\ref{continuity},
    \begin{align*}
        \sum_{\substack{0 \le |\beta| < \kappa \\ \beta \ne 0}} [f]_{\kappa, \beta} \lesssim [f]_{\kappa, \mathrm{pol}} + |\lambda| (\| f . \mathsf{z}_{\mathbf{0}} \| + |\lambda|).
    \end{align*}
    Also, as a consequence of our interpolation estimates in Lemma~\ref{l:intI} and multiplicativity, we have the estimate
    \begin{align*}
        & \sum_{\substack{0 \le |\beta| < \kappa \\ \beta \ne 0}} \| f . \mathsf{z}^{\beta} \| \\
        & \lesssim \sum_{\substack{0 \le |\beta| < \kappa \\ \beta \ne 0}} \lambda^{\beta(\mathfrak{3})} ([f]_{\kappa, \mathrm{pol}} + |\lambda|)^{\frac{\sum_{\mathbf{n}} \beta(\mathbf{n}) |\mathbf{n}|}{\kappa}} (\| f . \mathsf{z}_{\mathbf{0}} \| + |\lambda|)^{\frac{\kappa - |\mathbf{n}|}{\kappa}\sum_{\mathbf{n}} \beta(\mathbf{n})}.
    \end{align*}
    By distinguishing the cases where $\sum_{\mathbf{n}} \beta(\mathbf{n}) = 0$ and $\sum_{\mathbf{n}} \beta(\mathbf{n}) > 0$ and using the fact that the assumptions of Theorem~\ref{a priori} imply that
    \begin{align*}
        (\| f . \mathsf{z}_{\mathbf{0}} \| + |\lambda|)^{\sum_\mathbf{n} \beta(\mathbf{n}) - 1} \lesssim 1
    \end{align*}
    in the latter case,
    we see that the summand can be estimated by Young's inequality to the effect of
    \begin{align*}
        \sum_{\substack{0 \le |\beta| < \kappa \\ \beta \ne 0}} \lambda^{\beta(\mathfrak{3})} ([f]_{\kappa, \mathrm{pol}} + |\lambda|)^{\frac{\sum_{\mathbf{n}} \beta(\mathbf{n}) |\mathbf{n}|}{\kappa}} & (\| f . \mathsf{z}_{\mathbf{0}} \| + |\lambda|)^{\frac{\kappa - |\mathbf{n}|}{\kappa}\sum_{\mathbf{n}} \beta(\mathbf{n})} \\
       & \lesssim [f]_{\kappa, \mathrm{pol}} + \| f . \mathsf{z}_{\mathbf{0}} \| + |\lambda|.
    \end{align*}
    Thus we have shown that
    \begin{align*}
        \|f\|_\kappa \lesssim [f]_{\kappa, \mathrm{pol}} + \| f . \mathsf{z}_{\mathbf{0}} \| + |\lambda|.
    \end{align*}
    In order to obtain \eqref{apn07} it remains to remove the dependence on $\| f . \mathsf{z}_{\mathbf{0}} \|$ on the right-hand side.
    \medskip

    To this end, we use the periodicity of $f.\mathsf{z}_\mathbf{0}$ and the fact that $[f]_{\kappa^-, \mathrm{pol}}$ coincides with $[f.\mathsf{z}_\mathbf{0}]_{\kappa^-}$ to obtain
    \begin{align*}
        \| f . \mathsf{z}_{\mathbf{0}} \| \lesssim [f]_{\kappa^-, \mathrm{pol}} \lesssim ([f]_{\kappa, \mathrm{pol}} + |\lambda|)^{\frac{\kappa^-}{\kappa}}(\| f . \mathsf{z}_{\mathbf{0}} \| + |\lambda|)^{\frac{\kappa - \kappa^-}{\kappa}}
    \end{align*}
    where the last inequality follows from Lemma~\ref{l:intI}. 
    By Young's inequality this implies that for every $\varepsilon > 0$
    \begin{align*}
        \| f . \mathsf{z}_{\mathbf{0}} \| \lesssim C_{\varepsilon} ([f]_{\kappa, \mathrm{pol}} + |\lambda|) + \varepsilon \| f . \mathsf{z}_{\mathbf{0}} \| 
    \end{align*}
    which can be buckled to the effect of
    \begin{align*}
        \| f . \mathsf{z}_{\mathbf{0}} \| \lesssim [f]_{\kappa, \mathrm{pol}} + |\lambda|.
    \end{align*}
    This completes the proof of \eqref{apn07}. Combining \eqref{apn08} and \eqref{apn07} we see that
    \begin{align*}
        \|f\|_{\kappa} \lesssim |\lambda| +(\rho^{-\delta} + \rho^{p} |\lambda|) \|f\|_{\kappa}
    \end{align*}
    which can be buckled by taking first $\rho \gg 1$ and then $|\lambda| \ll 1$ to the effect of
    \begin{align*}
        \|f\|_\kappa \lesssim |\lambda|,
    \end{align*}
as desired.
\end{proof}

\section{Continuity of the Solution Map in the Model}\label{Con}

\subsection{Proof of Algebraic Continuity Lemma II}
Following a similar strategy of proof to that of Lemma~\ref{continuity}, our proof of Lemma~\ref{continuity_2} is split into two parts. We note that due to Theorem~\ref{a priori} it is legitimate to invest control on $\|f\|_\kappa$ rather than merely on $\|f.\mathsf{z}_\mathbf{0}\|$ in this section.

\medskip

In our first step, we establish the following interpolation estimate, which is the analogue of Lemma~\ref{l:intI}
for differences of modelled distributions.
\begin{lemma}[Interpolation II]\label{l:intII}
Suppose that $f, \tilde{f}$ are modelled distributions in the sense of Definition~\ref{rf01} 
with respect to models $(\Pi, \Pi^-, \Gamma^*)$, $(\tilde{\Pi}, \tilde{\Pi}^-, \tilde{\Gamma}^*)$ respectively (but for the same $\lambda$).
We assume that
\begin{align}\label{wr20}
	|\lambda|
		+ \| \Gamma^* \|_\kappa + \|\tilde{\Gamma}^*\|_\kappa 
		+ \sum_{|{\bf n}|<\kappa} ( \|f.\mathsf{z}_{\bf n}\| + \|\tilde f.\mathsf{z}_{\bf n}\| )	
		\le M
\end{align}
for some $M < \infty$.
Then we have 
\begin{align}
\|(f-\tilde f).\mathsf{z}_{\bf n}\|&
\lesssim
\big([f;\tilde f]_{\kappa,{\rm pol}}+|\lambda|\|(f-\tilde f).\mathsf{z}_{\bf 0}\|
+|\lambda|\|\Gamma^*-\tilde\Gamma^*\|\big)^{\frac{|{\bf n}|}{\kappa}}\nonumber\\
&\quad\times\big(\|(f-\tilde f).\mathsf{z}_{\bf 0}\|
+|\lambda|\|\Gamma^*-\tilde\Gamma^*\|\big)^{1-\frac{|{\bf n}|}{\kappa}}
\quad\mbox{for}\;|{\bf n}|<\kappa,\label{wr16}\\
[f;\tilde f]_{\tilde{\kappa},{\rm pol}}&
\lesssim
\big([f;\tilde f]_{\kappa,{\rm pol}}+|\lambda|\|(f-\tilde f).\mathsf{z}_{\bf 0}\|
+|\lambda|\|\Gamma^*-\tilde\Gamma^*\|\big)^{\frac{\tilde{\kappa}}{\kappa}}\nonumber\\
&\quad\times\big(\|(f-\tilde f).\mathsf{z}_{\bf 0}\|
+|\lambda|\|\Gamma^*-\tilde\Gamma^*\|\big)^{1-\frac{\tilde{\kappa}}{\kappa}}
\quad\mbox{for}\;\tilde{\kappa}<\kappa,\label{wr36}
\end{align}
where the implicit constant depends on $M, \alpha, d$ and $\kappa$.
\end{lemma}

In our second step, we establish the following preliminary algebraic estimate, which is the analogue of Lemma~\ref{l:pcI} for differences of modelled distributions.
\begin{lemma}\label{l:pcII}
Suppose that $f, \tilde{f}$ are modelled distributions in the sense of Definition~\ref{rf01} 
with respect to models $(\Pi, \Pi^-, \Gamma^*)$, $(\tilde{\Pi}, \tilde{\Pi}^-, \tilde{\Gamma}^*)$ respectively (but for the same $\lambda$).
We assume the a priori control of
\begin{align}\label{wr58}
	\| \Gamma^* \|_\kappa + \|\tilde{\Gamma}^*\|_\kappa 
		+ \|f\|_\kappa + \|\tilde f\|_\kappa
		\le M
\end{align}
for some $M < \infty$.
Then we have for all $\beta$ and $\eta$ with $|\beta|<\eta\le\kappa$
\begin{align}\label{wr59}
\lefteqn{[f; \tilde f]_{\eta,\beta}}\nonumber\\
&\lesssim|\lambda|^{\beta ( \mathfrak{3} )}
\big(|\lambda|\|\tilde\Gamma^*-\Gamma^*\|_{\eta}+\max_{\tilde\eta\le\eta}
[\tilde f;f]_{\tilde\eta,{\rm pol}}+\max_{|{\bf n}|<\eta}\|(\tilde f-f).\mathsf{z}_{\bf n}\|\big) ,
\end{align}
where the 
(first)
max runs over all $\tilde{\eta} > 0$ of the form $\tilde{\eta} = \eta - ( | \gamma | - \alpha )$
and where
the implicit constant depends on $M, \alpha, d$ and $\kappa$.
\end{lemma}

With these ingredients in hand, it is now straightforward to prove Lemma~\ref{continuity_2}.
\begin{proof}[Proof of Lemma~\ref{continuity_2}]
Recall that in the setting of Lemma~\ref{continuity_2} we assume the a priori estimate of Theorem~\ref{a priori} to hold, so that we may apply Lemmas~\ref{l:intII} and \ref{l:pcII}.
The desired result then follows by inserting \eqref{wr16} and \eqref{wr36} into the right-hand side of \eqref{wr59} (where we take $\eta = \kappa$).
\end{proof}

We turn to the proof of Lemma~\ref{l:intII}. 
\begin{proof}[Proof of Lemma~\ref{l:intII}]
We start by establishing the following
version of 
\eqref{wr08} for finite differences
\begin{align}\label{wr18}
\|(f-\tilde f).\mathsf{z}^\beta\|
\lesssim|\lambda|^{\beta ( \mathfrak{3} )}
\sum_{\bf n}\beta({\bf n})\|(f-\tilde f).\mathsf{z}_{\bf n}\|.
\end{align}
For this purpose, for $t\in[0,1]$ we introduce 
$f_t$ in line with Definition
\ref{rf01} by imposing
\begin{align}\label{wr19}
f_t.\mathsf{z}_{\mathfrak{3}}=\lambda,\quad
f_t.\mathsf{z}_{\bf n}=tf.\mathsf{z}_{\bf n}+(1-t)\tilde f.\mathsf{z}_{\bf n},
\end{align}
together with the multiplicative form \eqref{rf02}.
By construction, 
$f_t$ interpolates 
between
$f_{t=0}=\tilde f$ and $f_{t=1}=f$.
By definition (\ref{wr19}), assumption (\ref{wr20}) implies 
$\|f_t.\mathsf{z}_{\bf n}\|\lesssim 1$. 
Appealing to (\ref{wr08}) we learn
$\|f_t.\mathsf{z}^\beta\|$ $\lesssim 1$.
Note that 
$\frac{df}{dt}$
satisfies the form 
\eqref{t49}
w.r.t.\ $f_t$, so that by (\ref{wr20})
\begin{align*}
\Big \|\frac{df}{dt}.\mathsf{z}^\beta \Big \|\lesssim|\lambda|^{\beta ( \mathfrak{3} )}
\sum_{\bf n}\beta({\bf n})\Big \|\frac{df}{dt}.\mathsf{z}_{\bf n}\Big \|.
\end{align*}
Since by definition (\ref{wr19}) we have $\frac{df}{dt}.\mathsf{z}_{\bf n}$
$=(f-\tilde f).\mathsf{z}_{\bf n}$, we deduce (\ref{wr18}) by integration. 

\medskip

We now turn to the proof of (\ref{wr16}) and start from the definition \eqref{refonp19}
of $[f;\tilde f]_{\kappa,{\rm pol}}$, which we treat as $[f]_{\kappa,{\rm pol}}$
in the proof of Lemma \ref{l:intI} (see \eqref{wr21}). 
This leads to one additional term,
which by the last item in (\ref{mb07}) is restricted to $\beta$ not purely polynomial:
\begin{align*}
\lefteqn{\big|\sum_{|{\bf n}|<\kappa}(f_x-\tilde f_x).\mathsf{z}_{\bf n}h^{\bf n}\big|
\le[f;\tilde f]_{\kappa,{\rm pol}}|h|^\kappa+\|(f-\tilde f).\mathsf{z}_{\bf 0}\|}\nonumber\\
&+\|\Gamma^*\|\sum_{\stackrel{0\le|\beta|<\kappa}{\beta ( \mathfrak{3} )\not=0}}
\|(f-\tilde f).\mathsf{z}^\beta\||h|^{|\beta|}
+\|\Gamma^*-\tilde\Gamma^*\|\sum_{\stackrel{0\le|\beta|<\kappa}{\beta ( \mathfrak{3} )\not=0}}
\|\tilde f.\mathsf{z}^\beta\||h|^{|\beta|}.
\end{align*}
By the norm equivalence argument of Lemma \ref{l:intI}, see (\ref{wr22}), this yields
\begin{align*}
\lefteqn{\max_{|{\bf n}|<\kappa}\|(f_x-\tilde f_x).\mathsf{z}_{\bf n}\|r^{|{\bf n}|}
\lesssim[f;\tilde f]_{\kappa,{\rm pol}}r^\kappa+\|(f-\tilde f).\mathsf{z}_{\bf 0}\|}\nonumber\\
&+\max_{\stackrel{0\le|\beta|<\kappa}{\beta ( \mathfrak{3} )\not=0}}
\|(f-\tilde f).\mathsf{z}^\beta\|r^{|\beta|}
+\|\Gamma^*-\tilde\Gamma^*\|\max_{\stackrel{0\le|\beta|<\kappa}{\beta ( \mathfrak{3} )\not=0}}
\|\tilde f.\mathsf{z}^\beta\|r^{|\beta|}.
\end{align*}
For the penultimate term, we appeal to (\ref{wr18}).
For the last term we use (\ref{wr08}) and appeal to our assumption (\ref{wr20}), which yields
\begin{align}\label{wr37}
\|\tilde f.\mathsf{z}^\beta\|\lesssim|\lambda|\quad\mbox{since}\;\beta ( \mathfrak{3} )\ge 1,
\end{align}
and combine it with 
the elementary $\max_{0\le|\beta|<\kappa}r^{|\beta|}$ $\le r^\kappa+1$
to the effect of
\begin{align*}
\lefteqn{\max_{|{\bf n}|<\kappa}\|(f-\tilde f).\mathsf{z}_{\bf n}\|r^{|{\bf n}|}}\nonumber\\
&\lesssim([f;\tilde f]_{\kappa,{\rm pol}}+|\lambda|\|\Gamma^*-\tilde\Gamma^*\|)r^\kappa
+(\|(f-\tilde f).\mathsf{z}_{\bf 0}\|+|\lambda|\|\Gamma^*-\tilde\Gamma^*\|)\nonumber\\
&\quad+|\lambda|\max_{|{\bf n}|<\kappa}\|(f-\tilde f).\mathsf{z}_{\bf n}\|r^{|{\bf n}|}
\max\{r^\kappa,r^\delta\}.
\end{align*}
Absorbing the last term to the l.h.s. we obtain
	\begin{align*}
		\lefteqn{\max_{|{\bf n}|<\kappa}\|(f-\tilde f).\mathsf{z}_{\bf n}\|r^{|{\bf n}|}}\nonumber\\
&\lesssim([f;\tilde f]_{\kappa,{\rm pol}}+|\lambda|\|\Gamma^*-\tilde\Gamma^*\|)r^\kappa
+(\|(f-\tilde f).\mathsf{z}_{\bf 0}\|+|\lambda|\|\Gamma^*-\tilde\Gamma^*\|)\nonumber\\
	&\quad\text{provided $|\lambda|\max\{r^\kappa,r^\delta\}\ll1$.}
	\end{align*}
Since as in the proof of Lemma~\ref{l:intI} the proviso is equivalent to $|\lambda|r^{\kappa}\ll1$,
choosing $r$ to be a small multiple of 
	\begin{align*}
		\Big( \min \Big\lbrace \frac{\|(f-\tilde f).\mathsf{z}_{\bf 0}\|+|\lambda|\|\Gamma^*-\tilde\Gamma^*\|}{[f;\tilde f]_{\kappa,{\rm pol}}+|\lambda|\|\Gamma^*-\tilde\Gamma^*\|} , \frac{1}{|\lambda|} 
		\Big\rbrace \Big)^{\frac{1}{\kappa}}
	\end{align*}
yields \eqref{wr16}.

\medskip

We finally address (\ref{wr36}). 
Following the definition (\ref{refonp19}) 
of $[f;\tilde f]_{\tilde{\kappa},{\rm pol}}$
we consider ${\bf n}$ with $|{\bf n}|<\tilde{\kappa}$. As before we distinguish $|h|\ge r$ and
$|h|\le r$ for some $r$ to be chosen later.
In the range $|h|\ge r$ we write
\begin{align*}
\lefteqn{(f_{x+h}.-f_x.Q_{\tilde{\kappa}}\Gamma_{x\,x+h}^*)
-(\tilde f_{x+h}.-\tilde f_x.Q_{\tilde{\kappa}}\tilde\Gamma_{x\,x+h}^*)}\nonumber\\
&=(f_{x+h}-\tilde f_{x+h}).-(f_x-\tilde f_x).Q_{\tilde{\kappa}}\Gamma_{x\,x+h}^*
-\tilde f_x.Q_{\tilde{\kappa}}(\Gamma_{x\,x+h}^*-\tilde\Gamma_{x\,x+h}^*)
\end{align*}
to deduce the estimate
\begin{align*}
\lefteqn{\big|\big((f_{x+h}.-f_x.Q_{\tilde{\kappa}}\Gamma_{x\,x+h}^*)
-(\tilde f_{x+h}.-\tilde f_x.Q_{\tilde{\kappa}}\tilde\Gamma_{x\,x+h}^*)\big)\mathsf{z}_{\bf n}\big|}\\
&\le\|(f-\tilde f).\mathsf{z}_{\bf n}\|
+\|\Gamma^*\|_{\tilde{\kappa}}\sum_{\stackrel{\beta\;\mbox{\tiny populated}}
{|{\bf n}|\le|\beta|<\tilde{\kappa}}}
\|(f-\tilde f).\mathsf{z}^\beta\||h|^{|\beta|-|{\bf n}|}\nonumber\\
&\quad+\|\Gamma^*-\tilde\Gamma^*\|_{\tilde{\kappa}}
\sum_{\stackrel{\beta\;\mbox{\tiny populated},\;\beta ( \mathfrak{3} )\not=0}{|{\bf n}|\le|\beta|<\tilde{\kappa}}}
\|\tilde f.\mathsf{z}^\beta\||h|^{|\beta|-|{\bf n}|}\\
&\stackrel{(\ref{wr37})}{\lesssim} |h|^{\tilde{\kappa}-|{\bf n}|}
\max_{\stackrel{\beta\;\mbox{\tiny populated}}{0\le|\beta|<\tilde{\kappa}}}
r^{|\beta|-\tilde{\kappa}}\big(
\|(f-\tilde f).\mathsf{z}^\beta\|+|\lambda|\|\Gamma^*-\tilde\Gamma^*\|_{\tilde{\kappa}}\big).
\end{align*}
In the range $|h|\le r$ we write
\begin{align*}
\lefteqn{(f_{x+h}.-f_x.Q_{\tilde{\kappa}}\Gamma_{x\,x+h}^*)
-(\tilde f_{x+h}.-\tilde f_x.Q_{\tilde{\kappa}}\tilde\Gamma_{x\,x+h}^*)}\nonumber\\
&=(f_{x+h}.-f_x.Q_{\kappa}\Gamma_{x\,x+h}^*)
-(\tilde f_{x+h}.-\tilde f_x.Q_{\kappa}\tilde\Gamma_{x\,x+h}^*)\nonumber\\
&\quad+(f_x-\tilde f_x).(Q_\kappa-Q_{\tilde{\kappa}})\Gamma_{x\,x+h}^*
+\tilde f_x.(Q_\kappa-Q_{\tilde{\kappa}})(\Gamma^*-\tilde\Gamma^*)_{x\,x+h}
\end{align*}
to infer the estimate
\begin{align*}
\lefteqn{\big|\big((f_{x+h}.-f_x.Q_{\tilde{\kappa}}\Gamma_{x\,x+h}^*)
-(\tilde f_{x+h}.-\tilde f_x.Q_{\tilde{\kappa}}\tilde\Gamma_{x\,x+h}^*)\big).\mathsf{z}_{\bf n}\big|
}\nonumber\\
&\le[f;\tilde f]_{\kappa,{\rm pol}}|h|^{\kappa-|{\bf n}|}
+\|\Gamma^*\|_{\kappa}\sum_{\stackrel{\beta\;\mbox{\tiny populated}}
{\tilde{\kappa}\le|\beta|<\kappa}}
\|(f-\tilde f).\mathsf{z}^\beta\||h|^{|\beta|-|{\bf n}|}\nonumber\\
&\quad+\|\Gamma^*-\tilde\Gamma^*\|_{\kappa}\sum_{\stackrel{\beta\;\mbox{\tiny populated},\;\beta ( \mathfrak{3} ) \neq 0}{\tilde{\kappa} \le |\beta| < \kappa}}
\|\tilde f.\mathsf{z}^\beta\||h|^{|\beta|-|{\bf n}|}\nonumber\\
&\stackrel{(\ref{wr37})}{\lesssim} |h|^{\tilde{\kappa}-|{\bf n}|}
\Big(r^{\kappa-\tilde{\kappa}}[f;\tilde f]_{\kappa,{\rm pol}}\nonumber\\
&\quad+\max_{\stackrel{\beta\;\mbox{\tiny populated}}{\tilde{\kappa}\le|\beta|<\kappa}}r^{|\beta|-\tilde{\kappa}}
\big(\|(f-\tilde f).\mathsf{z}^\beta\|
+|\lambda|\|\Gamma^*-\tilde\Gamma^*\|_{\kappa}\big)\Big).
\end{align*}
Combining both ranges we obtain
\begin{align}\label{wr38}
[f;\tilde f]_{\kappa,{\rm pol}}
&\lesssim
r^{\kappa-\tilde{\kappa}}[f;\tilde f]_{\kappa,{\rm pol}}\nonumber\\
&\quad+\max_{\stackrel{\beta\;\mbox{\tiny populated}}{0\le|\beta|<\kappa}}
r^{|\beta|-\tilde{\kappa}}\big(
\|(f-\tilde f).\mathsf{z}^\beta\|+|\lambda|\|\Gamma^*-\tilde\Gamma^*\|_{\kappa}\big).
\end{align}

\medskip

Estimate \eqref{wr38} 
should be seen as the analogue of \eqref{wr30}, and we now proceed as for \eqref{wr01}. 
Inserting the interpolation estimate \eqref{wr16} for $\|(f- \tilde{f}). \mathsf{z}_\mathbf{n}\|$ into \eqref{wr18}, we learn that
\begin{align}
\|(f-\tilde f).\mathsf{z}^\beta\|
&\lesssim |\lambda|^{\beta ( \mathfrak{3} )}\max_{{\bf n}:\beta({\bf n})\not=0}
([f;\tilde f]_{\kappa,{\rm pol}}+|\lambda|\|\dot f.\mathsf{z}_{\bf 0}\|
+|\lambda|\|\Gamma^*-\tilde\Gamma^*\|_{\kappa})^\frac{|{\bf n}|}{\kappa}\nonumber\\
&\quad\times (\|(f-\tilde f).\mathsf{z}_{\bf 0}\|
+|\lambda|\|\Gamma^*-\tilde\Gamma^*\|_{\kappa})^{1-\frac{|{\bf n}|}{\kappa}} \label{wr31b}
\end{align}
Provided $\beta$ is populated with $0\le|\beta|<\kappa$,
we claim that this implies
\begin{align}
\|(f-\tilde f).\mathsf{z}^\beta\|
&\lesssim 
([f;\tilde f]_{\kappa,{\rm pol}}+|\lambda|\|(f-\tilde f).\mathsf{z}_{\bf 0}\|
+|\lambda|\|\Gamma^*-\tilde\Gamma^*\|_{\kappa})^\frac{|\beta|}{\kappa}\nonumber\\
&\quad\times (\|(f-\tilde f).\mathsf{z}_{\bf 0}\|
+|\lambda|\|\Gamma^*-\tilde\Gamma^*\|_{\kappa})^{1-\frac{|\beta|}{\kappa}}. \label{wr32b}
\end{align}
Indeed, in case of $\beta ( \mathfrak{3} )=0$, and for ${\bf n}$ with $\beta({\bf n})\not=0$ we
must have $\beta=\delta_{\bf n}$ since $\beta$ is populated, so that (\ref{wr31b})
is identical to (\ref{wr32b}).
In case of $\beta ( \mathfrak{3} )\not=0$ we have $\beta ( \mathfrak{3} )\ge|\beta|/\kappa$ so that
by assumption (\ref{wr20}) $|\lambda|^{\beta ( \mathfrak{3} )}$ $\le|\lambda|^{|\beta|/\kappa}$
$\le 1$. Since for ${\bf n}$ with $\beta({\bf n})\not=0$ we have $|{\bf n}|$ $\le|\beta|$,
we may pass from (\ref{wr31b}) to (\ref{wr32b}) by Young's inequality.
Obviously the r.h.s.\ of \eqref{wr32b} is still an estimate of
$\|(f-\tilde f).\mathsf{z}^\beta\|$ $+|\lambda|\|\Gamma^*-\tilde\Gamma^*\|_{\kappa}$.
Inserting this into 
\eqref{wr38} and choosing
\begin{align*}
r=\Big(\frac{\|(f-\tilde f).\mathsf{z}_{\bf 0}\|+|\lambda|\|\Gamma^*-\tilde\Gamma^*\|_{\kappa}}
{[f;\tilde f]_{\kappa,{\rm pol}}+|\lambda|\|(f-\tilde f).\mathsf{z}_{\bf 0}\|
+|\lambda|\|\Gamma^*-\tilde\Gamma^*\|_{\kappa}}\Big)^\frac{1}{\kappa},
\end{align*}
we obtain (\ref{wr36}).
\end{proof}

We 
turn to
the proof of 
Lemma~\ref{l:pcII}.
\begin{proof}[Proof of Lemma~\ref{l:pcII}]
We start by 
taking the differences of the versions of the formulas (\ref{wr41}) and (\ref{wr42})
for $f$ and $\tilde{f}$, 
which implies for differences of modelled distributions that
	\begin{align}
		& \big( (f_{x+h}.-f_x.Q_\eta\Gamma_{x\,x+h}^*)
		- (\tilde{f}_{x+h}.-\tilde{f}_x.Q_\eta\tilde{\Gamma}_{x\,x+h}^*) \big) \mathsf{z}^{\beta+\delta_{\mathfrak{3}}} \nonumber \\
		& =\lambda \big( (f_{x+h}.-f_x.Q_{\eta_{\delta_{\mathfrak{3}}}}\Gamma_{x\,x+h}^*) - (\tilde{f}_{x+h}.-\tilde{f}_x.Q_{\eta_{\delta_{\mathfrak{3}}}}\tilde{\Gamma}_{x\,x+h}^*) \big) \mathsf{z}^{\beta}\,
 , \label{wr48b}
	\end{align}
and 
	\begin{align}
		& \big( (f_{x+h}.-f_x.Q_\eta\Gamma_{x\,x+h}^*)
		- (\tilde{f}_{x+h}.-\tilde{f}_x.Q_\eta\tilde{\Gamma}_{x\,x+h}^*) \big) \mathsf{z}^{\beta+\delta_{\mathbf{n}}} \nonumber \\
		& = ((f - \tilde{f})_{x+h}.\mathsf{z}^{\beta})
(f_{x+h}.-f_x.Q_{\eta_{\beta}}\Gamma_{x\,x+h}^*)\mathsf{z}_{\bf n} \nonumber \\
			& \quad +( \tilde{f}_{x+h}.\mathsf{z}^{\beta})
				\big( (f_{x+h}.-f_x.Q_{\eta_{\beta}}\Gamma_{x\,x+h}^*) - (\tilde{f}_{x+h}.-\tilde{f}_x.Q_{\eta_{\beta}}\tilde{\Gamma}_{x\,x+h}^*) \big)\mathsf{z}_{\bf n} \nonumber \\
			& \quad +\sum_{|\gamma|<\eta_{\beta}} (f_{x+h}.-f_x.Q_{\eta_{\gamma}}\Gamma_{x\,x+h}^*)\mathsf{z}^\beta
(\Gamma_{x\,x+h}^*)_\gamma^{\delta_{\bf n}} (f - \tilde{f})_x.\mathsf{z}^\gamma \nonumber \\
			& \quad +\sum_{|\gamma|<\eta_{\beta}} (f_{x+h}.-f_x.Q_{\eta_{\gamma}}\Gamma_{x\,x+h}^*)\mathsf{z}^\beta
((\Gamma - \tilde{\Gamma})_{x\,x+h}^*)_\gamma^{\delta_{\bf n}}\tilde{f}_x.\mathsf{z}^\gamma \nonumber \\
			& \quad +\sum_{|\gamma|<\eta_{\beta}} \big( (f_{x+h}.-f_x.Q_{\eta_{\gamma}}\Gamma_{x\,x+h}^*) - (\tilde{f}_{x+h}.-\tilde{f}_x.Q_{\eta_{\gamma}}\tilde{\Gamma}_{x\,x+h}^*) \big)\mathsf{z}^\beta \nonumber \\
			& \phantom{\quad +\sum_{|\gamma|<\eta_{\beta}}} \quad \times (\tilde{\Gamma}_{x\,x+h}^*)_\gamma^{\delta_{\bf n}}\tilde{f}_x.\mathsf{z}^\gamma , \label{wr49b}
	\end{align}
where in the above we denoted 
$\eta_{\beta} = \eta-(|\beta|-\alpha)$ for convenience.
As in the proof of Lemma \ref{l:pcI}, the identities (\ref{wr48b})
and \eqref{wr49b}
yield the inequalities
\begin{align}\label{wr53b}
[f; \tilde{f}]_{\eta,\beta+\delta_{\mathfrak{3}}}\le|\lambda|[f; \tilde{f}]_{\eta-(|\delta_{\mathfrak{3}}|-\alpha),\beta}
\le|\lambda|\max_{\tilde\eta\le\eta}[f; \tilde{f}]_{\tilde\eta,\beta}
\end{align}
and
	\begin{align*}
		& [f; \tilde{f}]_{\eta, \beta + \delta_{\mathbf{n}}}
		\leq 
			\| (f - \tilde{f}).\mathsf{z}^{\beta} \| [f]_{\eta-(|\beta|-\alpha),\delta_{\bf n}}
			+ \| \tilde{f}.\mathsf{z}^{\beta} \| [f; \tilde{f}]_{\eta-(|\beta|-\alpha),\delta_{\bf n}} \\
			& \quad + \sum_{|\gamma|<\eta-(|\beta|-\alpha)}
				\big( 
				[f]_{\eta-(|\gamma|-\alpha),\beta} \| \Gamma^* \|_{\eta} \| (f - \tilde{f}) . \mathsf{z}^{\gamma} \| \\
				& \quad \phantom{\quad + \sum_{|\gamma|<\eta-(|\beta|-\alpha)}
				\big(}
				+ [f]_{\eta-(|\gamma|-\alpha),\beta} \| \Gamma^* - \tilde{\Gamma}^* \|_{\eta} \| \tilde{f} . \mathsf{z}^{\gamma} \| \mathbf{1}_{\gamma ( \mathfrak{3} ) \neq 0} \\
				& \quad \phantom{\quad + \sum_{|\gamma|<\eta-(|\beta|-\alpha)}
				\big(}
				+ [f; \tilde{f}]_{\eta-(|\gamma|-\alpha),\beta} \| \Gamma^* \|_{\eta} \| \tilde{f} . \mathsf{z}^{\gamma} \|
				\big) ,
	\end{align*}
where the presence of the indicator function is a consequence of the last item in \eqref{mb07}.
By inequality \eqref{wr08} and our assumption \eqref{wr58}
the last inequality yields the estimate
\begin{align*}
&[f; \tilde{f}]_{\eta,\beta+\delta_{\bf n}}
\lesssim
\| (f - \tilde{f}).\mathsf{z}^{\beta} \| 
+ | \lambda |^{\beta ( \mathfrak{3} )} [f; \tilde{f}]_{\eta-(|\beta|-\alpha),\delta_{\bf n}}\\
&+\hspace{-.5ex}\max_{|\gamma|<\eta-(|\beta|-\alpha)} \hspace{-.5ex}
\big(
| \lambda |^{\beta ( \mathfrak{3} )} \| (f - \tilde{f}) . \mathsf{z}^{\gamma} \| 
+ | \lambda |^{\beta ( \mathfrak{3} ) + 1} \| \Gamma^* - \tilde{\Gamma}^* \|_{\eta} 
+ [f; \tilde{f}]_{\eta-(|\gamma|-\alpha),\beta} 
\big) ,
\end{align*}
where we also used 
Lemma~\ref{l:pcI}
and the fact
that 
$\sup_{\tilde{\kappa} \leq \kappa} \| f \|_{\tilde{\kappa}} \lesssim \|f\|_{\kappa} \lesssim 1$
as readily follows from the interpolation inequalities of Lemma~\ref{l:intI}.
Appealing to inequality (\ref{wr18})  
this estimate reduces to
\begin{align} 
& [f; \tilde{f}]_{\eta,\beta+\delta_{\bf n}}
\lesssim
\max_{\tilde\eta\le\eta}
[f; \tilde{f}]_{\tilde\eta,\beta} \nonumber \\
& \quad + |\lambda|^{\beta ( \mathfrak{3} )}
\big(\max_{\tilde\eta\le\eta}[f; \tilde{f}]_{\tilde\eta,{\rm pol}}+\max_{|{\bf m}|<\eta}\|(f - \tilde{f}).\mathsf{z}_{\bf m}\| + |\lambda| \| \Gamma^* - \tilde{\Gamma}^* \|_{\eta} \big) . \label{wr54b}
\end{align}
As in the proof of Lemma \ref{l:pcI}, 
(\ref{wr59}) follows from (\ref{wr53b}) and (\ref{wr54b}) by induction
in the plain length of $\beta$. 
\end{proof}

\subsection{Proof of Three-point Argument II}

In this subsection, we provide a proof of the final remaining ingredient needed in preparation for our proof of Theorem~\ref{uniqueness} which is our second three-point argument, Lemma~\ref{3_point_2}. Since the proof requires only a slight adaptation of the proof of Lemma~\ref{3_point}, we only provide a sketch of the details that change.

\begin{proof}[Proof of Lemma~\ref{3_point_2}]
	By taking the difference of \eqref{3pid} with the corresponding expression for $\tilde{R}$, we see that for any Schwartz function $\psi$ 
	\begin{align*}
		& \langle R_x - R_y - \tilde{R}_x + \tilde{R}_y, \psi \rangle
		\\ & = \sum_{0 \le |\beta| < \kappa} \left [ (f_y. - f_x. Q_\kappa \Gamma_{xy}^*)\mathsf{z}^\beta \langle \Pi_{y\beta}, \psi \rangle - (\tilde{f}_y. - \tilde{f}_x. Q_\kappa \tilde{\Gamma}_{xy}^*)\mathsf{z}^\beta \langle \tilde{\Pi}_{y\beta},\psi \rangle  \right ].
	\end{align*}
	In particular, by preceding as in the proof of Lemma~\ref{3_point}, we see that it suffices to bound the summand for $\beta$ not purely-polynomial with $\psi$ replaced by $\psi_{|y-x|}$, where $\psi$ is chosen to have the correct behaviour when tested against monomials.
	The desired bound is then immediate by writing the summand as
	\begin{align*}
		(f_y. - f_x.Q_\kappa \Gamma_{xy}^* - \tilde{f}_y. &+ \tilde{f}_x.Q_\kappa \tilde{\Gamma}_{xy}^*) \mathsf{z}^\beta \langle \Pi_{y\beta}, \psi_{|y-x|} \rangle \\ &+ (\tilde{f}_y. - \tilde{f}_x.Q_\kappa \tilde{\Gamma}_{xy}^*) \mathsf{z}^\beta \langle \Pi_{y \beta} - \tilde{\Pi}_{y\beta}, \psi_{|y-x|}\rangle .
	\end{align*}
\end{proof}

\subsection{Proof of Uniqueness and Continuity in Model Norm}
	We now turn to the proof of Theorem~\ref{uniqueness}. Since the overall strategy of proof is quite similar to that of Theorem~\ref{a priori}, we choose to be more brief here. 
 \begin{proof}[Proof of Theorem~\ref{uniqueness}]
Throughout this proof, we assume for convenience and without loss of generality that the torus is of unit size i.e.\ that $\ell=1$.

\medskip

\PfStart{un} 
\PfStep{un}{un:1} 
\textsc{Algebraic continuity.}
By our second Algebraic Continuity Lemma, Lemma~\ref{continuity_2} for populated multi-indices $\beta$ with $|\beta| < \kappa$ and $\beta \ne$pp, we have that
\begin{align*}
    [f; f]_{\kappa, \beta} \lesssim \lambda([f;\tilde{f}]_{\kappa, \mathrm{pol}} + \| (f - \tilde{f}) . \mathsf{z}_{\mathbf{0}} \| + \|\Gamma^* - \tilde{\Gamma}^*\|_\kappa).
\end{align*}

\PfStep{un}{un:2} 
\textsc{Reconstruction.}
We now note that if $S_x^- = R_x^- - \tilde{R}_x^-$ then we can write
\begin{align} \label{mc00}
  S_{x+h}^- - S_x^- 
    & = \sum_{0 < |\beta| < \kappa} (f_x. - f_{x+h}. Q_\kappa \Gamma_{x+h \, x}^* - \tilde{f}_x. + \tilde{f}_{x+h}. Q_\kappa \tilde{\Gamma}_{x+h \, x}^*)\Pi_{x\beta}^- 
   	\nonumber \\
    & \qquad - (\tilde{f}_x. - \tilde{f}_{x+h}. Q_\kappa \tilde{\Gamma}_{x+h \, x}^*) (\tilde{\Pi}_{x \beta}^- - \Pi_{x\beta}^-).
\end{align}
In the same way that we obtain the bound \eqref{ap02} in the corresponding step in the proof of Theorem~\ref{a priori}, this implies that for 
$\mu \leq \ell = 1$,
\begin{align*}
    & \frac{|(S_{x+h}^- - S_x^-)_{\mu}(x)|}{\mu^{\kappa^- - 2} (\mu + |h|)^{\kappa - \kappa^-} } \\
    & \lesssim |\lambda|([f;\tilde{f}]_{\kappa, \mathrm{pol}} + \| ( f - \tilde{f} ) . \mathsf{z}_{\mathbf{0}} \| + \|\Gamma^* - \tilde{\Gamma}^* \|_\kappa + \|\Pi -\tilde{\Pi}\|_\kappa) 
\end{align*}
where we made use of the fact that $[\tilde{f}]_{\beta, \kappa} \lesssim |\lambda|$ by the a priori estimate given in Theorem~\ref{a priori} in order to obtain the prefactor $|\lambda|$ in front of the term in $\|\Pi - \tilde{\Pi}\|_\kappa$.
\medskip

By the Reconstruction Theorem given in Lemma~\ref{Recon} we then obtain for 
$\mu \leq \ell =1$
\begin{align} \label{ap03d}
  & |(R_x^- - \tilde{R}_x^-)_\mu(x)| \nonumber \\
  & \lesssim |\lambda| ([f;\tilde{f}]_{\kappa, \mathrm{pol}} + \| ( f - \tilde{f} ) . \mathsf{z}_{\mathbf{0}} \| + \|\Gamma^* - \tilde{\Gamma}^* \|_\kappa + \|\Pi - \tilde{\Pi}\|_\kappa)
  \mu^{\kappa- 2}
\end{align}
uniformly over $x \in \mathbb{R}^{1+d}$, $\mu \leq 1$ and $\psi$ with $\|\psi\|_p \le 1$ where $p$ is as in the statement of Lemma~\ref{Recon}.
which is the output of this step of the proof.

\PfStep{un}{un:3} 
\textsc{Integration.} 
As in the proof of Theorem~\ref{a priori}, it is again necessary to include a localisation before applying our Schauder estimate in order to deal with the restriction to scales at most 
$1$
in the output of our Reconstruction step.
As in the corresponding step of the proof of Theorem~\ref{a priori}, we observe that there exists a polynomial $P_x$ of degree strictly less than $\kappa$ such that
\begin{align*}
    LS_x = S_x^- + P_x ,
    \qquad 
    \text{where}
    \qquad P_x ( \cdot ) = \sum_{\kappa - 2 < | \mathbf{n} | < \kappa} c_{x\mathbf{n}} ( \cdot - x )^{\mathbf{n}} ,
\end{align*}
for some coefficients $c_{x\mathbf{n}} $, 
where arguing as for \eqref{apn10} 
reveals that
	\begin{align}\label{apn10d}
		\sup_{x \in \mathbb{R}^{1 + d}} | c_{x\mathbf{n}} |
		& \lesssim |\lambda | ( \| f; \tilde{f} \|_{\kappa} + \| \Pi - \tilde{\Pi} \|_{\kappa} ) .
	\end{align}

We now claim that
with $\delta>0$ as defined in \eqref{is03} we have
	\begin{align} \label{apn11d}
		& |(L S_x)_\mu (x) | \nonumber \\
		& \lesssim \big( \rho^{- \delta} + \rho^{p} | \lambda | \big) \big( \| f - \tilde{f} \|_{\kappa} + \| \Gamma^* - \tilde{\Gamma}^* \|_{\kappa} + \| \Pi - \tilde{\Pi} \|_{\kappa} \big) \, \mu^{\kappa - 2} ,
	\end{align}
where the implicit constant is uniform over $x \in \mathbb{R}^{1 + d}$, $\mu < \infty$, $\| \psi \|_{p} \leq 1$ and $\rho \geq 1$.
Again, we distinguish the regimes $\mu \lessgtr \rho$.
When $\mu \geq \rho$ we use the expression 
	\begin{align*}
     S_{x\mu}(x) = & [(f_{\cdot} - f_x - \tilde{f}_\cdot + \tilde{f}_x ).\mathsf{z}_\mathbf{0}] \ast \psi_\mu(x) \\ & - \sum_{\kappa^- \le |\beta| < \kappa} (f_x. \mathsf{z}^\beta - \tilde{f}_x. \mathsf{z}^\beta) \Pi_{x\beta \mu}(x) + \tilde{f}_x.\mathsf{z}^\beta (\Pi_{x\beta} - \tilde{\Pi}_{x\beta})_\mu(x),
 \end{align*}
which implies
	\begin{align*}
		|(L S_x)_\mu (x) |
		& \lesssim (\|f; \tilde{f}\|_{\kappa} + \|\Pi - \tilde{\Pi}\|_{\kappa}) \sum_{0 \leq |\beta|<\kappa} \mu^{| \beta | - 2} \\
		& \lesssim  (\|f; \tilde{f}\|_{\kappa} + \|\Pi - \tilde{\Pi}\|_{\kappa}) \rho^{- \delta} \, \mu^{\kappa - 2} ,
	\end{align*}
as desired.
On the other hand, when $\mu \leq \rho$ we write 
	\begin{align*}
		(L S_x)_\mu (x)
		& = S_x^{-} * ( \psi_{\rho} )_{\frac{\mu}{\rho}} ( x ) + P_{x \mu} ( x ) .
	\end{align*}
Since \eqref{apn10d} implies in this regime
	\begin{align*}
		| P_{x \mu} ( x ) | 
		& \lesssim \rho^p |\lambda | ( \| f; \tilde{f} \|_{\kappa} + \| \Pi - \tilde{\Pi} \|_{\kappa} ) \, \mu^{\kappa - 2} ,
	\end{align*}
and \eqref{ap03d} implies
	\begin{align*}
		|  S_x^{-} * ( \psi_{\rho} )_{\frac{\mu}{\rho}} |
		& \lesssim \rho^p | \lambda |
		(\|f;\tilde{f}\|_{\kappa} + \|\Gamma^* - \tilde{\Gamma}^* \|_\kappa + \|\Pi - \tilde{\Pi}\|_\kappa) \, \mu^{\kappa- 2} ,
	\end{align*}
the desired 
\eqref{apn11d} follows.

\medskip
Applying our Schauder estimate of Lemma~\ref{Schauder} 
with \eqref{apn11d} as input,
we conclude that there exists a $p^\prime$ such that for any $\rho \ge 1$,
 \begin{align}\label{mc03}
 	| S_{x \mu} ( x ) | 
		& \lesssim 
		\big( \rho^{- \delta} + \rho^{p} | \lambda | \big) \big( \| f - \tilde{f} \|_{\kappa} + \| \Gamma^* - \tilde{\Gamma}^* \|_{\kappa} + \| \Pi - \tilde{\Pi} \|_{\kappa} \big) \, \mu^{\kappa}  
    \end{align}
uniformly over $x \in \mathbb{R}^{1 + d}$, $\mu < \infty$, $\| \psi \|_{p^{\prime}} \leq 1$ and $\rho \geq 1$.

\PfStep{un}{un:4} 
\textsc{Three-Point Argument.}
As in the analogous step in the proof of Theorem~\ref{a priori}, we deduce from \eqref{mc03}, Lemma~\ref{continuity_2} and Lemma~\ref{3_point_2} that
    \begin{align}\label{mc04}
        [f; \tilde{f}]_{\kappa, \mathrm{pol}} 
        \lesssim (\rho^{-\delta} + \rho^{p} |\lambda|) \, \|f; \tilde{f}\|_{\kappa} + \|\Gamma^* - \tilde{\Gamma}^*\|_\kappa + \|\Pi - \tilde{\Pi}\|_\kappa .
    \end{align}

\PfStep{un}{un:5} 
\textsc{Incorporating Boundary Data.}
We begin this step by noting that as a consequence of the algebraic continuity result given in Lemma~\ref{continuity_2} and the interpolation result in Lemma~\ref{l:intII}, we have that
\begin{align}\label{mc10}
    \|f; \tilde{f}\|_\kappa \lesssim [f; \tilde{f}]_{\kappa, \mathrm{pol}} + \| ( f - \tilde{f} ) . \mathsf{z}_{\mathbf{0}} \| + \|\Pi - \tilde{\Pi}\|_\kappa + \|\Gamma^* - \tilde{\Gamma}^*\|_{\kappa}.
\end{align}
We now wish to remove the appearance of 
$\|(f- \tilde{f}).\mathsf{z}_{\mathbf{0}} \|$ 
from the right-hand side. 
To this end, we note that since $f. \mathsf{z}_\mathbf{0}, \tilde{f}. \mathsf{z}_\mathbf{0}$ are both periodic and of vanishing space-time average, 
we can write
\begin{align} \label{mc05}
   \|(f- \tilde{f}).\mathsf{z}_{\mathbf{0}} \|
    & \lesssim [f - \tilde{f}]_{\kappa^-} 
    = [f; \tilde{f}]_{\kappa^-, \delta_\mathbf{0}} 
    \\ \nonumber
    & \lesssim
    \big([f;\tilde f]_{\kappa,{\rm pol}}+|\lambda|\|(f-\tilde f).\mathsf{z}_{\bf 0}\|
+|\lambda|\|\Gamma^*-\tilde\Gamma^*\|\big)^{\frac{\kappa^{-}}{\kappa}}\nonumber\\
&\quad\times\big(\|(f-\tilde f).\mathsf{z}_{\bf 0}\|
+|\lambda|\|\Gamma^*-\tilde\Gamma^*\|\big)^{1-\frac{\kappa^{-}}{\kappa}}, \nonumber
\end{align}
where the second line followed from the use of the interpolation estimate \eqref{wr36}. 
By appealing to Young's inequality in the same manner as in the analogous step in the proof of Theorem~\ref{a priori}, we obtain that
\begin{align}\label{mc09}
    \|(f- \tilde{f}).\mathsf{z}_{\mathbf{0}} \| \lesssim [f; \tilde{f}]_{\kappa, \mathrm{pol}} + \|\Gamma^* - \tilde{\Gamma}^*\|_\kappa.
\end{align}

Substituting \eqref{mc09} into the right-hand side of \eqref{mc10} gives us the estimate
\begin{align*}
    \|f; \tilde{f}\|_\kappa & \lesssim [f; \tilde{f}]_{\kappa, \mathrm{pol}} + \|\Gamma^* - \tilde{\Gamma}^*\|_\kappa + \|\Pi - \tilde{\Pi}\|_\kappa
    \\ & \lesssim (\rho^{-\delta} + \rho^{p} |\lambda|) \|f;\tilde{f}\|_\kappa + \|\Gamma^* - \tilde{\Gamma}^* \|_\kappa + \|\Pi - \tilde{\Pi}\|_\kappa
\end{align*}
where the second line follows by use of the output \eqref{mc04} of the previous step. Taking $\rho$ large and then $|\lambda|$ small in the same manner as in Theorem~\ref{a priori} concludes the proof.
 \end{proof}

\section{Pathwise Existence for Fixed (Deterministic) Model with Qualitative Smoothness} \label{ss:exist}

\subsection{Proof of Consistency}\label{ss:lift}

\begin{proof}[Proof of Lemma~\ref{l:c}]
\PfStart{c1} 
\PfStep{c1}{c1:1} 
For the convenience of the reader, we first recall the asumptions of the lemma and spell out the tasks to perform.
We are given a smooth model, i.e.\ as in Definition~\ref{d:smooth}.
We are also given $\lambda \in \mathbb{R}$ and a smooth $\ell$-periodic function $u$ satisfying
the PDE \eqref{t30}, which we rewrite as
	\begin{align}\label{t21}
		L u = P u^- , \qquad \fint u = 0, \qquad u^- \coloneqq \lambda u^3 + h_{\lambda} u + \xi ,
	\end{align}
where $h_{\lambda} = \sum_k c_k \lambda^k$.
We shall prove that $f$ satisfies the robust formulation Definition~\ref{rf03}.
For that purpose, inspecting Definition~\ref{rf03}, it remains to prove:
\begin{enumerate}
	\item that the map $x \mapsto f_x$ defined by \eqref{rf02} with the choices \eqref{t9}-\eqref{t9b} for $\mathsf{z}_\mathbf{n}$ and $\mathsf{z}_\mathfrak{3}$ is well-defined and periodic (see Step~\ref{c1:2} below); 
	\item that the quantities $[f]_{\kappa, \beta}$, $\| f . \mathsf{z}^{\beta} \|$ are finite for all $| \beta | < \kappa$ (see Steps~\ref{c1:3} and \ref{c1:6} below); 
	\item that the germs $R$ and $R^-$ defined by \eqref{t12} satisfy \eqref{t13} (see Step~\ref{c1:5b} below).
\end{enumerate}

\PfStep{c1}{c1:2}
We first emphasize that
\eqref{t9}-\eqref{t9b} 
in combination with 
the form \eqref{rf02} 
provides an unambiguous definition of $f$.
We first observe from 
the definition \eqref{t45} of homogeneity 
and the range \eqref{t53} of $\alpha$ 
that
	\begin{align} \label{t54}
		\text{for non-purely-polynomial $\beta$}, 
		\quad \mathsf{z}^{\beta} = \mathsf{z}_{\mathfrak{3}}^{\beta ( \mathfrak{3} )} \prod_{|\mathbf{n}| < |\beta|} \mathsf{z}_{\mathbf{n}}^{\beta ( \mathbf{n} )} .
	\end{align}	
Thus the multiplicativity
\eqref{rf02}
of $f$ implies
	\begin{align}\label{t55}
	\text{for non-purely-polynomial $\beta$}, 
	\quad f_x . \mathsf{z}^{\beta} = \lambda^{\beta ( \mathfrak{3} )} \prod\nolimits_{| \mathbf{n} | < | \beta |} ( f_x . \mathsf{z}_{\mathbf{n}} )^{\beta ( \mathbf{n} )} .
	\end{align}
In particular,
$f_x . \mathsf{z}^{\beta}$
is a function of $(f_x . \mathsf{z}_{\mathbf{n}})_{| \mathbf{n} | < |\beta |}$ only, 
so that the definition \eqref{t9}-\eqref{t9b} is not circular.
Now, the periodicity of $x \mapsto f_x$ directly follows from
that definition
along with the assumption that $u$ is periodic and the (joint) periodicity assumption \eqref{t15} on $\Pi$.

\PfStep{c1}{c1:3} 
In this step, we claim that $\| f . \mathsf{z}^{\beta} \| < \infty$ for all $| \beta |<\kappa$.
Our argument is recursive in the homogeneity $| \beta |$.
The base case of $\beta = 0$ follows by \eqref{rf02} which enforces $f_x . \mathbf{1} = 1$.
We turn to the induction step, and consider $\beta \neq 0$.
We distinguish the cases whether $\beta$ is pp or not.
If $\beta \neq \text{pp}$, the desired claim follows from \eqref{t55} and the induction hypothesis in the case where $\beta$ is purely polynomial.
It thus remains to consider the case $\beta = \delta_{\mathbf{n}}$.
Since by the definition \eqref{t9b}, $f_x.\mathsf{z}_\mathbf{n} = 0$ for $|\mathbf{n}| > \kappa$ the claim follows in this case.
Thus, we let $| \mathbf{n} | < \kappa$, where $f_x . \mathsf{z}_{\mathbf{n}}$ is defined by \eqref{t10}, which implies 
$\| f . \mathsf{z}_{\mathbf{n}} \| \leq \| \partial^{\mathbf{n}} u \| + \sum_{| \beta | < | \mathbf{n} |} \| f . \mathsf{z}^{\beta} \| \, \sup_{x \in \mathbb{R}^{1 + d}} | \partial^{\mathbf{n}} \Pi_{x \beta} ( x ) |$.
Since $u$ is periodic and $\kappa$-H\"older, $\| \partial^{\mathbf{n}} u \| < \infty$.
By induction, $\| f . \mathsf{z}^{\beta} \|< \infty$ in this sum.
Finally, 
	\begin{align}
		\sup_{x \in \mathbb{R}^{1 + d}} | \partial^{\mathbf{n}} \Pi_{x \beta} ( x ) |
		& = \sup_{| x | \leq \sqrt{D} \ell} | \partial^{\mathbf{n}} \Pi_{x \beta} ( x ) | \label{t94} \\
		& = \sup_{| x | \leq \sqrt{D} \ell} \Big| \sum_{| \gamma | \leq | \beta |} ( \Gamma_{x 0}^* )^{\gamma}_{\beta} \partial^{\mathbf{n}} \Pi_{0 \gamma} ( x ) \Big|
		< \infty , \nonumber
	\end{align}
where in the first equality we used the (joint) periodicity \eqref{t15} of $\Pi$, 
in the second equality we used the reexpansion property \eqref{mb04},
and we then uniformly bounded $\Gamma$ by \eqref{mb03} and $\partial^{\mathbf{n}}\Pi_{0 \gamma}$ by the $\kappa$-H\"older assumption.
This establishes the desired claim.

\PfStep{c1}{c1:4} 
In this step, we claim that for all 
$x \in \mathbb{R}^{1 + d}$ and 
$| \mathbf{n} | < \kappa$
	\begin{align}\label{t16}
		\partial^{\mathbf{n}} R_x ( x ) = 0,  
		\qquad R_x^{-} ( x ) = 0 .
	\end{align}
We let $| \mathbf{n} | < \kappa$ 
and start with the first identity in \eqref{t16}.
Since $\Pi_{x \beta}$ is $\kappa$-H\"older
it follows from standard properties of mollification that for all $|\mathbf{n}| < \kappa$, $\partial^\mathbf{n} \Pi_{x \beta} ( x ) = \lim_{r \to 0} \partial^\mathbf{n}\Pi_{x \beta r}( x )$
(provided $\int \psi = 1$). 
Thus, in view of the model bound \eqref{mb01} and the property \eqref{t93} of the homogeneity,
we learn:
	\begin{align}\label{t35}
		\text{$\partial^{\mathbf{n}} \Pi_{x \beta} ( x ) = 0$ unless $| \beta | < | \mathbf{n} |$ or $\beta = \mathbf{n}$,}
	\end{align}
so that 
	\begin{align} \label{t95} 
		Q_\kappa \partial^{\mathbf{n}} \Pi_x ( x ) = Q_{\mathbf{n}} \partial^{\mathbf{n}} \Pi_x ( x ) +  \mathbf{n}! \, \mathsf{z}_{\mathbf{n}} .
	\end{align}
Substituting into the definition \eqref{t12} of $R$,
we deduce that
	\begin{align*}
		\partial^{\mathbf{n}} R_x ( x ) 
			= \partial^{\mathbf{n}} u ( x ) - f_x . Q_{| \mathbf{n} |} \partial^{\mathbf{n}} \Pi_x ( x ) - \mathbf{n}! \, f_x . \mathsf{z}_{\mathbf{n}} , 
	\end{align*}
which vanishes as desired by the definition \eqref{t10} of $f_x . \mathsf{z}_{\mathbf{n}}$.

\medskip
We turn to the second identity in \eqref{t16}.
Again by the model bound 
\eqref{mb01}, 
we learn
	\begin{align} \label{t75}
		\Pi_{x \beta}^- ( x ) = 0 \text{ unless } | \beta | \leq 2 ,
	\end{align}
so that 
	\begin{align} \label{t96}
		Q_{\kappa} \Pi_x^- ( x ) = \Pi_x^- ( x ) = \mathsf{z}_{\mathfrak{3}} \Pi_x ( x )^3 + c \Pi_x ( x ) + \xi ( x ) \mathsf{1} .
	\end{align}
By the form \eqref{rf02} of $f$ and the first identity in \eqref{t16}, we deduce that
	\begin{align*}
		f_x . Q_{\kappa} \Pi_x^- ( x )
		= \lambda u ( x )^3 + h_{\lambda} u ( x ) + \xi ( x ) = u^{-} ( x ) ,
	\end{align*}
so that we have established $R_x^- ( x ) = 0$ after reporting in the definition \eqref{t12} of $R^-$.

\PfStep{c1}{c1:5} 
In this step, we post-process \eqref{t16} into
	\begin{align}
		\sup_{x \in \mathbb{R}^{1 + d} , \, | h | \leq \ell , \, | \mathbf{n} | < \kappa } 
			\frac{| \partial^{\mathbf{n}} R_x ( x + h ) |}{| h |^{\kappa - | \mathbf{n} |}} 
			+ \frac{| R_x^- ( x + h ) |}{| h |^{\kappa^-}} 
		 	< \infty . \label{t44}
	\end{align}
We only prove the claim for $R$ and $\mathbf{n} = 0$, since the cases of $\mathbf{n} \neq 0$ and $R^-$ are obtained in exactly the same way.
To that effect,
we first recall that since $u$ is periodic and $\Pi$ is (jointly) periodic \eqref{t15}, also $(x, y) \mapsto R_x ( y )$ is periodic, so that in \eqref{t44} one may restrict to $|x| \leq \sqrt{D} \ell$ without changing the value of the supremum.
Now we use \eqref{t16} to rewrite for $|x|, | h | \leq \sqrt{D} \ell$,
	\begin{align*}
		R_x ( x + h )
		& = R_x ( x + h ) - \sum_{|\mathbf{n}| < \kappa} \frac{1}{\mathbf{n} !} \partial^{\mathbf{n}} R_x ( x ) \, h^{\mathbf{n}} .
	\end{align*}
Appealing to the explicit expression \eqref{t12} of $R$ along with the reexpansion property \eqref{mb04} of the model, 
this implies that
	\begin{align*}
		& R_x ( x + h )
		= u ( x + h ) - \sum_{|\mathbf{n}| < \kappa} \frac{1}{\mathbf{n} !} \partial^{\mathbf{n}} u ( x ) \, h^{\mathbf{n}} \\
			& \quad - \sum_{| \gamma | \leq | \beta | < \kappa} f_x . \mathsf{z}^{\beta} ( \Gamma_{x 0}^* )_{\beta}^{\gamma} \Big( \Pi_{0 \gamma} ( x + h ) - \sum_{|\mathbf{n}| < \kappa} \frac{1}{\mathbf{n} !} \partial^{\mathbf{n}} \Pi_{0 \gamma} ( x ) \, h^{\mathbf{n}} \Big) .
	\end{align*}
The desired conclusion follows by Taylor's Theorem from the assumption that $u$ and $\Pi_{0 \gamma}$ are $\kappa$-H\"older, the model bound \eqref{mb03} on $\Gamma^*$, and the output of Step~\ref{c1:3}.

\PfStep{c1}{c1:5b}
In this step, we post-process \eqref{t44} into the desired vanishing \eqref{t13}.
Once again, we only prove the first item in \eqref{t13}, since the claim on $R^-$ is established in exactly the same way.
To that effect, let $x \in \mathbb{R}^{1 + d}$ and $\psi \in \mathcal{S}$ 
be the Schwartz function defined in \eqref{t90}.
Fix a smooth cut-off $\chi$ i.e.\
	\begin{align*}
		\chi \in C_c^{\infty} ( \mathbb{R}^{1 + d} ) , \qquad \mathbf{1}_{B_{\ell/2}(0)} \leq \chi \leq \mathbf{1}_{B_\ell(0)} .
	\end{align*}
Then, we write
	\begin{align*}
		R_{x \mu} ( x ) 
		& = \int d h \, R_x ( x - h ) \chi ( h ) \psi_{\mu} ( h ) 
		+ R_x * ( (1 - \chi) \psi_{\mu} ) ( x ) .
	\end{align*}
The first term in the r.h.s.\ is a $O ( \mu^{\kappa} ) = o_{\mu \to 0} ( 1 )$ 
uniformly in $x \in \mathbb{R}^{1 + d}$
by \eqref{t44}.
On the other hand,
one readily checks that $((1 - \chi) \psi_{\mu})_{\mu \to 0}$ converges to zero in the Schwartz class $\mathcal{S}$.
This implies that the second term in the r.h.s.\
is also a $o_{\mu \to 0} ( 1 )$ uniformly in $x$
after noting that
by definition of $R$, the model estimates \eqref{mb01} and e.g.\ the change of kernel result of Lemma~\ref{kernel_swap}, one has for some $p \in \mathbb{N}$,
$\sup_{\| \psi \|_p \leq 1} \sup_{x \in \mathbb{R}^d} | R_{x \, \mu = \ell} ( x ) | < \infty$.
This establishes the desired \eqref{t13}.

\PfStep{c1}{c1:6}
In this step, we claim that for all populated $\beta$ with $|\beta|<\kappa$, $[f]_{\kappa, \beta}< \infty$.
In combination with the output of Step~\ref{c1:3}, this then yields the desired $\| f \|_{\kappa} < \infty$.
As in Step~\ref{c1:3} above, our argument is inductive in the homogeneity $|\beta|$.
More precisely, our induction hypothesis is that $\sup_{0 \leq \eta \leq \kappa} [f]_{\eta, \beta} < \infty$.
In fact, by 
the interpolation inequality of Lemma~\ref{l:intI}
and the output of Step~\ref{c1:3}, this is equivalent to $[f]_{\kappa, \beta} < \infty$.
The base case $\beta = 0$ of the induction follows from \eqref{rf02} which implies that $[f]_{\eta, 0} = 0$.
We now turn to the inductive step, and consider $\beta \neq 0$.
We disinguish the case where $\beta$ is pp from the one where it is not.
If $\beta \neq \text{pp}$ the desired claim follows from the induction hypothesis and the output of Step~\ref{c1:3} by
the output of Lemma~\ref{l:pcI}.
It remains to consider the case $\beta = \delta_{\mathbf{n}}$.
When $|\mathbf{n}| > \kappa$, one has $[f]_{\kappa, \delta_{\mathbf{n}}}=0$ by definition so that we consider $| \mathbf{n} | < \kappa$ in what follows.
We use the expansion
	\begin{align}\label{t32}
		R_{x + h} ( \cdot ) - R_x ( \cdot )
		& = - \sum\limits_{\substack{0<| \beta | < \kappa \\ \beta \neq \text{pp}}} \big( f_{x + h} . - f_x . Q_{\kappa} \Gamma_{x \, x+ h}^* ) \mathsf{z}^{\beta} \, \Pi_{x + h \, \beta} ( \cdot) \\
		& \quad - \sum_{| \mathbf{n} | < \kappa} \big( f_{x + h} . - f_x . Q_{\kappa} \Gamma_{x \, x+ h}^* ) \mathsf{z}_{\mathbf{n}} \, ( \cdot - x - h)^{\mathbf{n}} \nonumber .
	\end{align}
Applying $\partial^{\mathbf{n}}$ and evaluating at $x + h$, we learn by \eqref{t35} and \eqref{t16} that 
	\begin{align}
		& \big( f_{x + h} . - f_x . Q_{\kappa} \Gamma_{x \, x+ h}^* ) \mathsf{z}_{\mathbf{n}} \nonumber \\
		& = \partial^{\mathbf{n}} R_x ( x + h ) 
			- \sum\limits_{\substack{| \beta | < | \mathbf{n} | \\ \beta \neq \text{pp}}} \big( f_{x + h} . - f_x . Q_{\kappa} \Gamma_{x \, x+ h}^* ) \mathsf{z}^{\beta} \, \partial^{\mathbf{n}} \Pi_{x + h \, \beta} ( x + h ) . \nonumber
	\end{align}
Now appealing to \eqref{t44}, the induction hypothesis, and \eqref{t94}, we deduce
	\begin{align}
		\sup_{x \in \mathbb{R}^{1 + d}, | h | \leq \ell} \frac{\big| \big( f_{x + h} . - f_x . Q_{\kappa} \Gamma_{x \, x+ h}^* ) \mathsf{z}_{\mathbf{n}} \big|}{| h |^{\kappa - | \mathbf{n} |}}
		& < \infty . \nonumber
	\end{align}
By 
the triangle inequality,
the output of Step~\ref{c1:3} 
and \eqref{mb03}, 
also the supremum over $| h | > \ell$ is finite, so that we have established $[ f ]_{\kappa, \delta_{\mathbf{n}}} < \infty$ as desired.
\end{proof}

\subsection{Proof of Consistency for Linearisation}

\begin{proof}[Proof of Lemma~\ref{l:c2}]
\PfStart{c3} 
\PfStep{c3}{c3:1} 
We follow the same strategy as for the proof of the Consistency Lemma, Lemma~\ref{l:c}.
We remain brief when the arguments are similar, 
but present the details when they are different.
Our task is to prove:
\begin{enumerate}
	\item that the map $x \mapsto \dot{f}_x$ defined by the form \eqref{t49} with \eqref{t50}-\eqref{t50b} is well-defined and periodic (see Step~\ref{c3:2} below); 
	\item that the quantities $[\dot{f}]_{\kappa, \beta}$, $\| \dot{f} . \mathsf{z}^{\beta} \|$ are finite for all $| \beta | < \kappa$ (see Steps~\ref{c3:3} and \ref{c3:6} below); 
	\item that the germs $\dot{R}$ and $\dot{R}^-$ defined by \eqref{t12l} satisfy \eqref{t13l} (see Step~\ref{c3:5} below).
\end{enumerate}

\PfStep{c3}{c3:2}
We first emphasize that
$\dot{f}$ is well-defined.
Indeed,
\eqref{t54}
in combination with
the Leibniz rule
\eqref{t49},
implies that 
	\begin{align}\label{t55b}
	\text{for non-purely-polynomial $\beta$}, 
	\quad \dot{f}_x . \mathsf{z}^{\beta} = \sum\nolimits_{| \mathbf{n} | < | \beta |} ( \dot{f}_x . \mathsf{z}_{\mathbf{n}} ) \, (f_x.\partial_{\mathsf{z}_{\mathbf{n}}} \mathsf{z}^{\beta}) ,
	\end{align}
so that the definition \eqref{t50}-\eqref{t50b} is not circular.
The periodicity of $x \mapsto \dot{f}_x$ follows similarly by that of $\dot{u}$ along with \eqref{t15}.

\PfStep{c3}{c3:3} 
In this step, we claim that $\| \dot{f} . \mathsf{z}^{\beta} \| < \infty$ for all $| \beta |<\kappa$.
As in the proof of Lemma~\ref{l:c} above, we argue by recursion in $|\beta|$.
When $\beta = 0$ the claim follows from $\dot{f}_x . \mathbf{1} = 0$.
In the induction step, the case $\beta = \text{pp}$ proceeds as in the proof of Lemma~\ref{l:c}, recall the definition \eqref{t51} of $\dot{f}.\mathsf{z}_{\mathbf{n}}$.
Finally, the case $\beta \neq \text{pp}$ follows directly from \eqref{t55b}.

\PfStep{c3}{c3:4} 
In this step, we claim that for all 
$x \in \mathbb{R}^{1 + d}$ and 
$| \mathbf{n} | < \kappa$
	\begin{align}\label{t16b}
		\partial^{\mathbf{n}} \dot{R}_x ( x ) = 0, 
		\qquad \dot{R}_x^{-} ( x ) = 0 .
	\end{align}
The first item in \eqref{t16b} is obtained from \eqref{t95} as in the proof of Lemma~\ref{l:c} above, and follows from (in fact, motivates) the definition \eqref{t51} of $\dot{f}$.
Turning to the second identity in \eqref{t16b}, we deduce from \eqref{t96} in combination with the Leibniz rule \eqref{t49} and the first identity in \eqref{t16b} that
	\begin{align*}
		\dot{f}_x.Q_{\kappa}\Pi_x^- ( x )
		& = 3 \lambda \, ( f_x. \Pi_x ( x ))^2 \dot{f}_x . \Pi_x ( x ) + h_{\lambda} \dot{f}_x . \Pi_x ( x ) \\
		& = ( 3 \lambda u ( x )^2 + h_{\lambda} ) \dot{u} ( x ) = \dot{u}^- (x) .
	\end{align*}
This is the desired $\dot{R}_x^- ( x ) = 0$.

\PfStep{c3}{c3:5} 
Arguing exactly as in the Steps~\ref{c1:5} and \ref{c1:5b} of the proof of Lemma~\ref{l:c} above
(simply replacing all quantities with their dotted versions), 
the output of Step~\ref{c3:4} can be post-processed to
	\begin{align*}
		\sup_{x \in \mathbb{R}^{1 + d} , \, | h | \leq \ell , \, | \mathbf{n} | < \kappa } 
			\frac{| \partial^{\mathbf{n}} \dot{R}_x ( x + h ) |}{| h |^{\kappa - | \mathbf{n} |}} 
			+ \frac{| \dot{R}_x^- ( x + h ) |}{| h |^{\kappa^-}} 
		 	< \infty ,
	\end{align*}
and in turn into \eqref{t13l}.

\PfStep{c3}{c3:6}
In this step, we claim that for all populated $\beta$ with $|\beta|<\kappa$, $[\dot{f}]_{\kappa, \beta}< \infty$.
In combination with the output of Step~\ref{c1:3}, this then yields the desired $\| \dot{f} \|_{\kappa} < \infty$. 
As in the corresponding step in the proof of Lemma~\ref{l:c}, we argue inductively in the homogeneity $|\beta|$,
and we distinguish whether 
$\beta$ is pp, or not.
In the former case, we argue exactly as in the proof of Lemma~\ref{l:c}.
In the latter case, the argument is also very similar, the only difference being that we appeal to
Lemma~\ref{l:pcIII} instead of Lemma~\ref{l:pcI}.
\end{proof}

\subsection{Proof of Algebraic Continuity Lemma III}

As in the case of Lemmas~\ref{continuity} and \ref{continuity_2},
our proof of the algebraic continuity lemma
for the linearisation
is split into two main parts.

\medskip

In a first part, we will prove the following interpolation estimate which is the analogue of Lemmas~\ref{l:intI} and \ref{l:intII} for linearised modelled distributions. 
In view of the already established a priori estimate (Theorem \ref{a priori})
we may assume strong a priori control of $f$, cf.~(\ref{wr05}).
\begin{lemma}[Interpolation III]\label{l:intIII} 
Suppose that $f$ is a modelled distribution in the sense of Definition~\ref{rf01} with respect to a model $(\Pi, \Pi^-, \Gamma^*)$, and $\dot f$ a linearised modelled distribution as in Definition \ref{rf01l}.
We assume that
\begin{align}\label{wr05}
|\lambda| + \| \Gamma^* \|_{\kappa} + \sum_{| \mathbf{n} | < \kappa} \|f.\mathsf{z}_{\bf n}\|\le M
\end{align}
for some $M < \infty$.
Then we have
\begin{align}
\|\dot f.\mathsf{z}_{\bf n}\|&\lesssim
\big([\dot f]_{\kappa,{\rm pol}}
+|\lambda|\|\dot f.\mathsf{z}_{\bf 0}\|\big)^\frac{|{\bf n}|}{\kappa}
\|\dot f.\mathsf{z}_{\bf 0}\|^{1-\frac{|{\bf n}|}{\kappa}}\quad
\mbox{for}\;|{\bf n}|<\kappa,\label{wr03}\\
[\dot f]_{\tilde{\kappa},{\rm pol}}&\lesssim
\big([\dot f]_{\kappa,{\rm pol}}
+|\lambda|\|\dot f.\mathsf{z}_{\bf 0}\|\big)^\frac{\tilde{\kappa}}{\kappa}
\|\dot f.\mathsf{z}_{\bf 0}\|^{1-\frac{\tilde{\kappa}}{\kappa}}\quad
\mbox{for}\;\tilde{\kappa}\le\kappa,\label{wr04}
\end{align}
where the implicit constant depends on $M, \alpha, d$ and $\kappa$.
\end{lemma}

In a second part, we 
establish the following preliminary algebraic estimate, which is the analogue of Lemmas~\ref{l:pcI} and \ref{l:pcII} for linearised modelled distributions.
\begin{lemma}\label{l:pcIII}
Suppose that $f$ is a modelled distribution in the sense of Definition~\ref{rf01} with respect to a model $(\Pi, \Pi^-, \Gamma^*)$, and $\dot f$ a linearised modelled distribution in the sense of Definition \ref{rf01l}.
We assume that
\begin{align}\label{wr55bis}
	\| \Gamma^* \|_{\kappa} + \|f\|_{\kappa} \le M
\end{align}
for some $M < \infty$.
Then we have for all $\beta$ and $\eta$ with $|\beta|<\eta\le\kappa$
\begin{align}\label{wr57}
[\dot f]_{\eta,\beta}
&\lesssim|\lambda|^{\beta ( \mathfrak{3} )}
\big(\max_{\tilde\eta\le\eta}[\dot f]_{\tilde\eta,{\rm pol}}
+\max_{|{\bf n}|<\eta}\|\dot f.\mathsf{z}_{\bf n}\|\big),
\end{align}
where the (first) max runs over all $\tilde{\eta} > 0$ of the form $\tilde{\eta} = \eta - ( | \gamma | - \alpha )$
and where
the implicit constant depends on $M, \alpha, d$ and $\kappa$.
\end{lemma}

Again, with these ingredients in hand, it is now straightforward to establish Lemma~\ref{continuity_3}.
\begin{proof}[Proof of Lemma~\ref{continuity_3}]
The desired result follows directly by inserting \eqref{wr03} and \eqref{wr04} into the right-hand side of \eqref{wr57} (where we take $\eta = \kappa$).
\end{proof}

We turn to the proof of Lemma~\ref{l:intIII}.
\begin{proof}[Proof of Lemma~\ref{l:intIII}]
As in the proof of Lemma \ref{l:intI}
we have 
\begin{align}\label{wr14}
\max_{|{\bf n}|<\kappa}\|\dot f.\mathsf{z}_{\bf n}\|r^{|{\bf n}|}
\lesssim[\dot f]_{\kappa,{\rm pol}}r^\kappa
+\|\dot f.\mathsf{z}_{\bf 0}\|+\max_{\stackrel{0\le|\beta|<\kappa}{\beta ( \mathfrak{3} )\not=0}}
\|\dot f.\mathsf{z}^\beta\|r^{|\beta|}.
\end{align}
We note that the assumption \eqref{wr05} and the inequality \eqref{wr08} together imply that $\|f.\mathsf{z}^\beta\| \lesssim 1$. Therefore, according to \eqref{t49} we have that
\begin{align}\label{wr17}
\|\dot f.\mathsf{z}^\beta\|\lesssim|\lambda|^{\beta ( \mathfrak{3} )}
\sum_{\bf n}\beta({\bf n})\|\dot f.\mathsf{z}_{\bf n}\|.
\end{align}
By the same argument as in Lemma \ref{l:intI} we obtain
\begin{align}\label{wr15}
\max_{\stackrel{0\le|\beta|<\kappa}{\beta ( \mathfrak{3} )\not=0}}\|\dot f.\mathsf{z}^\beta\|r^{|\beta|}
\lesssim|\lambda|
\big(\max_{|{\bf n}|<\kappa}\|\dot f.\mathsf{z}_{\bf n}\|r^{|{\bf n}|}\big)
\max\{r^{\kappa},r^{\delta}\}.
\end{align}
The combination of \eqref{wr14} and \eqref{wr15} yields
\begin{align*}
\max_{|{\bf n}|<\kappa}\|\dot f.\mathsf{z}_{\bf n}\|r^{|{\bf n}|}
\lesssim[\dot f]_{\kappa,{\rm pol}}r^\kappa+\|\dot f.\mathsf{z}_{\bf 0}\|
\quad\mbox{if}\quad
|\lambda|\max\{r^{\kappa},r^{\delta}\}\ll 1.
\end{align*}
Since as in the proof of Lemma \ref{l:intI}, the proviso is equivalent
to $|\lambda|r^\kappa\ll 1$, choosing $r$ to be a small multiple
of 
	\begin{align*}
		\Big(\min\Big\{\frac{\|\dot f.\mathsf{z}_{\bf 0}\|}{[\dot f]_{\kappa,{\rm pol}}},\frac{1}{|\lambda|}\Big\}\Big)^\frac{1}{\kappa}
	\end{align*}
yields (\ref{wr03})
in the equivalent form of
$\|\dot f.\mathsf{z}_{\bf n}\|$
$\lesssim[\dot f]_{\kappa,{\rm pol}}^{|{\bf n}|/\kappa}$
$\|\dot f.\mathsf{z}_{\bf 0}\|^{1-|{\bf n}|/\kappa}$
$+|\lambda|^{|{\bf n}|/\kappa}$ $\|\dot f.\mathsf{z}_{\bf 0}\|$.

\medskip

We now turn to (\ref{wr04}). As for (\ref{wr30}) we have for arbitrary $r$
\begin{align}\label{wr33}
[\dot f]_{\tilde{\kappa},{\rm pol}}
&\lesssim r^{\kappa-\tilde{\kappa}}[\dot f]_{\kappa,{\rm pol}}
+\max_{\beta\;\mbox{\tiny populated}:0\le|\beta|<\kappa}
r^{|\beta|-\tilde{\kappa}}\|\dot f.\mathsf{z}^\beta\|.
\end{align}
Inserting (\ref{wr03}) into (\ref{wr17}) we obtain for the last term
\begin{align}\label{wr31}
\|\dot f.\mathsf{z}^\beta\|\lesssim |\lambda|^{\beta ( \mathfrak{3} )}\max_{{\bf n}:\beta({\bf n})\not=0}
([\dot f]_{\kappa,{\rm pol}}+|\lambda|\|\dot f.\mathsf{z}_{\bf 0}\|)^\frac{|{\bf n}|}{\kappa}
\|\dot f.\mathsf{z}_{\bf 0}\|^{1-\frac{|{\bf n}|}{\kappa}}.
\end{align}
Provided $\beta$ is populated with $0\le|\beta|<\kappa$, we claim that this implies the more compact
\begin{align}\label{wr32}
\|\dot f.\mathsf{z}^\beta\|&\lesssim
([\dot f]_{\kappa,{\rm pol}}+|\lambda|\|\dot f.\mathsf{z}_{\bf 0}\|)^\frac{|\beta|}{\kappa}
\|\dot f.\mathsf{z}_{\bf 0}\|^{1-\frac{|\beta|}{\kappa}}\nonumber\\
&\sim [\dot f]_{\kappa,{\rm pol}}^\frac{|\beta|}{\kappa}
\|\dot f.\mathsf{z}_{\bf 0}\|^{1-\frac{|\beta|}{\kappa}}
+|\lambda|^\frac{|\beta|}{\kappa}\|\dot f.\mathsf{z}_{\bf 0}\|.
\end{align}
Indeed, in case of $\beta ( \mathfrak{3} )=0$, and for ${\bf n}$ with $\beta({\bf n})\not=0$ we
must have $\beta=\delta_{\bf n}$ since $\beta$ is populated, so that (\ref{wr31})
is identical to (\ref{wr32}).
In case of $\beta ( \mathfrak{3} )\not=0$ we have $\beta ( \mathfrak{3} )\ge|\beta|/\kappa$ so that
by assumption (\ref{wr05}) $|\lambda|^{\beta ( \mathfrak{3} )}$ $\le|\lambda|^{|\beta|/\kappa}$
$\le 1$. Since for ${\bf n}$ with $\beta({\bf n})\not=0$ we have $|{\bf n}|$ $\le|\beta|$,
we may pass from (\ref{wr31}) to (\ref{wr32}) by Young's inequality.
Inserting (\ref{wr32}) into (\ref{wr33}) we obtain
\begin{align*}
[\dot f]_{\tilde{\kappa},{\rm pol}}
&\lesssim r^{\kappa-\tilde{\kappa}}([\dot f]_{\kappa,{\rm pol}}
+|\lambda|\|\dot f.\mathsf{z}_{\bf 0}\|)\nonumber\\
&\quad+\max_{|\beta|<\kappa}r^{|\beta|-\tilde{\kappa}}
([\dot f]_{\kappa,{\rm pol}}+|\lambda|\|\dot f.\mathsf{z}_{\bf 0}\|)^\frac{|\beta|}{\kappa}
\|\dot f.\mathsf{z}_{\bf 0}\|^{1-\frac{|\beta|}{\kappa}}.
\end{align*}
The choice of $r=(\frac{\|\dot f.\mathsf{z}_{\bf 0}\|}
{[\dot f]_{\kappa,{\rm pol}}+|\lambda|\|\dot f.\mathsf{z}_{\bf 0}\|})^\frac{1}{\kappa}$ 
yields (\ref{wr04}).
\end{proof}

We turn to
the proof of 
Lemma~\ref{l:pcIII}.

\begin{proof}[Proof of Lemma~\ref{l:pcIII}]
We start by establishing the infinitesimal analogue of the formulas (\ref{wr41}) and (\ref{wr42}).
Namely, by replacing the use of the multiplicativity of $f$ by the Leibniz rule on $\dot{f}$ one readily obtains
\begin{align}
\lefteqn{(\dot f_{x+h}.-\dot f_x.Q_\eta\Gamma_{x\,x+h}^*)
\mathsf{z}^{\beta+\delta_{\mathfrak{3}}}}\nonumber\\
&= \lambda \, (\dot f_{x+h}.-\dot f_x.Q_{\eta-(|\delta_{\mathfrak{3}}|-\alpha)}\Gamma_{x\,x+h}^*)\mathsf{z}^{\beta} \label{wr48}
\end{align}
and
\begin{align}\label{wr49}
\lefteqn{(\dot f_{x+h}.-\dot f_x.Q_\eta\Gamma_{x\,x+h}^*)
\mathsf{z}^{\beta+\delta_{\bf n}}}\nonumber\\
&=(\dot f_{x+h}.\mathsf{z}^{\beta})
(f_{x+h}.-f_x.Q_{\eta-(|\beta|-\alpha)}\Gamma_{x\,x+h}^*)\mathsf{z}_{\bf n}\nonumber\\
&\quad+(f_{x+h}.\mathsf{z}^{\beta})
(\dot f_{x+h}.-\dot f_x.Q_{\eta-(|\beta|-\alpha)}\Gamma_{x\,x+h}^*)\mathsf{z}_{\bf n}\nonumber\\
&\quad+\sum_{|\gamma|<\eta-(|\beta|-\alpha)} 
(\dot f_{x+h}.-\dot f_x.Q_{\eta-(|\gamma|-\alpha)}\Gamma_{x\,x+h}^*)\mathsf{z}^\beta \, (\Gamma_{x\,x+h}^*)_\gamma^{\delta_{\bf n}}f_x.\mathsf{z}^\gamma\nonumber\\
&\quad+\sum_{|\gamma|<\eta-(|\beta|-\alpha)}
(f_{x+h}.-f_x.Q_{\eta-(|\gamma|-\alpha)}\Gamma_{x\,x+h}^*)\mathsf{z}^\beta \, (\Gamma_{x\,x+h}^*)_\gamma^{\delta_{\bf n}}\dot f_x.\mathsf{z}^\gamma .
\end{align}

As in the proof of Lemma \ref{l:pcI}, the identities (\ref{wr48})
and \eqref{wr49}
yield the inequalities
\begin{align}\label{wr53}
[\dot f]_{\eta,\beta+\delta_{\mathfrak{3}}}\le|\lambda|[\dot f]_{\eta-(|\delta_{\mathfrak{3}}|-\alpha),\beta}
\le|\lambda|\max_{\tilde\eta\le\eta}[\dot f]_{\tilde\eta,\beta}
\end{align}
and
\begin{align*}
\lefteqn{[\dot f]_{\eta,\beta+\delta_{\bf n}}
\le\|\dot f.\mathsf{z}^{\beta}\|[f]_{\eta-(|\beta|-\alpha),\delta_{\bf n}}
+\|f.\mathsf{z}^{\beta}\|[\dot f]_{\eta-(|\beta|-\alpha),\delta_{\bf n}}}\nonumber\\
&\quad+\|\Gamma^*\|_\eta \sum_{|\gamma|<\eta-(|\beta|-\alpha)}
\big(
[\dot f]_{\eta-(|\gamma|-\alpha),\beta}\|f.\mathsf{z}^\gamma\|
+[f]_{\eta-(|\gamma|-\alpha),\beta}\|\dot f.\mathsf{z}^\gamma\|\big).
\end{align*}
By inequality (\ref{wr08}) and our assumption (\ref{wr55bis})
the last inequality yields the estimate
\begin{align*}
[\dot f]_{\eta,\beta+\delta_{\bf n}}
& \lesssim\|\dot f.\mathsf{z}^{\beta}\|
+|\lambda|^{\beta ( \mathfrak{3} )}[\dot f]_{\eta-(|\beta|-\alpha),\delta_{\bf n}}\nonumber\\
&\quad+\max_{|\gamma|<\eta-(|\beta|-\alpha)}
\big([\dot f]_{\eta-(|\gamma|-\alpha),\beta}
+|\lambda|^{\beta(\mathfrak{3})}\|\dot f.\mathsf{z}^\gamma\|\big),
\end{align*}
where we
also 
used 
Lemma~\ref{l:pcI}
and the fact
that 
$\sup_{\tilde{\kappa} \leq \kappa} \| f \|_{\tilde{\kappa}} \lesssim \|f\|_{\kappa} \lesssim 1$
as readily follows from the interpolation inequalities of Lemma~\ref{l:intI}.
Appealing to inequality (\ref{wr17})  
this estimate reduces to
\begin{align}\label{wr54}
[\dot f]_{\eta,\beta+\delta_{\bf n}}
&\lesssim|\lambda|^{\beta ( \mathfrak{3} )}
\big(\max_{\tilde\eta\le\eta}[\dot f]_{\tilde\eta,{\rm pol}}+\max_{|{\bf m}|<\eta}\|\dot f.\mathsf{z}_{\bf m}\|\big) +\max_{\tilde\eta\le\eta}
[\dot f]_{\tilde\eta,\beta} .
\end{align}
As in the proof of Lemma \ref{l:pcI}, 
(\ref{wr57}) follows from (\ref{wr53}) and (\ref{wr54}) by induction
in the plain length of $\beta$. 
\end{proof}

\subsection{Proof of A Priori Estimates for the Linearisation}

\begin{proof}[Proof of Theorem~\ref{a priori lin}]
We follow again the same strategy of proof as in Theorems~\ref{a priori} and \ref{uniqueness}, 
remaining brief when the steps are exact adaptations of the corresponding steps therein.
We assume to be given $f$, $\dot{f}$, $\zeta$, $\kappa$
as in the statement of the theorem along with $|\lambda| \leq \lambda_0$, where $\lambda_0$ is free to be tuned inside the proof.
Throughout this proof, we assume for convenience and without loss of generality that the torus is of unit size i.e.\ that $\ell=1$.

\medskip
\PfStart{ap3} 
\PfStep{ap3}{ap3:1} 
\textsc{Algebraic Continuity.}
We recall the output of Lemma~\ref{continuity_3}: for all populated and non-pp $\beta$ with $|\beta|< \kappa$,
	\begin{align*}
		[\dot{f}]_{\kappa, \beta}
		& \lesssim |\lambda| \, ( [\dot{f}]_{\kappa, \mathrm{pol}} + \| \dot{f} . \mathsf{z}_{\mathbf{0}} \| )
		\lesssim |\lambda | \, \| \dot{f} \|_{\kappa} .
	\end{align*}

\PfStep{ap3}{ap3:2} 
\textsc{Reconstruction.}
We then note that, 
as for \eqref{ap01}, 
one has
	\begin{align*}
		\dot{R}_{x+h}^- - \dot{R}_x^-
		& = \sum_{0 < |\beta | < \kappa} ( \dot{f}_x . - \dot{f}_{x+ h} . Q_{\kappa} \Gamma_{x+h \, x} ) \mathsf{z}^{\beta} \, \Pi_{x \beta}^- ,
	\end{align*}
the sum being over non-pp $\beta$ by the population conditions on $\Pi^-$.
In turn, we learn 
that uniformly over 
$\mu \leq \ell = 1$
and $x \in \mathbb{R}^{1 + d}$,
	\begin{align*}
		| (\dot{R}_{x+h}^- - \dot{R}_x^-)_{\mu} ( x ) |
		& \lesssim |\lambda | \, \| \dot{f} \|_{\kappa} \, \mu^{\kappa^- - 2} ( \mu + | h | )^{\kappa - \kappa^-} ,
	\end{align*}
where we used the definition of $\kappa^-$.
Since $\dot{R}_{x \mu} ( x ) \to 0$ as $\mu \downarrow 0$  and $\kappa > 2$, the Reconstruction Theorem of Lemma~\ref{Recon} yields that
	\begin{align} \label{apl01}
		| \dot{R}_{x \mu}^- ( x ) |
		& \lesssim |\lambda | \, \| \dot{f} \|_{\kappa} \, \mu^{\kappa - 2} 
	\end{align}
	uniformly over $x \in \mathbb{R}^{1+d}$, $\mu \le 1$ and $\|\psi\|_p \le 1$.

\PfStep{ap3}{ap3:3} 
\textsc{Integration.}
Here, we would like to apply the Schauder estimate of Lemma~\ref{Schauder}, where we recall that $\dot{R}$ and $\dot{R}^-$ are related by the PDE
	\begin{align*}
		L \dot{R}_x 
		& = \dot{R}_x^- + \zeta + P_x ,
	\end{align*}
where $P_x$ is some polynomial of degree $< \kappa$.
In fact, recalling the notation
$T_x^{\kappa-2} \zeta ( y ) \coloneqq \sum_{| \mathbf{n} |< \kappa-2} \frac{1}{\mathbf{n}!} \partial^{\mathbf{n}} \zeta ( x ) y^{\mathbf{n}}$
for the (parabolic) Taylor germ of $\zeta$, 
it will be convenient 
to rename $P$ and
rewrite the above PDE as
	\begin{align}\label{apl02}
		L \dot{R}_x 
		& = \dot{R}_x^- + ( \zeta - T_x^{\kappa - 2} \zeta ) + P_x ,
	\end{align}
where $P_x$ is still some polynomial of degree $< \kappa$.
One advantage of defining $P$ by \eqref{apl02} is that arguing as in \eqref{po1} and the discussion thereafter reveals that it is again of the form $P_x = \sum_{\kappa - 2 < |\mathbf{n} | < \kappa} a_{x\mathbf{n}} ( \cdot - x )^{\mathbf{n}}$.
Furthermore, 
arguing as for \eqref{apn10} 
reveals that
	\begin{align}\label{apn10l}
		\sup_{x \in \mathbb{R}^{1 + d}} | a_{x\mathbf{n}}|
		& \lesssim |\lambda | \ \| \dot{f} \|_{\kappa} .
	\end{align}
We now claim that
with $\delta>0$ as defined in \eqref{is03} we have
	\begin{align} \label{apn11l}
		|(L \dot{R}_x)_\mu (x) | 
		& \lesssim \Big( 
			\big( \rho^{- \delta} + \rho^{p} | \lambda | \big) \, \| \dot{f} \|_{\kappa} 
			+ [\zeta]_{\kappa - 2}
			\Big) \, \mu^{\kappa - 2} ,
	\end{align}
where the implicit constant is uniform over $x \in \mathbb{R}^{1 + d}$, $\mu < \infty$, $\| \psi \|_{p} \leq 1$ and $\rho \geq 1$.
Again, we distinguish the regimes $\mu \lessgtr \rho$.
When $\mu \geq \rho$ we use the expression 
	\begin{align*}
		\dot{R}_{x \mu} ( x )
		& = [( \dot{f}_{\cdot} - \dot{f}_x ).\mathsf{z}_\mathbf{0}] * \psi_\mu ( x ) - \sum_{\kappa^- \leq | \beta | < \kappa} \dot{f}_x . \mathsf{z}^{\beta} \, \Pi_{x \beta \mu} ( x ).
 \end{align*}
which
can be proven in the same way as \eqref{apn12}, recall Remark~\ref{r:gDP},
and
implies
	\begin{align*}
		|(L \dot{R}_x)_\mu (x) | 
		& \lesssim \|\dot{f}\|_{\kappa} \, \rho^{- \delta} \, \mu^{\kappa - 2} ,
	\end{align*}
as desired.
On the other hand, when $\mu \leq \rho$ we write 
	\begin{align*}
		(L \dot{R}_x)_\mu (x)
		& = \dot{R}_x^{-} * ( \psi_{\rho} )_{\frac{\mu}{\rho}} ( x ) 
			+ (\zeta - T_x^{\kappa - 2} \zeta)_{\mu} ( x )
			+ P_{x \mu} ( x ) .
	\end{align*}
Observe that \eqref{apn10l} implies in this regime
	\begin{align*}
		| P_{x \mu} ( x ) | 
		& \lesssim \rho^p |\lambda | ( \| f; \tilde{f} \|_{\kappa} + \| \Pi - \tilde{\Pi} \|_{\kappa} ) \, \mu^{\kappa - 2} .
	\end{align*}
Furthermore, by the definition of the $(\kappa-2)$-H\"older seminorms
	\begin{align*}
		| (\zeta - T_x^{\kappa - 2} \zeta)_{\mu} ( x ) |
		& \lesssim [ \zeta ]_{\kappa - 2} \, \mu^{\kappa - 2} ,
	\end{align*}
where we implicitly used the fact that $p$ can be chosen so large that $p > \kappa - 2 + D$.
Finally, since \eqref{apl01} implies
	\begin{align*}
		| \dot{R}_x^{-} * ( \psi_{\rho} )_{\frac{\mu}{\rho}} |
		& \lesssim \rho^p | \lambda |
			\| \dot{f} \|_{\kappa} \, \mu^{\kappa- 2} ,
	\end{align*}
the desired 
\eqref{apn11d} follows.

\medskip
Applying our Schauder estimate of Lemma~\ref{Schauder} 
with \eqref{apn11d} as input,
we conclude that there exists a $p^\prime$ such that for any $\rho \ge 1$,
 \begin{align*} 
 	| \dot{R}_{x \mu} ( x ) | 
		& \lesssim
		\Big( 
			\big( \rho^{- \delta} + \rho^{p} | \lambda | \big) \, \| \dot{f} \|_{\kappa} 
			+ [\zeta]_{\kappa - 2}
			\Big) \, \mu^{\kappa} 
    \end{align*}
uniformly over $x \in \mathbb{R}^{1 + d}$, $\mu < \infty$, $\| \psi \|_{p^{\prime}} \leq 1$ and $\rho \geq 1$.

\PfStep{ap3}{ap3:4} 
\textsc{Three-Point Argument.}
This step also proceeds exactly as the corresponding step in the proof of Theorem~\ref{a priori}.
Namely, the output of the previous step, Lemma~\ref{3_point} and Lemma~\ref{continuity_3} together imply that
	\begin{align} \label{apl03}
		[\dot{f}]_{\kappa, \mathrm{pol}}
		& \lesssim ( \rho^{-\delta} + \rho^{p} \, | \lambda | ) \| \dot{f} \|_{\kappa} + [ \zeta ]_{\kappa-2} .
	\end{align}

\PfStep{ap3}{ap3:5} 
\textsc{Incorporating Boundary Data.}
We claim that, 
in part due to the periodicity of $\dot{f}$, 
one has
	\begin{align} \label{apl04}
		\| \dot{f} \|_{\kappa}
		& \lesssim [\dot{f}]_{\kappa, \mathrm{pol}} .
	\end{align}
We note that \eqref{apl03} and \eqref{apl04} are enough to conclude the announced  $\| \dot{f} \|_{\kappa} \lesssim [\zeta]_{\kappa - 2}$, 
by absorption after choosing $\rho$ large enough and $|\lambda|$ small enough in that order.
Thus, it only remains to prove \eqref{apl04}.
For that purpose, recall that by definition of $\| \dot{f} \|_{\kappa}$
	\begin{align*}
		\| \dot{f} \|_{\kappa}
		& \lesssim \sum_{|\beta|<\kappa, \beta\neq 0} ( \| \dot{f} . \mathsf{z}^{\beta} \| + [\dot{f}]_{\kappa, \beta} ).
	\end{align*}
In that sum, by the Algebraic Continuity Lemma~\ref{continuity_3} we may bound $[\dot{f}]_{\kappa, \beta} \lesssim [\dot{f}]_{\kappa, \mathrm{pol}} + |\lambda| \| \dot{f} . \mathsf{z}_{\mathbf{0}} \|$.
Furthermore, by the a priori estimates for $f$, Theorem~\ref{a priori}, 
in combination with the algebraic form \eqref{t49} of $\dot{f}$, 
we may also bound $\| \dot{f} . \mathsf{z}^{\beta} \| \lesssim \sum_{| \mathbf{n} | < \kappa} \| \dot{f} . \mathsf{z}_{\mathbf{n}} \|$, 
where the condition $| \mathbf{n} | < \kappa$ is due to
the fact that $\beta ( \mathbf{n} ) = 0$ for $|\mathbf{n}|\geq |\beta|$.
Thus, by the interpolation estimate 
of Lemma~\ref{l:intIII}, 
we deduce
	\begin{align*}
		\| \dot{f} \|_{\kappa}
		& \lesssim [ \dot{f} ]_{\kappa; \mathrm{pol}} + \| \dot{f} . \mathsf{z}_{\mathbf{0}} \| .
	\end{align*}
We finally incorporate the periodicity of $\dot{f}$.
By \eqref{i2}, the map $x \mapsto \dot{f}_x.\mathsf{z}_{\mathbf{0}}$ is periodic of vanishing space-time average.
Hence
	\begin{align*}
		\| \dot{f} . \mathsf{z}_{\mathbf{0}} \|
		& \lesssim [ \dot{f}. \mathsf{z}_{\mathbf{0}}]_{\kappa^-} 
		= [ \dot{f} ]_{\kappa^-, \mathrm{pol}}.
	\end{align*}
The desired result then follows by applying the interpolation inequality \eqref{wr04}, Young's inequality and then buckling in the same way as in the proofs of Theorem~\ref{a priori} and Theorem~\ref{uniqueness}.
\end{proof}

\subsection{Proof of Existence by Continuity Method}

This subsection is dedicated to the proof of the existence Theorem~\ref{existence}.
\begin{proof}[Proof of Theorem~\ref{existence} (Existence)]
\PfStart{c2} 
\PfStep{c2}{c2:1} 
Setting.
We follow the 
strategy outlined in Subsection~\ref{ss:exist_strat}.
Namely, in view of the Consistency Lemma~\ref{l:c}, 
it suffices to prove
that for a $\lambda_0$ (to be adjusted) depending only on $M, \alpha, d, \kappa$ and $\ell$, 
the set $\Lambda$ is non-empty, open, and closed, where we recall here the relevant notations for the convenience of the reader:
	\begin{align}
		\Lambda 
		\coloneqq \lbrace \lambda \in [- \lambda_0, \lambda_0] , \text{there is $u \in C_{\mathrm{per},0}^{\kappa}$ satisfying } \eqref{eq:P1} \text{ and } \eqref{eq:P2} \rbrace , \nonumber
	\end{align}
where the properties $P_1$ and $P_2$ are defined by
	\begin{align}
		\tag{$P_1 ( \lambda , u)$}
		& L u = P ( \lambda u^3 + h_{\lambda} u + \xi ) , \text{ and $\| f . \mathsf{z}_{\mathbf{0}} \| \leq 1$ (recall \eqref{t48})} . \\
		\tag{$P_2 ( \lambda, u )$}
		& \text{The linearised operator} \\
		& \quad 
		\begin{array}[t]{lrcl}
		\dot{\Phi} ( \lambda, u ) : & C_{\mathrm{per}, 0}^{\kappa} & \longrightarrow & C_{\mathrm{per}, 0}^{\kappa - 2} \\
    & \dot{u} & \longmapsto & L \dot{u} - P \big( ( 3 \lambda u^2 + h_{\lambda} ) \dot{u} \big) ,
  		\end{array} \nonumber \\
  		&\text{is continuously invertible.} \nonumber
	\end{align}
Here, up to choosing $\lambda_0$ small enough, by the a priori estimate \eqref{t58} of Theorem~\ref{a priori}, we may (and will throughout this proof) assume that $M \geq 2$.
The remainder of this proof is organised as follows:
in Step~\ref{c2:2}, we make precise our choice of $\lambda_0$; 
in Step~\ref{c2:3}, we prove that $\Lambda$ is not empty;
in Step~\ref{c2:4}, we prove that $\Lambda$ is open;
Step~\ref{c2:5a} is a preparatory step for Step~\ref{c2:5}, where we prove that $\Lambda$ is closed.

\PfStep{c2}{c2:2}
We now make precise our choice of $\lambda_0 > 0$.
We choose $\lambda_0$ small enough in function of $M, \alpha, d, \kappa$ and $\ell$ in such a way that:
	\begin{enumerate}
		\item The a priori estimates Theorems~\ref{a priori} and \ref{a priori lin} hold w.~r.~t.~ 
		$\lambda_0$.
		
		\item The implicit multiplicative constant
		in the output \eqref{t58} of Theorem~\ref{a priori} is at most $\lambda_0^{-1}$.
	\end{enumerate}
By these theorems, such a $\lambda_0$ exists.
The reason behind this choice is that it ensures that if $f$ is a solution of the robust formulation from
Definition~\ref{rf03} then
	\begin{align}\label{t59}
		\text{$| \lambda | \leq \lambda_0$ and $\| f . \mathsf{z}_{\mathbf{0}} \| \leq 2$, imply that $\| f . \mathsf{z}_{\mathbf{0}} \| \leq 1$.}
	\end{align}
Indeed, since $\| f . \mathsf{z}_{\mathbf{0}} \| \leq 2 \leq M$,
Theorem~\ref{a priori} applies 
and yields
$\| f . \mathsf{z}_{\mathbf{0}} \| \leq \| f \|_{\kappa} \leq \lambda_0^{-1} | \lambda | \leq 1$.
The property \eqref{t59} will ensure that the second condition in \eqref{eq:P1} is not problematic in the `openness’ argument below.

\PfStep{c2}{c2:3}
In this step, we prove that $0 \in \Lambda$.
When $\lambda = 0$, we may take $u = \Pi_{\beta = 0}$, i.e.\ $f \equiv 0$, which indeed satisfies \eqref{eq:P1}.
Furthermore, the linearised operator is simply the heat operator
$\dot{\Phi} ( \lambda = 0, u ) = L \colon C_{\mathrm{per}, 0}^{\kappa} \to  C_{\mathrm{per}, 0}^{\kappa-2}$.
It is continuously invertible by classical Schauder theory,
which we will further use throughout this proof in the form $[u]_{\eta} \lesssim_{\eta, d} [L u]_{\eta - 2}$ for $u \in C_{\mathrm{per}, 0}^{\eta}$ and $\eta > 2$, $\eta \notin \mathbb{N}$, recall e.g.\ \cite[Theorem~1]{Sim97}.
In turn, \eqref{eq:P2} is also satisfied, 
as desired.

\PfStep{c2}{c2:4} \textsc{Openness.}
In this step, we prove that $\Lambda$ is open.
To that effect, we fix throughout this step some $\lambda_* \in \Lambda$,
and we shall prove that some neighbourhood of $\lambda_*$ is still contained into $\Lambda$.
More precisely, 
we let $u_* \in C_{\mathrm{per}, 0}^{\kappa}$ 
satisfy the properties $(P_1 (\lambda_*, u_*))$ and $(P_2 (\lambda_*, u_*))$, 
and
we shall prove that in a neighbourhood of $\lambda_*$ the properties \ref{eq:P1} and \ref{eq:P2} are still satisfied, in that order.
For the former property, we argue by the Implicit Function Theorem.
For the latter property, we appeal to a (standard) perturbation argument.
Both rely on `classical’ arguments in the well-known H\"older topologies.

\medskip
\textit{Openness of \eqref{eq:P1}.}
In this paragraph we prove that \eqref{eq:P1} remains satisfied on a neighbourhood of $\lambda_*$.
We consider the map 
	\begin{align*}
	\begin{array}[t]{lrcl}
		\Phi : & \mathbb{R} \times C_{\mathrm{per}, 0}^{\kappa} & \longrightarrow & C_{\mathrm{per}, 0}^{\kappa - 2} \\
    & (\lambda, u ) & \longmapsto & L u - P ( \lambda u^3 + h_{\lambda} u + \xi ) .
  		\end{array}
	\end{align*}
By assumption, $\Phi ( \lambda_*, u_*) = 0$.
Let us assume for the moment the following properties, which we prove below.
	\begin{align}
		& \text{$\Phi$ is continuously Fr\'echet differentiable,} \label{t60} \\
		& \text{$\partial_2 \Phi (\lambda_*, u_*) \colon C_{\mathrm{per}, 0}^{\kappa} \to C_{\mathrm{per}, 0}^{\kappa - 2}$ is continuously invertible.} \label{t61}
	\end{align}
Then, we may apply the Implicit Function Theorem
(see e.g.\ \cite[Theorem~5.9]{Lan99}):
there exists a neighbourhood $U$ of $\lambda_*$ along with a Fr\'echet differentiable function $\Psi \colon U \to C_{\mathrm{per}, 0}^{\kappa}$ such that $\Psi ( \lambda_* ) = u_*$ and $\Phi ( \lambda, \Psi ( \lambda ) ) = 0$ for all $\lambda \in U$.
For those $\lambda \in U$ we will denote $u \coloneqq u_\lambda \coloneqq \Psi ( \lambda )$, so that the first item in \eqref{eq:P1} is satisfied for all $\lambda \in U$.
We now argue that also the second item in \eqref{eq:P1} holds (up to possibly further restricting the neighbourhood $U$).
Indeed, denote by $f_{\lambda}$ the modelled distribution provided by the consistency Lemma~\ref{l:c} for $u_{\lambda}$.
We know
from the output of the Implicit Function Theorem
that $U \ni \lambda \mapsto u_{\lambda} = \Psi ( \lambda ) \in C_{\mathrm{per}, 0}^{\kappa}$ is continuous.
Appealing to the explicit expression \eqref{t10} of $(f_{\lambda})_x.\mathsf{z}_{\mathbf{0}}$ and the fact that each $\Pi_\beta$ appearing therein is periodic and $\kappa$-H\"older by Definition~\ref{d:smooth},
one deduces that
also the map 
$\lambda \mapsto \| f_\lambda . \mathsf{z}_{\mathbf{0}} \|$ 
is continuous.
Since for $\lambda = \lambda_*$,
one has
$\| f_{\lambda_*} . \mathsf{z}_{\mathbf{0}} \| \leq 1$ by assumption,  
it follows that 
on a neighbourhood of $\lambda_*$ one has 
$\| f_\lambda . \mathsf{z}_{\mathbf{0}} \| \leq 2$ and in fact $\| f_\lambda . \mathsf{z}_{\mathbf{0}} \| \leq 1$ by \eqref{t59}.

\medskip
\textit{Proof of \eqref{t60} and \eqref{t61}.}
We now establish the properties \eqref{t60} and \eqref{t61} which we used above.
We start with the former:
by composition, 
and since $\lambda \mapsto h_{\lambda}$ is smooth (it is polynomial by \eqref{t22}), 
\eqref{t60} follows from the continuous Fr\'echet differentiability of the maps $u \mapsto L u$ and $u \mapsto P ( u^3 )$
from $C_{\mathrm{per}, 0}^{\kappa}$ to $C_{\mathrm{per}, 0}^{\kappa-2}$. 
But both are elementary to check.
Turning to \eqref{t61}, we note that by direct computation,
$\partial_2 \Phi ( \lambda_*, u_* ) = \dot{\Phi} ( \lambda_*, u_* )$, 
so that \eqref{t61} is nothing else but a reformulation of our
assumption $P_2 (\lambda_*, u_* )$.

\medskip
\textit{Openness of \eqref{eq:P2}.}
In this paragraph we prove that \eqref{eq:P2} remains satisfied on a neighbourhood of $\lambda_*$, where $U \ni \lambda \mapsto u \in C_{\mathrm{per}, 0}^{\kappa}$ is the continuous map obtained above.
Let $\lambda \in U$
and let $\zeta \in C_{\mathrm{per}, 0}^{\kappa - 2}$.
The desired \eqref{eq:P2} will follow once we prove that the equation
	\begin{align}
		L \dot{u} - P \big( ( 3 \lambda u_{\lambda}^2 + h_{\lambda } ) \dot{u} \big) = \zeta , \nonumber
	\end{align}
has a unique solution $\dot{u} \in C_{\mathrm{per}, 0}^{\kappa}$; and that $[ \dot{u} ]_{\kappa}$ is bounded by a constant multiple of $[ \zeta ]_{\kappa - 2}$.
To that effect, observe that $\dot{u}$ solves this equation if and only if it is a fixed point of
	\begin{align*}
	\begin{array}[t]{lrcl}
		F : 
		& C_{\mathrm{per}, 0}^{\kappa} 
			& \longrightarrow 
			& C_{\mathrm{per}, 0}^{\kappa} \\
 		 & \dot{u}
   		 	& \longmapsto 
   		 	& \dot{\Phi} ( \lambda_*, u_* )^{-1} \big( P ( v_{\lambda} \dot{u} ) + \zeta \big) .
   	\end{array}
	\end{align*}
where we have denoted
	\begin{align}
		v_{\lambda} \coloneqq 3 ( \lambda u_{\lambda}^2 - \lambda_* u_*^2 ) + ( h_{\lambda} - h_{\lambda_*} ) . \nonumber
	\end{align}	 
The map $U \ni \lambda \mapsto v_{\lambda} \in C_{\mathrm{per}, 0}^{\kappa}$ is continuous and vanishes at $\lambda = \lambda_*$. 
Thus, the Lipschitz constant of $F$, which is proportional to $\| \dot{\Phi} ( \lambda_*, u_* )^{-1} \| \, [ v_{\lambda} ]_{\kappa - 2}$ (where we have denoted by $\| \cdot \|$ the operator norm from $C_{\mathrm{per}, 0}^{\kappa - 2}$ to $C_{\mathrm{per}, 0}^{\kappa}$),
vanishes as $\lambda \to 0$ by the embedding 
$[ \cdot ]_{\kappa - 2} \lesssim_{\kappa, \ell, d} [\cdot]_{\kappa}$.
Thus, $F$ is a contraction
for all $\lambda$ in a neighbourhood of $\lambda_*$.
Up to choosing that neigbourhood small enough, we may assume that the rate of contraction therein is $\leq 1/2$.
By Banach’s fixed-point theorem, 
$F$ then has a unique fixed-point $\dot{u} \eqqcolon \Phi ( \lambda, u_{\lambda} )^{-1} \zeta$, 
whence the desired invertibility, 
and furthermore
$[\dot{u}]_{\kappa} \leq 2 [ F ( 0 ) ]_{\kappa}
= 2 [ \dot{\Phi} ( \lambda_*, u_* )^{-1} \zeta ]_{\kappa} 
\leq 2 \| \dot{\Phi} ( \lambda_*, u_* )^{-1} \| \, [ \zeta ]_{\kappa - 2}$, 
whence the desired continuity.
This concludes Step~\ref{c2:4}, i.e.\ we have proven that $\Lambda$ is open.

\PfStep{c2}{c2:5a}
Before turning to the proof that $\Lambda$ is closed, we record two useful a priori estimates 
which follow from Theorems~\ref{a priori} and \ref{a priori lin}.
Namely, we claim that 
there is a constant $C<\infty$ depending possibly on $\Pi, \Pi^-, \Gamma, M, \alpha, d, \kappa$ and $\ell$, but not on $u$, $\dot{u}$, nor on $\zeta \in C_{\mathrm{per}, 0}^{\kappa - 2}$, such that for all 
$| \lambda | \leq \lambda_0$
and all $u \in C_{\mathrm{per}, 0}^{\kappa}$ 
satisfying \eqref{eq:P1},
then 
	\begin{align} 
		& [ u ]_{\kappa} \leq C ,\label{t62} \\
		& \text{If } \dot{u} \in C_{\mathrm{per}, 0}^{\kappa} \text{ and } L \dot{u} - P \big( ( 3 \lambda u^2 +h_{\lambda} ) \dot{u} ) = \zeta , 
		\text{ then } [ \dot{u} ]_{\kappa} \leq C [ \zeta ]_{\kappa - 2} . \label{t63}
	\end{align}

We note that \eqref{t62} amounts to an a priori estimate of \eqref{t30} in the $\kappa$-H\"older norm, while \eqref{t63} amounts to the a priori boundedness of $\dot{\Phi}^{-1} ( \lambda, u )$.
We also note that the smallest constant in \eqref{t62} and \eqref{t63} depends on a stronger topology on $\xi$ than is stable in the singular SPDE setting. The key point is that due to our use of the stronger a priori estimate of Theorem~\ref{a priori}, these bounds to do hold for $|\lambda| \le \lambda_0$ with $\lambda_0$ not dependent on the norm of the noise in a strong topology. 
In the course of this step, we use the letter $C$ to designate generic constants as above whose value might change from line to line.

\medskip

We turn to the proof of \eqref{t62}.
We first prove that the version of \eqref{t62} with $\kappa^-$ in place of $\kappa$ holds.
Let $f$ be the modelled distribution associated to $u$ by the consistency lemma, Lemma~\ref{l:c}.
We start by monitoring the $\kappa^-$-H\"older norm of
the function $x \mapsto f_x . \mathsf{z}_{\mathbf{0}}$.
Thanks to our choice of $\lambda_0$, $f$ satisfies the assumptions of Theorem~\ref{a priori}, yielding the a priori estimate $\| f \|_{\kappa} \leq 1$.
Using
\eqref{mb05} 
to express the increments of $f . \mathsf{z}_{\mathbf{0}}$ as
	\begin{align}
		f_y . \mathsf{z}_{\mathbf{0}} - f_x . \mathsf{z}_{\mathbf{0}} 
		 = \big( f_y . - f_x . Q_{\kappa} \Gamma_{x y}^* \big) \mathsf{z}_{\mathbf{0}} 
		 + f_x . Q_{(0, \kappa)} \Gamma_{x y}^* \mathsf{z}_{\mathbf{0}} , \label{t64}
	\end{align}
and appealing to 
the model bounds 
\eqref{mb03} 
on $\Gamma^*$
and the definition \eqref{i7} of $\kappa^-$, 
we deduce that $[f .\mathsf{z}_{\mathbf{0}}]_{\kappa^-} \leq C$.
Recall the explicit expression \eqref{t10} of $f.\mathsf{z}_{\mathbf{0}}$, and note that only $\beta < 0$ contribute therein,
so that the corresponding $\Pi_{\beta}$ are periodic and $\kappa$-H\"older
(thus also $\kappa^-$-H\"older) 
by Definition~\ref{d:smooth}.
We learn that in turn, $[u]_{\kappa^{-}} \leq [f .\mathsf{z}_{\mathbf{0}}]_{\kappa^-} + \sum_{| \beta | < 0} |\lambda|^{\beta ( \mathfrak{3} )} [ \Pi_{0 \beta} ]_{\kappa^-} \leq C$.
We conclude by a bootstrapping argument.
Recalling that $u$ satisfies the PDE 
$L u = P ( \lambda u^3 + h_{\lambda} u + \xi )$, we deduce by examining the r.h.s.\ that $[L u]_{\kappa^-} \leq C$. In turn, by classical Schauder theory we see that $[u]_{\kappa^- + 2} \leq C$.
The desired \eqref{t62} follows by iterating that procedure until one reaches the regularity exponent $\kappa$.

\medskip

We turn to the proof of \eqref{t63}, proceeding similarly to \eqref{t62}.
Let $\dot{f}$ be the modelled distribution 
corresponding to $\dot{u}$ provided by the consistency lemma for the linearisation, Lemma~\ref{l:c2}.
Observe that the assumptions of the a priori estimates Theorem~\ref{a priori lin} for the linearisation are satisfied, 
yielding $\| \dot{f} \|_{\kappa} \leq C [\zeta]_{\kappa - 2}$.
Recall the explicit expression \eqref{t51} of $\dot{f}.\mathsf{z}_{\mathbf{n}}$
and observe that for $\mathbf{n} = \mathbf{0}$
it reduces to
$\dot{u} = \dot{f}.\mathsf{z}_{\mathbf{0}}$.
In turn, 
rewriting the increments of $\dot{f}.\mathsf{z}_{\mathbf{0}}$ as in \eqref{t64}, 
appealing to the model bounds 
\eqref{mb03} 
on $\Gamma^*$
and the definition \eqref{i7} of $\kappa^-$,
we deduce
that 
$[ \dot{u} ]_{\kappa^-} \leq C [\zeta]_{\kappa - 2}$.
Finally, the desired \eqref{t63} is obtained by the same bootstrapping argument as described in the proof of \eqref{t62},
this time using the PDE 
$L \dot{u} = P \big( ( 3 \lambda u^2 +h_{\lambda} ) \dot{u} ) + \zeta$.

\PfStep{c2}{c2:5} \textsc{Closedness.}
In this step, we prove that $\Lambda$ is closed.
To that effect, let
$(\lambda_n)_{n \in \mathbb{N}} \subset \Lambda$ be a sequence of elements of $\Lambda$ such that $\lambda_n \to_{n \to \infty} \lambda \in \mathbb{R}$.
We want to prove that $\lambda \in \Lambda$.
We establish the two properties \eqref{eq:P1} and \eqref{eq:P2} separately, using as key ingredient the
qualitative a priori estimates
from Step~\ref{c2:5a} above.

\medskip
\textit{Closedness of \eqref{eq:P1}.}
Here we prove that \eqref{eq:P1} holds for some $u$.
By assumption that $\lambda_n \in \Lambda$, there is a sequence $(u_n)_{n} \subset C_{\mathrm{per}, 0}^{\kappa}$
of functions with 
$L u_n = P ( \lambda_n u_n^3 + h_{\lambda_n} u_n + \xi )$, 
along with coresponding modelled distributions $f_n$ given by the consistency Lemma~\ref{l:c},
such that $\| f_n . \mathsf{z}_{\mathbf{0}} \| \leq 1$.
We apply \eqref{t62}
to $u_n$, yielding $\sup_n [u_n]_{\kappa} < \infty$.
Then, the sequence $(u_n)_n$
of periodic functions with vanishing space-time average
is uniformly bounded and uniformly equicontinuous.
We apply the Arzel\`a--Ascoli theorem: up to a subsequence, $(u_n)_n$ converges
uniformly to some function $u\in C_{\mathrm{per}, 0}^{\kappa}$.
Since $\kappa > 2$, by interpolation the convergence also holds in $C^2$.
Letting $n \to \infty$ in the sequence of equations satisfied by $u_n$, and because $\lambda \mapsto h_{\lambda}$ is continuous 
(it is polynomial by \eqref{t22}),
we learn that 
$L u = P ( \lambda u^3 + h_{\lambda} u + \xi )$.
Finally, by the explicit expression \eqref{t10} of $(f_n)_x.\mathsf{z}_0$,
and the fact that each $\Pi_{\beta}$ appearing therein is periodic and $\kappa$-H\"older (thus bounded) by Definition~\ref{d:smooth}, 
it follows that $\| (f - f_n).\mathsf{z}_0 \| \to_{n\to \infty} 0$, so that $\| f . \mathsf{z}_{\mathbf{0}} \| \leq 1$, and \eqref{eq:P1} is satisfied,
as desired.

\medskip
\textit{Closedness of \eqref{eq:P2}.}
Here we prove that \eqref{eq:P2} holds.
Let $\zeta \in C_{\mathrm{per}, 0}^{\kappa - 2}$.
For $n \in \mathbb{N}$, let $\dot{u}_n \coloneqq \dot{\Phi} ( \lambda_n, u_n )^{-1} \zeta$, so that \eqref{t63} yields
$\sup_{n} [ \dot{u}_n ]_{\kappa}\leq C [ \zeta ]_{\kappa - 2}$.
Thus, the sequence $(\dot{u}_n)_n$ of periodic functions with vanishing space-time average is uniformly bounded and uniformly equicontinuous.
We apply the Arzel\`a--Ascoli theorem:
up to a subsequence, $(\dot{u}_n)_n$ converges
uniformly
to some function $\dot{u} \in C_{\mathrm{per}, 0}^{\kappa}$.
Letting $n \to \infty$ in the sequence of equations satisfied by $\dot{u}_n$, we learn that $\dot{u}$ satisfies the condition in \eqref{t63}.
Observe that there can be at most one such $\dot{u}$, as follows e.g.\ from taking the difference of two possible solutions and applying \eqref{t63} to $\zeta = 0$.
In conclusion, since the choice of $\zeta$ in this paragraph was arbitrary, 
we have proven that $\dot{\Phi} ( \lambda, u )$ is invertible.
By the output of \eqref{t63}, its inverse is continuous.
Thus, \eqref{eq:P2} is satisfied, as desired.
\end{proof}

\appendix

\section{Change of Kernels}

In this section of the appendix, we provide a tool to upgrade bounds for a distribution $g$ phrased in terms of the semigroup $\Psi_t$ 
defined in \eqref{t92}
to bounds for $g(\psi^\mu)$ for an arbitrary Schwartz kernel $\psi$ that depend on only finitely many of the seminorms of $\psi$.
These estimates are used to upgrade the model assumption \eqref{mb01} that is expressed in terms of the semigroup.
We remark that the main idea of the proof of this result (namely the representation formula \eqref{rep}) is contained in \cite[Section 3.4]{BOT} (see also \cite[Lemma A.3]{otto2018parabolic}) with the main addition here being to provide a precise and general statement which also accommodates germs.
We recall that the notations $F_t = F * \Psi_t$ and $F_{\mu} = F * \psi_{\mu}$
are compared in the discussion below \eqref{t92}.
\begin{lemma}\label{kernel_swap}
    Fix $L < \infty$, $\eta\in \mathbb{R}$.
    Suppose that $F$ is a distribution such that 
    \begin{align*}
        \sup_{t \leq L^4} \sup_{x \in \mathbb{R}^{1+d}} (\sqrt[4]{t})^{- \eta} w(\sqrt[4]{t} + |x|)^{-1} |F_t(x)| \le 1.
    \end{align*}
    where $w: (0, \infty) \to (0, \infty)$ is an increasing function satisfying the doubling condition for some constant $C_w$
    $$w(2\mu) \le C_w w(\mu) \qquad \text{ for all }\mu \in \mathbb{R}. $$
    Then there exists a $p^{\prime} \in \mathbb{N}$ depending only on $\eta, C_w$ and $d$ such that
    \begin{align*}
        \sup_{\mu \leq L} \sup_{x \in \mathbb{R}^{1+d}} 
        \sup_{\|\psi\|_{p^{\prime} \leq 1}} 
        \mu^{- \eta} w(\mu + |x|)^{-1} |F_{\mu} ( x )| \lesssim 1  ,
    \end{align*}
    where the implicit constant depends only on $\eta, C_w$ and $d$.
\end{lemma}

\begin{proof}
    We recall from \cite[Section 3.4]{BOT} that we can write
    \begin{align}\label{rep}
        F_{\mu} ( x ) 
        & = \sum_{j = 0}^k \frac{1}{j!} \int dy \, \langle F , \Psi_{\mu^4}(\cdot - y) \rangle \, ( (L^* L)^j \psi)_\mu(y-x) \nonumber
         \\ & \quad + \frac{1}{k!} \int_0^1 ds \int dy  \,  s^k \langle F , \Psi_{s \mu^4}(\cdot - y) \rangle \, ((L^* L)^{k+1} \psi)_\mu(y - x) 
    \end{align}
    where we are free to choose the value of $k \ge 0$.
    We then deduce that
    \begin{align*}
        |F_{\mu} ( x ) | 
        & \le 
        \sum_{j = 0}^k \frac{1}{j!} \mu^{\eta} \int dy \, w(\mu + |y|) |( (L^* L)^j \psi)_\mu(y-x)|  
        \\ & \quad + \frac{1}{k!} \int_0^1 ds (\sqrt[4]{s} \mu)^{\eta} s^k \int dy \, w(\mu + |y|) |( (L^* L)^{k+1} \psi)_\mu(y-x)|.
    \end{align*}

    Choosing $k$ such that $k + \eta/4 > -1$, the $s$-integral can be evaluated as $C_\eta \mu^\eta$. 
    It then only remains to bound integrals of the form 
    \begin{align*}
        \int dy \, w(\mu + |y|) \, |( (L^* L)^j \psi)_\mu(y-x)| .
    \end{align*}
    for $j \le k+1$.
    To do this we note that
    for any $N \in \mathbb{N}$ to be adjusted later,
    \begin{align*}
        \int dy \, w(\mu + |y|) |( (L^* L)^j \psi)_\mu(y-x)| 
        &\le \| \psi \|_{N} \int dy \, \mu^{-D} \frac{w(\mu + |y|)}{(1+ \mu^{-1} |y-x|)^N}
        \\ & \le \| \psi \|_{N} \int dy \, \mu^{N-D} \frac{w(\mu + |y| + |x|)}{(\mu + |y|)^N} .
    \end{align*}
    We consider the regions $|y| \le \mu$, $\mu < |y| \le \mu + |x|$ and $\mu + |x| < |y|$ separately. In the first region, the integral is bounded by
    \begin{align*}
        \int_{|y| \le \mu} dy \mu^{-D} w(2(\mu + |x|)) \lesssim_d C_w w(\mu + |x|).
    \end{align*}
    In the second region, the integral is bounded by
    \begin{align*}
        \int_{\mu < |y| \le \mu + |x|} dy \mu^{N-D} \frac{w(2(\mu + |x|))}{|y|^N} \lesssim_d C_w w(\mu + |x|)
    \end{align*}
    so long as $N > D$. 
    The final region will take some more care. Here the integral is bounded by
    \begin{align*}
        \int_{|y| > \mu + |x|} dy \mu^{N-D} \frac{w(2|y|)}{|y|^N} \le C_w \mu^{N-D} \int_{|y| > \mu + |x|} dy \frac{w(y)}{|y|^N}
    \end{align*}
    so that it remains to see that the doubling condition implies the convergence of this integral for $N$ large enough with a bound of order $w(\mu + |x|)$.
    We make use of the fact that $w$ is increasing to write
    \begin{align}\label{fi01}
        \int_{|y| > \mu + |x|} dy \frac{w(y)}{|y|^N} \le \sum_{n = 0}^\infty \int_{\mathcal{A}_n}dy  \frac{w(2^{n+1} (\mu + |x|))}{|y|^N} 
    \end{align}
    where $\mathcal{A}_n \coloneqq \{y: 2^n (\mu + |x|) < |y| \le 2^{n+1}(\mu + |x|)\}$.
    We then repeatedly apply the doubling condition to bound \eqref{fi01} by
    \begin{align*}
        & \sum_{n = 0}^\infty C_w^{n+1} w(\mu + |x|) \int_{\mathcal{A}_n}dy |y|^{-N} \\
        & \lesssim_d  w(\mu + |x|) (\mu + |x|)^{D - N} \sum_{n = 0}^\infty C_w^{n+1} 2^{n (D-N)}. 
    \end{align*}
    For $N$ large enough (depending on $d$ and $C_w$), the sum converges so that we obtain a bound by
    $\mu^{D-N} w(\mu + |x|)$
    which completes the proof.
\end{proof}

\section{Some Considerations on the Construction of the Periodic Model} \label{s:model}

We now wish to briefly sketch some of the changes that would be required in order to adapt the construction of \cite{BOT} to the periodic case.
In this appendix, 
we follow the notations of \cite{BOT}.
In particular, given a distribution $F$ we denote by $F_t$ the convolution with the semi-group, 
and by $F_r$ the convolution with an arbitrary Schwartz function at scale $r$.
We again freely use the change of kernel result given in Lemma~\ref{kernel_swap}, so that estimates involving these arbitrary Schwartz functions are
in fact uniform over suitable balls of test-functions.

\medskip
{\bf Spectral gap assumption.}
To start with, the choice of Cameron--Martin space in the assumption of spectral gap should respect the periodic boundary conditions.
More precisely, defining $\dot{H}^{-s}$ to be the space of periodic homogeneous Sobolev distributions with vanishing space-time average,
with norm given in Fourier space by 
	\begin{align*}
		\| u \|_{\dot{H}^{-s}}^2 
		\coloneqq \sum\limits_{k \in \frac{2 \pi}{\ell^2} \mathbb{Z} \times ( \frac{2 \pi}{\ell} \mathbb{Z})^d \setminus \lbrace 0 \rbrace} | \hat{u} ( k ) |^2 \, | k |^{-2 s} ,
	\end{align*}
the assumption 
\cite[Section~2.11,~Eq.~(184)]
{BOT} is replaced by the corresponding 
	\begin{align} \label{t100}
		\mathbb{E} ( F - \mathbb{E} F )^2 
		\leq \mathbb{E} \Big\| \frac{\partial F}{\partial \xi} \Big\|_{\dot{H}^{-s}}^2 , 
		\qquad \text{for cylindrical $F$.}
	\end{align}

\medskip
{\bf Base case and Besov norms.}		
When attempting to adapt the inductive loop of estimates in \cite{BOT},
a first difficulty is found in the base case, see \cite[Section~2.12]{BOT} where, due to periodicity of the integrand, the bound \cite[Eq.~(191)]{BOT} fails to hold if one insists on the domain of integration being $\mathbb{R}^{1 + d}$.
Rather, it has to be replaced by 
	\begin{align*}
		 \Big( \int_Q d x \, \mathbb{E}^{\frac{2}{q}} | (\delta \Pi - \mathrm{d} \Gamma_x^* \Pi_x)_{0 r} ( x ) |^q \Big)^{\frac{1}{2}}
		& \lesssim r^{2 + s} ,
	\end{align*}
where $Q \subset \mathbb{R}^{1 + d}$ denotes any (measurable) region of diameter $\leq \ell$.
More generally, for $\beta \neq 0$ \cite[Eq~(192)]{BOT} 
is replaced by 
	\begin{align*}
		 \Big( \int_Q d x \, \mathbb{E}^{\frac{2}{q}} | (\delta \Pi - \mathrm{d} \Gamma_x^* \Pi_x)_{\beta r} ( x ) |^q \Big)^{\frac{1}{2}}
		& \lesssim r^{\alpha + \frac{D}{2}} ( r + R )^{| \beta | - \alpha} ,
	\end{align*}
where again $Q \subset \mathbb{R}^{1 + d}$ is any region of diameter $\leq \ell$, and $R \coloneqq \sup_{z \in Q} | z - x |$ is the distance from $x$ to $Q$.
Similarly, the Besov-type norms appearing in the remaining estimates of the inductive loop, 
\cite[Eqs.~(194), (195\&210), (209\&215), (216)]{BOT}
should be adapted in the appropriate way by replacing the region $B_R$ of integration therein by $Q$ as above.

\medskip
{\bf Tensor periodicity.}
A further difference with \cite{BOT} is
that, 
while all estimates involving a scale parameter $r$ therein are uniform over $r < \infty$, 
in the periodic case some estimates are uniform only on the range $r \leq \ell$.
This is due to the loss of global scaling invariance which arises from working on the torus, and concerns the bounds on the quantities $\Pi^-$ and its Malliavin derivative $\delta \Pi^-$.
Namely, the bound \cite[Eq.~(141)]{BOT} has to be replaced by
	\begin{align} \label{i3}
		\mathbb{E}^{\frac{1}{p}} | \Pi_{x \beta r}^- ( x ) |^p
		\lesssim r^{|\beta|-2} 
		\qquad \text{uniformly over $r \leq \ell$,}
	\end{align}
and analogously for the bounds \cite[Eqs.~(194), (213), (216)]{BOT}.
As it turns out, it is still possible to achieve some form of uniformity over all scales, by exploiting the `tensor-periodic' structure of $\Pi, \Pi^-$ (and their respective Malliavin derivatives).
This structure was already introduced in \cite{LO22} in a quasilinear context, 
and  
takes the form
	\begin{align*}
		\Pi_{x \beta}^{-} (\cdot)
		& = \sum_{|\mathbf{m}| \leq |\beta| \vee 0} ( \cdot - x )^{\mathbf{m}} \, \Pi_{x \beta}^{(\mathbf{m})-} (\cdot) ,
	\end{align*}
where each $\Pi_{x \beta}^{(\mathbf{m})-}$ is a periodic distribution.
The same decomposition holds for $\Pi_{x \beta}^{}$, where the polynomial degree $|\beta| \vee 0$ is readily checked to be stable under the inductive construction of $\Pi, \Pi^-$.
This structure may be understood as a generalisation of periodicity 
(where the polynomials arise from inductively solving the hierarchy of equations on $\Pi$) and as such comes with a notion of space-time average, which is the polynomial defined by
	\begin{align*}
		\Big( \fint \Pi_{x \beta}^{-} \Big) ( \cdot )
		& \coloneqq \sum_{|\mathbf{m}| \leq |\beta| \vee 0} ( \cdot - x )^{\mathbf{m}} \, \fint \Pi_{x \beta}^{(\mathbf{m})-} .
	\end{align*}

We may now state a lemma which allows us to describe the large-scale behaviour of distributions possessing this `tensor-periodic' structure,
and which generalises the more classical property that large-scale averages of a periodic distribution converge to its average.
\begin{lemma}[From Small To Large Scales] \label{l:lspp}
Let $x \in \mathbb{R}^{1 + d}$, $m_0 \in \mathbb{N}_0$, $\alpha \in \mathbb{R}$.
Let $u$ be a distribution of the form $u (\cdot) = \sum_{|\mathbf{m}| \leq m_0} ( \cdot - x )^{\mathbf{m}} \, u^{(\mathbf{m})} (\cdot)$ where each $u^{(\mathbf{m})}$ is periodic.
For any $p \in \mathbb{N}$ there exists $p^{\prime} \in \mathbb{N}$,
depending only on $m_0, \alpha, d$ and $p$,
such that 
	\begin{align*}
	\sup_{\| \psi \|_{p^{\prime}} \leq 1} \sup_{r \geq \ell} \frac{| \big( u - \fint u \big)_r ( x) |}{r^{\alpha}} 
	& \lesssim
		\sup_{\| \psi \|_p \leq 1} \sup_{r \leq \ell} \frac{| u_r ( x) |}{r^{\alpha}} ,
	\end{align*}
	where the implicit constant depends only on $m_0, \alpha, d$ and $p$.
\end{lemma}
Arguing e.g.\ by duality in the probabilistic $L^p$ spaces as in \cite[Section~3.5]{BOT}, the same statement holds when $u$ is random and the absolute value is replaced by $\mathbb{E}^{1/p} |\cdot |^p$, for any $p \in (1, \infty)$.
Thus, next to \eqref{i3} one also adds the estimate
	\begin{align*}
		\mathbb{E}^{\frac{1}{p}} \Big| \Big(\Pi_{x \beta}^- - \fint \Pi_{x \beta}^{-} \Big)_r ( x ) \Big|^p 
		\lesssim r^{|\beta|-2} 
		\qquad \text{uniformly over $r \geq \ell$,}
	\end{align*}
along with the analogous bounds for \cite[Eqs.~(194), (213), (216)]{BOT}.

\medskip 
{\bf Non-uniqueness of counterterms.}
Arguing as in \cite[Section~3.3]{BOT}, and making use of Lemma~\ref{l:lspp}, one obtains for the expectation of $\Pi^-$ that for all $\beta$ with $| \beta | < 2$,
	\begin{align*}
		t \Big| \frac{d}{dt} \mathbb{E} \Pi_{\beta t}^- ( 0 ) \Big| \lesssim (\sqrt[4]{t})^{| \beta | - 2} , 
		\qquad
		\text{uniformly over $t < \infty$.}
	\end{align*}
Thus, just as in \cite{BOT}, there is a unique choice of counter-term $c$ which guarantees the following estimate at all scales:
	\begin{align*}
		\big| \mathbb{E} \Pi_{\beta t}^- ( 0 ) \big| 
		\lesssim (\sqrt[4]{t})^{| \beta | - 2} , 
		\qquad
		\text{uniformly over $t < \infty$,}
	\end{align*}
and this choice corresponds to postulating $\mathbb{E} \Pi_{\beta t}^- ( 0 ) \to_{t \to \infty} 0$, 
or equivalently $\fint \mathbb{E} \Pi_{\beta}^{(\mathbf{0})-} = 0$ in view of Lemma~\ref{l:lspp}.
However, since the relevant model estimate \eqref{i3} in the periodic setting necessitates uniformity only over $r \leq \ell$, 
it follows that any bounded shift to this counter-term also produces a model which satisfies the desired stochastic estimates.
Thus, this construction does not give rise to a unique model, but rather to a finite-dimensional family of renormalised models.

\medskip
{\bf Kolmogorov argument.}
Here we very briefly sketch how to pass from annealed model estimates to the quenched estimates taken as input in Definition~\ref{d:model}.
This is reminiscent of the Kolmogorov continuity theorem in stochastic analysis, we refer to \cite{Hai14,HS23,Tem24} for similar arguments in the context of model estimates in regularity structures.
More precisely, in this paragraph we take as input the (annealed) estimates over $r \leq \ell$
and $x, y \in \mathbb{R}^{1 + d}$,
	\begin{align} \label{i6}
		\mathbb{E}^{\frac{1}{p}} | \Pi_{x \beta r} ( x ) |^p \lesssim r^{| \beta |} ,
		\qquad \mathbb{E}^{\frac{1}{p}} | (\Gamma^*_{x y})_{\beta}^{\gamma} |^p \lesssim | y - x |^{| \beta | - | \gamma |} ,
	\end{align}
and derive the (quenched) estimate
	\begin{align} \label{i5}
		 | \Pi_{x \beta r} ( x ) | \lesssim r^{| \beta |_-} ,
		 \qquad 
		 \text{over $r \leq \ell$, $x \in \mathbb{R}^{1 + d}$, a.s.}
	\end{align}
where by $| \beta |_{-}$ we mean the homogeneity of $\beta$, recall \eqref{t45}, but where $\alpha$ is replaced by $\alpha - \epsilon$ for some fixed arbitrarily small $\epsilon>0$ where the implicit constant in inequalities such as \eqref{i5} now depends on $\epsilon$.
One then obtains the version of \eqref{i5} over all $r < \infty$ by appealing to Lemma~\ref{l:lspp}.
Indeed, when $| \beta | \geq 0$ the desired scaling behaviour on large scales follows from the fact that $\fint \Pi_{x \beta}$ is a polynomial of degree $| \beta |$, 
while when $| \beta | < 0$ our construction implies that $\Pi_{x \beta}$ is actually a periodic function of vanishing space-time average $\fint \Pi_{x \beta} = 0$.
By the same arguments, one obtains the desired pathwise estimate for $\Pi^-$, 
where we note that Lemma~\ref{l:lspp} does \emph{not} imply a control of order $r^{| \beta | - 2}$ over scales $r \geq \ell$, since the polynomial $\fint \Pi_{x \beta}^-$ is of degree $| \beta | \vee 0$ which is strictly larger than the desired $| \beta | - 2$.
Hence, the model estimate on $\Pi$ is uniform over all scales while that on $\Pi^-$ is only uniform over $\mu < \ell$,
thus justifying the model norm \eqref{mb01}.
Finally, the quenched estimates for $\Gamma$ follow from quenched estimates for $\Pi$ by applying the algebraic argument and the three-point argument of \cite[Section~2.19]{BOT} in a pathwise manner.

\medskip
We now turn to \eqref{i5}.
We note on the one hand that \eqref{i6} in combination with
the reexpansion property \eqref{mb04} implies when $| y - x | \le r$ that
	\begin{align*}
		\mathbb{E}^{\frac{1}{p}} |  \Pi_{y \beta r} ( x ) - \Pi_{x \beta r} ( x ) |^p
		\lesssim | y - x |^{\eta}  r^{| \beta | - \eta} , 
	\end{align*}
for any $\eta \in (0, \kappa^-]$ where $\kappa^- \in (0, 1)$ is as in \eqref{i7}.
Furthermore Taylor's theorem and \eqref{i6} yield
	\begin{align*}
		\mathbb{E}^{\frac{1}{p}} |  \Pi_{y \beta r} ( x ) - \Pi_{y \beta r} ( y ) |^p
		\lesssim | y - x | r ^{| \beta | -1} , 
	\end{align*}
	
	so that 
	\begin{align*}
		\mathbb{E}^{\frac{1}{p}} |  \Pi_{y \beta r} ( y ) - \Pi_{x \beta r} ( x ) |^p
		\lesssim | y - x |^{\eta} r^{| \beta | - \eta } .
	\end{align*}
Meanwhile, when $|y-x| \ge r$, we have the straightforward consequence of \eqref{i6}
	\begin{align*}
		\mathbb{E}^{\frac{1}{p}} |  \Pi_{y \beta r} ( y ) - \Pi_{x \beta r} ( x ) |^p \lesssim r^{| \beta |}.
	\end{align*}
Taken together the above estimates yield for $ r \lesssim \ell$
	\begin{align*}
		\mathbb{E}^{\frac{1}{p}} | \Pi_{y \beta r} ( y ) - \Pi_{x \beta r} ( x ) |^p
		\lesssim | y - x |^{\eta} \, r^{| \beta | - \eta} . 
	\end{align*}
Turning to the Kolmogorov argument, 
we recall the Sobolev–Slobodeckij inequality, valid for 
any $p \geq 1$, any continuous periodic function $f$, and any exponent $\eta > 2d/p$:
	\begin{align*}
		\sup_{x \in \mathbb{R}^{1 + d}} | f ( x ) |^p
		& \lesssim_{\ell, \eta, d, p} | f ( 0 ) |^p
			+ \iint_{B_{\ell} \times B_{\ell}} dy \, d y^{\prime} \, \frac{| f ( y ) - f (y^{\prime} ) |^p}{| y - y^{\prime} |^{\eta	p}} .
	\end{align*}
Choosing $f \colon x \mapsto \Pi_{x \beta r} ( x )$, 
it follows that
	\begin{align*}
		\mathbb{E}^{\frac{1}{p}} \sup\limits_{x \in \mathbb{R}^{1 + d}} | \Pi_{x \beta r} ( x ) |^p
		& \lesssim r^{| \beta | - \eta} ,	
		\qquad \text{over $r \leq \ell$,}	
	\end{align*}
provided $p$ is sufficiently large so that $\eta - 2 d / p > 0$.
In order to tackle the supremum over $r$, we recall the following weighted version of Sobolev's inequality, valid for all $g \in C^1 ( (0, \ell] )$, $p \in (1, \infty)$, $a \in \mathbb{R}$:
	\begin{align*}
		\sup_{r \leq \ell} \big| r^{- (a - \frac{1}{p})} g ( r ) \big|^p
		\lesssim_{\ell, a, p}
		\int_{0}^\ell d s \, \big( | s^{-a} g ( s ) |^p + | s^{-a+1} g^{\prime} ( s ) |^p \big) ,
	\end{align*}
which we apply to $g \colon r \mapsto \Pi_{x \beta r} ( x )$ and $a \coloneqq | \beta |- \eta$
to the effect of
	\begin{align*}
		\mathbb{E}^{\frac{1}{p}} \sup_{ x \in \mathbb{R}^{1+d}} \sup_{r \leq \ell} \, \big| r^{- (| \beta | - \eta - \frac{1}{p})} \Pi_{x \beta r} ( x ) \big|^p
		< \infty .
	\end{align*}
This implies the desired \eqref{i5} provided one takes $\eta$ sufficiently small and $p$ sufficiently large.

\subsection*{Acknowledgements}
FO acknowledges the hospitality of
the Bernoulli Center Program
``New developments and challenges in Stochastic Partial Differential Equations'',
where he presented parts of this material in the summer of 2024.

\bibliographystyle{alpha}
\bibliography{bib.bib}
\end{document}